\documentclass{amsart}

\usepackage[draft=false,pagebackref]{hyperref}
\usepackage{amssymb,amsbsy,mathtools}

\newtheorem{thm}[equation]{Theorem}
\newtheorem{lem}[equation]{Lemma}

\theoremstyle{definition}
\newtheorem{defn}[equation]{Definition}
\newtheorem{rmk}[equation]{Remark}

\numberwithin{equation}{section}

\newcommand\abs[2][empty]{\csname#1\endcsname \lvert{#2}\csname#1\endcsname\rvert}
\newcommand\doublebar[2][empty]{\csname#1\endcsname \lVert{#2}\csname#1\endcsname\rVert}

\newcommand\mat[1]{\boldsymbol{#1}}
\newcommand\arr[1]{\boldsymbol{\dot{#1}}}

\newcommand\Xint[1]{\mathchoice
{\XXint\displaystyle\textstyle{#1}}%
{\XXint\textstyle\scriptstyle{#1}}%
{\XXint\scriptstyle\scriptscriptstyle{#1}}%
{\XXint\scriptscriptstyle\scriptscriptstyle{#1}}%
\!\int}
\newcommand\XXint[3]{{\setbox0=\hbox{$#1{#2#3}{\int}$ }
\vcenter{\hbox{$#2#3$ }}\kern-.6\wd0}}
\newcommand\fint{\Xint-}

\newcommand\dist{\mathop{\mathrm{dist}}\nolimits}
\newcommand\Div{\mathop{\mathrm{div}}\nolimits}
\newcommand\Tr{\mathop{\smash{\arr{\mathrm{Tr}}}\vphantom{T}}\nolimits}
\newcommand\Trace{\mathop{\mathrm{Tr}}\nolimits}
\newcommand\M{\mathop{\smash{\arr{\mathrm{M}}}\vphantom{M}}\nolimits}

\newcommand\supp{\mathop{\mathrm{supp}}\nolimits}
\newcommand\diam{\mathop{\mathrm{diam}}\nolimits}

\newcommand\re{\mathop{\mathrm{Re}}\nolimits}

 \let\R\RR

\newcommand\N{\mathbb{N}} 
\newcommand\1{\mathbf{1}}
\newcommand\D{\mathcal{D}}
\newcommand\s{\mathcal{S}}

\newcommand\pureH{\parallel}

\newcommand\dmn{{n+1}}
\newcommand\pdmn{{(n+1)}}
\newcommand\dmnMinusOne{n}

\begin{document}

\title[Rough bounds on layer potentials]%
{Bounds on layer potentials with rough inputs for higher order elliptic equations}

\author{Ariel Barton}
\address{Ariel Barton, Department of Mathematical Sciences,
			309 SCEN,
			University of Ar\-kan\-sas,
			Fayetteville, AR 72701}
\email{aeb019@uark.edu}

\author{Steve Hofmann}
\address{Steve Hofmann, 202 Math Sciences Bldg., University of Missouri, Columbia, MO 65211}
\email{hofmanns@missouri.edu}
\thanks{Steve Hofmann is partially supported by the NSF grant DMS-1664047.}

\author{Svitlana Mayboroda}
\address{Svitlana Mayboroda, Department of Mathematics, University of Minnesota, Minneapolis, Minnesota 55455}
\email{svitlana@math.umn.edu}
\thanks{Svitlana Mayboroda is partially supported by the NSF CAREER Awards DMS 1056004 and 1220089, the NSF INSPIRE Award DMS 1344235, the NSF Materials Research Science and Engineering Center Seed Grant,  the NSF RAISE-TAQ grant DMS 1839077, the Simons Fellowship, and the Simons Foundation grant 563916, SM}

\begin{abstract}
In this paper we establish square-function estimates on the double and single layer potentials with rough inputs for divergence form elliptic operators, of arbitrary even order $2m$, with variable $t$-independent coefficients in the upper half-space.
\end{abstract}



\maketitle

\tableofcontents

\section{Introduction}

This paper is part of an ongoing study of elliptic differential operators of the form
\begin{equation}\label{A:eqn:divergence}
Lu = (-1)^m \sum_{\abs{\alpha}=\abs{\beta}=m} \partial^\alpha (A_{\alpha\beta} \partial^\beta u)\end{equation}
for $m\geq 2$, with general bounded measurable coefficients.

Specifically, we hope to study boundary value problems such as the Dirichlet problem
\begin{equation}\label{A:eqn:introduction:Dirichlet}
Lu =0 \text{ in }\Omega,
\qquad
\nabla^{m-1} u= \arr f\text{ on }\partial\Omega\end{equation}
with boundary data $\arr f\in L^2(\partial\Omega)$ or in the boundary Sobolev space $\dot W^2_1(\partial\Omega)$. We are also interested in the higher order Neumann problem. We remark that in the second order case ($m=1$), there are many known results concerning these boundary value problems, while in the higher order case, there are very few known results for variable coefficients.

\subsection{Layer potentials in the second order case}

Classic tools for solving second order boundary value problems are the double and single layer potentials, given by
\begin{align}
\label{A:eqn:introduction:D}
\D^{\mat A}_\Omega f(X) &= \int_{\partial\Omega} \overline{\nu\cdot \mat A^*(Y)\nabla_{Y} E^{L^*}(Y,X)} \, f(Y)\,d\sigma(Y)
,\\
\label{A:eqn:introduction:S}
\s^L_\Omega g(X) &= \int_{\partial\Omega}E^{L}(X,Y) \, g(Y)\,d\sigma(Y)
\end{align}
where $\nu$ is the unit outward normal to~$\Omega$ and where $E^L(X,Y)$ is the fundamental solution for the operator~$L$.
Layer potentials are suggested by the Green's formula: if $Lu=0$ in $\Omega$ and $u\in \dot W^2_1(\Omega)$, for some second-order operator~$L$, then
\begin{equation}
\label{A:eqn:green:introduction}
u(X)=-\D^{\mat A}(u\big\vert_{\partial\Omega})+\s^L(\nu\cdot \mat A\nabla u)\quad\text{for all $X\in\Omega$}
\end{equation}
where $u\big\vert_{\partial\Omega}$ and $\nu\cdot \mat A\nabla u$ are the Dirichlet and Neumann boundary values of~$u$.

It is possible, though somewhat involved, to generalize formulas~\eqref{A:eqn:introduction:D} and~\eqref{A:eqn:introduction:S} to the higher order case, and multiple subtly different generalizations exist. We will use the potentials introduced in \cite{BarHM17}; these potentials are similar to but subtly different from those used in \cite{CohG83,Ver05,MitM13B} to study the biharmonic operator $\Delta^2$ and in \cite{Agm57,MitM13A} to study more general constant coefficient operators.

There are many ways to use layer potentials to study boundary value problems; see \cite{FabJR78,Ver84,DahK87,FabMM98,Zan00} in the case of harmonic functions (that is, the case $\mat A=\mat I$ and $L=-\Delta$) and \cite{KenR09,Rul07,Agr09,AlfAAHK11,Bar13,BarM16A,Ros13,AusM14,HofKMP15B,HofMayMou15,HofMitMor15} in the case of more general second order problems.
In particular, the second-order double and single layer potentials have been used to study higher-order differential equations in \cite{PipV92,BarM13}. In many cases an important first step is to establish boundedness of layer potentials in suitable function spaces; indeed the extensive use of harmonic layer potentials in Lipschitz domains began with the boundedness of the Cauchy integral on a Lipschitz curve \cite{CoiMM82}, which implies boundedness of layer potentials with $L^2$ inputs, and much recent work has begun with estimates on layer potentials with more general coefficents.

\subsection{New bounds on layer potentials: the main results of this paper}

The main results of this paper are the following bounds on layer potentials. The novelty of this result is the case $m\geq 2$; the case $m=1$ was established in \cite[Theorem~12.7]{AusS16}.
\begin{thm}\label{A:thm:square:rough}
Suppose that $L$ is an operator of the form~\eqref{A:eqn:divergence}, of order~$2m$, $m\geq 1$, acting on functions defined in $\R^\dmn$, $\dmnMinusOne\geq 1$, associated with coefficients $\mat A$ that satisfy the condition
\begin{equation}\label{A:eqn:t-independent}\mat A(x,t)=\mat A(x,s)=\mat A(x) \quad\text{for all $x\in\R^n$ and all $s$, $t\in\R$}\end{equation}
and the ellipticity conditions \eqref{A:eqn:elliptic} and~\eqref{A:eqn:elliptic:bounded}.

Then the double layer potential $\D^{\mat A}$ in the half-space, as defined by formula~\eqref{A:dfn:D:newton}, satisfies the bound
\begin{align}
\label{A:eqn:D:square:rough:intro}
\int_{\R^n}\int_{-\infty}^\infty \abs{\nabla^m \D^{\mat A} \arr f(x,t)}^2\,\abs{t}\,dt\,dx
	& \leq C \doublebar{\arr f}_{L^2(\R^n)}^2
\end{align}
for all $\arr f = \nabla^{m-1}\varphi\big\vert_{\partial\R^\dmn_+}$ for some $\varphi\in C^\infty_0(\R^\dmn)$.

Furthermore, the single layer potential $\s^L$ defined by formula~\eqref{A:dfn:S:newton} extends to an operator that satisfies the bound
\begin{align}
\label{A:eqn:S:square:rough:intro}
\int_{\R^n}\int_{-\infty}^\infty \abs{\nabla^m \s^L \arr g(x,t)}^2\,\abs{t}\,dt\,dx
	&\leq C \doublebar{\arr g}_{\dot W^2_{-1}(\R^n)}^2
\end{align}
for all $\arr g$ in the negative smoothness space $\dot W^2_{-1}(\R^n)$ (that is, for any array $\arr g$ of bounded operators on $\dot W^2_1(\R^n)$).

Here $C$ depends only on the dimension $\dmn$ and the ellipticity constants $\lambda$ and $\Lambda$ in the bounds \eqref{A:eqn:elliptic} and~\eqref{A:eqn:elliptic:bounded}.
\end{thm}

The bound \eqref{A:eqn:D:square:rough:intro} will be established in Section~\ref{A:sec:D:square:variant}, and the bound
\eqref{A:eqn:S:square:rough:intro} (or, rather, the equivalent bound~\eqref{A:eqn:S:square:variant:intro}) will be established in Section~\ref{A:sec:S:square:variant}.

We conjecture that this theorem may be generalized from the half-space to Lipschitz graph domains, but in contrast to the case of second order operators, the method of proof of \cite{BarHM17} requires the extra structure of~$\R^\dmn_+$.

Even in the case of second-order equations, some regularity assumption must be imposed on the coefficients $\mat A$ in order to ensure well-posedness of boundary-value problems. See the classic example of Caffarelli, Fabes, and Kenig \cite{CafFK81}, in which real, symmetric, bounded, continuous, elliptic coefficients $\mat A$ are constructed for which the Dirichlet problem with $L^p$ boundary data is not well-posed in the unit disk for any $1<p<\infty$.
A common starting regularity condition is that the coefficients $\mat A$ be $t$-independent in the sense of satisfying formula~\eqref{A:eqn:t-independent}.
Boundary value problems for such coefficients have been investigated extensively.
See, for example,
\cite{JerK81A,FabJK84,KenP93,KenKPT00,Rul07,AlfAAHK11,AusAH08,AusAM10A,Bar13,AusM14,HofKMP15A,HofKMP15B,HofMitMor15,BarM16A}.

The only result known to be valid for operators with variable $t$-independent coefficients of arbitrary order is layer potential estimates for a higher order of regularity, that is, when the data $\arr f$ lie in $\dot W^2_1(\R^n)$ and $\arr g$ lie in~$L^2(\R^n)$. These results were proven by the authors of the present paper in {\cite{BarHM17}}. Specifically, under the same conditions as in Theorem~\ref{A:thm:square:rough}, we have the estimates
\begin{align}
\label{A:eqn:D:square}
\int_{\R^n}\int_{-\infty}^\infty \abs{\nabla^m \partial_t \D^{\mat A} \arr f(x,t)}^2\,\abs{t}\,dt\,dx
	& \leq C \doublebar{\arr f}_{\dot W^2_1(\R^\dmnMinusOne)}^2
	= C \doublebar{\nabla_\pureH\arr f}_{L^2(\R^n)}^2
,\\
\label{A:eqn:S:square}
\int_{\R^n}\int_{-\infty}^\infty \abs{\nabla^m \partial_t\s^L \arr g(x,t)}^2\,\abs{t}\,dt\,dx
	& \leq C \doublebar{\arr g}_{L^2(\R^n)}^2
\end{align}
for all $\arr g\in {L^2(\R^n)}$ and where $\arr f$ is as in Theorem~\ref{A:thm:square:rough}.

The approach of the present paper is somewhat different from that of \cite{BarHM17}, as we may exploit the bounds \eqref{A:eqn:D:square} and \eqref{A:eqn:S:square} established therein. The arguments of both papers, however, rely on $T1$ or $Tb$ type theorems.

The space $\dot W^2_{-1}(\R^\dmnMinusOne)$ is difficult to work with, and so it is often convenient to define an auxiliary operator whose boundedness on $L^2(\R^\dmnMinusOne)$ implies boundedness of $\s^L$ on $\dot W^2_{-1}(\R^\dmnMinusOne)$. In the second order case, this auxiliary operator is the modified single layer potential $\s^L\nabla$ used in \cite{AlfAAHK11,HofMitMor15,HofMayMou15}, and is given by
\begin{align*}
(\s^L\nabla)\cdot\vec h(x,t) &= \int_{\R^n}\nabla_{y,s} E^{L}(x,t,y,s)\big\vert_{s=0} \cdot \vec h(y)\,dy
.\end{align*}
We will define the higher order modified single layer potential $\s^L_\nabla$ in Section~\ref{A:sec:dfn:potentials}.
For $t$-independent coefficients the bounds~\eqref{A:eqn:S:square:rough:intro} and \eqref{A:eqn:S:square} are equivalent to the bound
\begin{align}
\label{A:eqn:S:square:variant:intro}
\int_{\R^n}\int_{-\infty}^\infty \abs{\nabla^m \s^L_\nabla \arr h(x,t)}^2\,\abs{t}\,dt\,dx
	&\leq C \doublebar{\arr h}_{L^2(\R^n)}^2
.\end{align}
Observe that the kernel of $\s^L\nabla$ involves a gradient of the fundamental solution and so has one fewer degree of smoothness than the kernel of~$\s^L$ in formula~\eqref{A:eqn:introduction:S}. (The same is true of the higher order operators $\s^L$, $\s^L_\nabla$ given in Section~\ref{A:sec:dfn:potentials}.) This additional degree of smoothness is exploited in \cite{BarHM17} to establish the quasi-orthogonality estimate required by the $Tb$ theorem of \cite{GraH16}, and so the arguments of \cite{BarHM17} cannot be applied to bound~$\s^L_\nabla$; we must exploit other $T1$ type theorems.

\subsection{\texorpdfstring{$L^p$}{Lp} and Carleson bounds on the single layer potential}

We will establish some further estimates on the single layer potential $\s^L$. We will also establish some bounds on the modified single layer potential~$\s^L_\nabla$. We will apply these bounds in the paper \cite{BarHM17pB}.

Specifically, we have the following theorem. In this theorem $\mathcal{A}_2^\pm$ denotes the Lusin area integral. See formula~\eqref{A:eqn:lusin} below for a precise definition; here we will merely observe that
\begin{equation*}\doublebar{\mathcal{A}_2^\pm H}_{L^2(\R^n)}^2 =c_n \int_{\R^\dmn_\pm} \abs{H(x,t)}^2\frac{1}{\abs{t}}\,dt\,dx\end{equation*}
and so the bounds \eqref{A:eqn:S:square} and~\eqref{A:eqn:S:square:variant:intro} may be written as
\begin{align*}\doublebar{\mathcal{A}_2^\pm (\abs{t}\,\nabla^m \partial_t \s^L \arr g)}_{L^2(\R^n)}
&\leq C \doublebar{\arr g}_{L^2(\R^n)}
,\\
\doublebar{\mathcal{A}_2^\pm (\abs{t}\,\nabla^m \s^L_\nabla \arr h)}_{L^2(\R^n)}
&\leq C \doublebar{\arr h}_{L^2(\R^n)}
.\end{align*}

\begin{thm}\label{A:thm:Carleson:intro}
Suppose that $L$ is an operator of the form~\eqref{A:eqn:divergence}, of order~$2m$, $m\geq 1$, acting on functions defined in $\R^\dmn$, $\dmnMinusOne\geq 1$, associated with coefficients $\mat A$ that are $t$-independent in the sense of formula~\eqref{A:eqn:t-independent} and satisfy the ellipticity conditions \eqref{A:eqn:elliptic} and~\eqref{A:eqn:elliptic:bounded}.
Let $\s^L$ and $\s^L_\nabla$ be given by formulas~\eqref{A:dfn:S:newton} and~\eqref{A:eqn:S:S:vertical}--\eqref{A:eqn:S:S:horizontal}.

There is some $\varepsilon>0$, depending only on $m$, the dimension $\dmn$ and the parameters $\lambda$ and $\Lambda$ in formulas~\eqref{A:eqn:elliptic} and~\eqref{A:eqn:elliptic:bounded}, such that if $2-\varepsilon<p\leq 2$, then there is a number $C(p)$ such that
\begin{align}
\label{A:eqn:S:p-}
\doublebar{\mathcal{A}_2^\pm (\abs{t}\,\nabla^m \partial_t \s^L \arr g)}_{L^p(\R^n)}
&\leq C(p) \doublebar{\arr g}_{L^p(\R^n)}, && 2-\varepsilon<p\leq 2
,\\
\label{A:eqn:S:p-:variant}
\doublebar{\mathcal{A}_2^\pm (\abs{t}\,\nabla^m \s^L_\nabla \arr h)}_{L^p(\R^n)}
&\leq C(p) \doublebar{\arr h}_{L^p(\R^n)}
, && 2-\varepsilon<p\leq 2
.\end{align}

If $\dmn=2$ or $\dmn=3$ then we have the estimates
\begin{align}
\label{A:eqn:lusin:atomic}
\doublebar{\mathcal{A}_2^\pm (\abs{t}\,\nabla^{m-1} \partial_t^2 \s^L \arr g)}_{L^p(\R^n)}
&\leq C(p) \doublebar{\arr g}_{L^p(\R^n)}, && 1<p\leq 2,
\\
\label{A:eqn:lusin:atomic:k}
\doublebar{\mathcal{A}_2^\pm (\abs{t}^k\,\nabla^{m} \partial_t^k \s^L \arr g)}_{L^p(\R^n)}
&\leq C(k,p) \doublebar{\arr g}_{L^p(\R^n)}
,&& 1<p\leq 2,\>k\geq 2.
\end{align}

If $k$ is large enough (depending on $m$ and~$n$), then the following statements are true.

First, we have the Carleson measure estimates
\begin{align}
\label{A:eqn:S:Carleson:intro}
\sup_{Q\subset\R^n}
\frac{1}{\abs{Q}}\int_Q\int_{0}^{\ell(Q)} \abs{t^{k} \nabla^m\partial_t^{k}\s^L \arr b(x,t)}^2\,\frac{1}{t}\,dt\,dx
	& \leq C(k) \doublebar{\arr b}_{L^\infty(\R^n)}^2
,\\
\label{A:eqn:S:variant:Carleson}
\sup_{Q\subset\R^n}
\frac{1}{\abs{Q}}\int_Q\int_{0}^{\ell(Q)} \abs{t^k \nabla^m\partial_t^k\s^L_\nabla \arr b(x,t)}^2\,{t}\,dt\,dx
	& \leq C(k) \doublebar{\arr b}_{L^\infty(\R^n)}^2
\end{align}
where the supremum is over all cubes $Q\subset\R^n$, and where $\ell(Q)$ denotes the side length of~$Q$.
We also have the corresponding estimates in the lower half-space.

We have the area integral estimates
\begin{align}
\label{A:eqn:lusin:p:intro}
\doublebar{\mathcal{A}_2^\pm (\abs{t}^k\,\nabla^m \partial_t^{k} \s^L \arr g)}_{L^p(\R^n)}
&\leq C(k,p) \doublebar{\arr g}_{L^p(\R^n)}
, && 2\leq p<\infty,\\
\label{A:eqn:lusin:p:variant:intro}
\doublebar{\mathcal{A}_2^\pm (\abs{t}^{k+1}\,\nabla^m \partial_t^{k}\s^L_\nabla \arr h)}_{L^p(\R^n)}
&\leq C(k,p) \doublebar{\arr h}_{L^p(\R^n)}
, && 2\leq p<\infty.\end{align}

Let $\eta$ be a Schwartz function defined on $\R^n$ with $\int\eta=1$.
Let $\mathcal{Q}_t$ denote convolution with $\eta_t=t^{-n}\eta(\,\cdot\,/t)$.
Let $\arr b$ be any array of bounded functions. Then for any $p$ with $1<p<\infty$, we have that
\begin{align}
\label{A:eqn:S:Schwartz:p}
\doublebar{\mathcal{A}_2^\pm(\abs{t}^{k+1}\partial_\perp^{k+m}\s^L_\nabla (\arr b \mathcal{Q}_{\abs{t}} h))}_{L^p(\R^n)}
&\leq C(p,k,\eta)\doublebar{\arr b}_{L^\infty(\R^n)}\doublebar{h}_{L^p(\R^n)}, && 1<p<\infty
\end{align}
where the constant $C(p,k,\eta)$ depends only on $p$, $k$, the Schwartz constants of~$\eta$, and on the standard parameters $n$, $m$, $\lambda$, and~$\Lambda$.

\end{thm}

The bounds \eqref{A:eqn:lusin:atomic} and~\eqref{A:eqn:lusin:atomic:k} are also valid in higher dimensions if $L$ satisfies a De Giorgi-Nash-Moser condition; see Lemma~\ref{A:lem:atomic}.

The estimates \eqref{A:eqn:S:Carleson:intro} and~\eqref{A:eqn:lusin:p:intro} will be established in Section~\ref{A:sec:S:vertical}, and will be needed to prove Theorem~\ref{A:thm:square:rough}; for completeness, we will establish the similar estimates \eqref{A:eqn:S:variant:Carleson} and~\eqref{A:eqn:lusin:p:variant:intro} in Section~\ref{A:sec:S:vertical:variant}. The estimate \eqref{A:eqn:S:Schwartz:p} will be proven in Lemma~\ref{A:lem:S:Schwartz:p}.
We will prove the bounds \eqref{A:eqn:lusin:atomic} and~\eqref{A:eqn:lusin:atomic:k} in Section~\ref{A:sec:atomic}, and will prove the bounds \eqref{A:eqn:S:p-} and~\eqref{A:eqn:S:p-:variant} in Section~\ref{A:sec:p-}.

As mentioned above, the $q=2$ case of the bounds \eqref{A:eqn:S:p-} and~\eqref{A:eqn:S:p-:variant} are simply the bounds \eqref{A:eqn:S:square} and~\eqref{A:eqn:D:square:rough:intro}; the $q=2$ case of the bounds \eqref{A:eqn:lusin:atomic:k}, \eqref{A:eqn:lusin:p:intro} and~\eqref{A:eqn:lusin:p:variant:intro} follow from the Caccioppoli inequality applied in Whitney cubes, and so the novelty lies in the cases $2-\varepsilon<q<2$, $1<q<2$ or $2<q<\infty$.

\subsection{Boundary value problems and future work}
It is our intention to use the classic method of layer potentials to establish well-posedness of boundary value problems.

If the coefficents $\mat A$ of the operator $L$ given by formula~\eqref{A:eqn:divergence} are real and constant, and if $\Omega$ is a bounded Lipschitz domain, then by \cite{PipV95B} and \cite{DahKPV97} the boundary value problems
\begin{align*}
Lu&=0\text{ in }\Omega,& \nabla^{m-1}u\big\vert_{\partial\Omega}&=\arr f, & \int_\Omega \abs{\nabla^m u(X)}^2 \dist(X,\partial\Omega)\,dX&\leq C \doublebar{\arr f}_{L^2(\partial\Omega)}^2
,\\
Lu&=0\text{ in }\Omega,& \nabla^{m-1}u\big\vert_{\partial\Omega}&=\arr f, & \int_\Omega \abs{\nabla^{m+1} u(X)}^2 \dist(X,\partial\Omega)\,dX&\leq C \doublebar{\arr f}_{\dot W_1^2(\partial\Omega)}^2
\end{align*}
are well-posed; that is, there is at most one solution to each equation, and if $\arr f=\nabla^{m-1}\varphi\vert_{\partial\Omega}$ for \emph{some} function~$\varphi$, then a solution exists, that is, $\arr f=\nabla^{m-1}u\vert_{\partial\Omega}$ for some $u$ with $Lu=0$ in~$\Omega$ that satisfies appropriate integral estimates. If $L=\Delta^2$ is the biharmonic operator, then by \cite{Ver05} and \cite{PipV91} the Neumann problem
\begin{equation}
\label{A:eqn:Neumann:biharmonic}
\Delta^2 u=0\text{ in }\Omega, \quad \M_{\mat A}^\Omega u=\arr g, \quad \int_\Omega \abs{\nabla^3 u(X)}^2 \dist(X,\partial\Omega)\,dX\leq C \doublebar{\arr g}_{L^2(\partial\Omega)}^2
\end{equation}
is well posed, where $\M_{\mat A}^\Omega u$ denotes the Neumann boundary values of~$u$ associated to the coefficients $\mat A$ of the operator~$\Delta^2$. Formulation of the Neumann boundary values of solutions to elliptic equations is a difficult problem, tightly intertwined with the formulation of layer potentials; we refer the interested reader to \cite{Ver10,BarM16B} for a discussion of related issues and to \cite{BarHM17} for the formulation of Neumann boundary values used in \cite{BarHM17pB,BarHM18}.

Notice that if $\Omega=\R^\dmn_+$ is the half-space and we identify $\R^n$ with $\partial\R^\dmn_+$, then the area integral appearing in these problems may be written as
\begin{align*}
\int_\Omega \abs{\nabla^m u(X)}^2 \dist(X,\partial\Omega)\,dX
&=\int_{\R^n}\int_0^\infty \abs{\nabla^{m}u(x,t)}^2\,{t}\,dt\,dx
,\\
\int_\Omega \abs{\nabla^{m+1} u(X)}^2 \dist(X,\partial\Omega)\,dX
&=\int_{\R^n}\int_0^\infty \abs{\nabla^{m+1}u(x,t)}^2\,{t}\,dt\,dx
.\end{align*}

In \cite{BarHM18}, we shall establish well posedness of the Neumann problems
\begin{gather}
\label{A:eqn:Neumann:regular}
Lu=0\text{ in }\R^\dmn_+,\quad \M_{\mat A}^+ u=\arr g, \quad \int_{\R^n}\int_0^\infty \abs{\nabla^{m}\partial_t u(x,t)}^2\,{t}\,dt\,dx\leq C \doublebar{\arr g}_{L^2(\partial\Omega)}^2
,\\
\label{A:eqn:Neumann:rough}
Lu=0\text{ in }\R^\dmn_+,\quad \M_{\mat A}^+ u=\arr g, \quad \int_{\R^n}\int_0^\infty \abs{\nabla^{m} u(x,t)}^2\,{t}\,dt\,dx\leq C \doublebar{\arr g}_{\dot W_{-1}^2(\partial\Omega)}^2
\end{gather}
whenever $\mat A$ is a bounded, $t$-independent, self-adjoint matrix of coefficients satisfying an ellipticity condition (stronger than the bound~\eqref{A:eqn:elliptic} below). Our solution $u$ will be of the form $u=\D^{\mat A} \arr f$ for some appropriate $\arr f$ with $\doublebar{\arr f}_{\dot W^2_1(\R^n)}\leq C\doublebar{\arr g}_{L^2(\R^n)}$ or $\doublebar{\arr f}_{L^2(\R^n)}\leq C\doublebar{\arr g}_{\dot W^2_{-1}(\R^n)}$; thus, the bound on solutions in the statement of well posedness is a direct consequence of the bound~\eqref{A:eqn:D:square:rough:intro} of the present paper or of the bound~\eqref{A:eqn:D:square} of \cite{BarHM17}.

In order to establish that an appropriate $\arr g$ exists, we shall use some arguments introduced in \cite{Ver84,BarM13,BarM16A}. In order to apply these arguments, we shall require boundedness of the single layer potential as well as the double layer potential (the bounds \eqref{A:eqn:S:square:rough:intro} and~\eqref{A:eqn:S:square}).

We will also need trace theorems; that is, we shall need the fact that
any solution $u$ to $Lu=0$ in $\R^\dmn_+$ with
\begin{equation*}\int_{\R^n}\int_{{0}}^\infty \abs{\nabla^m u(x,t)}^2\,{t}\,dt\,dx<\infty\end{equation*}
satisfies $\nabla^{m-1}u\big\vert_{\partial\R^\dmn_+}\in L^2(\R^n)$ and $\M_{\mat A}^+ u\in \dot W^2_{-1}(\R^n)$, and that any solution $u$ with
\begin{equation*}\int_{\R^n}\int_{{0}}^\infty \abs{\nabla^m \partial_t u(x,t)}^2\,{t}\,dt\,dx<\infty\end{equation*}
satisfies $\nabla^{m-1}u\big\vert_{\partial\R^\dmn_+}\in \dot W^2_1(\R^n)$ and $\M_{\mat A}^+ u\in L^2(\R^n)$.
These two facts shall be established in \cite{BarHM17pB}.
We remark that we will use the estimates~\eqref{A:eqn:lusin:p:variant:intro} and~\eqref{A:eqn:S:Schwartz:p} in \cite{BarHM17pB}; it is this intended use that makes these estimates of immediate interest to us.

\section{Definitions}

In this section, we will provide precise definitions of the notation and concepts used throughout this paper. We mention that throughout this paper, we will work with elliptic operators~$L$ of order~$2m$, $m\geq 1$, in the divergence form \eqref{A:eqn:divergence} acting on functions defined on~$\R^\dmn$, $\dmn\geq 2$.

If $Q$ is a cube, we let $\ell(Q)$ be its side length, and we let $cQ$ be the concentric cube of side length $c\ell(Q)$. If $E$ is a set of finite measure, we let $\fint_E f(x)\,dx=\frac{1}{\abs{E}}\int_E f(x)\,dx$.

\subsection{Multiindices and arrays of functions}

We will reserve the letters $\alpha$, $\beta$, $\gamma$, $\zeta$ and~$\xi$ to denote multiindices in $(\N_0)^\dmn$. (Here $\N_0$ denotes the nonnegative integers.) If $\zeta=(\zeta_1,\zeta_2,\dots,\zeta_\dmn)$ is a multiindex, then we define $\abs{\zeta}$ and $\partial^\zeta$
in the usual ways, as $\abs{\zeta}=\zeta_1+\zeta_2+\dots+\zeta_\dmn$ and $\partial^\zeta=\partial_{x_1}^{\zeta_1}\partial_{x_2}^{\zeta_2} \cdots\partial_{x_\dmn}^{\zeta_\dmn}$.

Recall that a vector $\vec H$ is a list of numbers (or functions) indexed by integers $j$ with $1\leq j\leq N$ for some $N\geq 1$. We similarly let an array $\arr H$ be a list of numbers or functions indexed by multiindices~$\zeta$ with $\abs\zeta=k$ for some $k\geq 1$. It is by now standard to denote arrays using overdots.
In particular, if $\varphi$ is a function with weak derivatives of order up to~$k$, then we view $\nabla^k\varphi$ as such an array.

The inner product of two such arrays of numbers $\arr F$ and $\arr G$ is given by
\begin{equation*}\bigl\langle \arr F,\arr G\bigr\rangle =
\sum_{\abs{\zeta}=k}
\overline{F_{\zeta}}\, G_{\zeta}.\end{equation*}
If $\arr F$ and $\arr G$ are two arrays of functions defined in a set $\Omega$ in Euclidean space, then the inner product of $\arr F$ and $\arr G$ is given by
\begin{equation*}\bigl\langle \arr F,\arr G\bigr\rangle_\Omega =
\sum_{\abs{\zeta}=k}
\int_{\Omega} \overline{F_{\zeta}(X)}\, G_{\zeta}(X)\,dX.\end{equation*}

We let $\vec e_j$ be the unit vector in $\R^\dmn$ in the $j$th direction; notice that $\vec e_j$ is a multiindex with $\abs{\vec e_j}=1$. We let $\arr e_{\zeta}$ be the ``unit array'' corresponding to the multiindex~$\zeta$; thus, $\langle \arr e_{\zeta},\arr F\rangle = F_{\zeta}$.

We will let $\nabla_\pureH$ denote either the gradient in~$\R^n$, or the $n$ horizontal components of the full gradient~$\nabla$ in $\R^\dmn$.
If $\zeta$ is a multiindex with $\zeta_\dmn=0$, we will occasionally use the terminology $\partial_\pureH^\zeta$ to emphasize that the derivatives are taken purely in the horizontal directions.

\subsection{Elliptic differential operators and their bounds}

Let $\mat A = \begin{pmatrix} A_{\alpha\beta} \end{pmatrix}$ be a matrix of measurable coefficients defined on $\R^\dmn$, indexed by multtiindices $\alpha$, $\beta$ with $\abs{\alpha}=\abs{\beta}=m$. If $\arr F$ is an array, then $\mat A\arr F$ is the array given by
\begin{equation*}(\mat A\arr F)_{\alpha} =
\sum_{\abs{\beta}=m}
A_{\alpha\beta} F_{\beta}.\end{equation*}

We will consider coefficients that satisfy the G\r{a}rding inequality
\begin{align}
\label{A:eqn:elliptic}
\re {\bigl\langle\nabla^m \varphi,\mat A\nabla^m \varphi\bigr\rangle_{\R^\dmn}}
&\geq
	\lambda\doublebar{\nabla^m\varphi}_{L^2(\R^\dmn)}^2
	\quad\text{for all $\varphi\in\dot W^2_m(\R^\dmn)$}
\end{align}
and the bound
\begin{align}
\label{A:eqn:elliptic:bounded}
\doublebar{\mat A}_{L^\infty(\R^\dmn)}
&\leq
	\Lambda
\end{align}
for some $\Lambda>\lambda>0$.
In this paper we will focus exclusively on coefficients that are $t$-inde\-pen\-dent, that is, that satisfy formula~\eqref{A:eqn:t-independent}.

We let $L$ be the $2m$th-order divergence-form operator associated with~$\mat A$ acting on Sobolev spaces. That is, we say that $L u=0$ in an open set $\Omega\subseteq\R^\dmn$ in the weak sense if the weak gradient $\nabla^m u$ is locally square integrable in~$\Omega$ and if, for every $\varphi$ smooth and compactly supported in~$\Omega$, we have that
\begin{equation}
\label{A:eqn:L}
\bigl\langle\nabla^m\varphi, \mat A\nabla^m u\bigr\rangle_\Omega
=\sum_{\abs{\alpha}=\abs{\beta}=m}
\int_{\Omega}\partial^\alpha \bar \varphi\, A_{\alpha\beta}\,\partial^\beta u
=
0
.\end{equation}

Throughout the paper we will let $C$ denote a constant whose value may change from line to line, but which depends only on the dimension $\dmn$, the ellipticity constants $\lambda$ and $\Lambda$ in the bounds \eqref{A:eqn:elliptic} and~\eqref{A:eqn:elliptic:bounded}, and the order~$2m$ of our elliptic operators. Any other dependencies will be indicated explicitly.

\subsection{Function spaces and boundary data}

Let $\Omega\subseteq\R^n$ or $\Omega\subseteq\R^\dmn$ be a measurable set in Euclidean space. We will let $L^p(\Omega)$ denote the usual Lebesgue space with respect to Lebesgue measure with norm given by
\begin{equation*}\doublebar{f}_{L^p(\Omega)}=\biggl(\int_\Omega \abs{f(x)}^p\,dx\biggr)^{1/p}.\end{equation*}

If $\Omega$ is a connected open set, then we let the homogeneous Sobolev space $\dot W^p_m(\Omega)$ be the space of equivalence classes of functions $u$ that are locally integrable in~$\Omega$ and have weak derivatives in $\Omega$ of order up to~$m$ in the distributional sense, and whose $m$th gradient $\nabla^m u$ lies in $L^p(\Omega)$. Two functions are equivalent if their difference is a polynomial of order~$m-1$.
We impose the norm
\begin{equation*}\doublebar{u}_{\dot W^p_m(\Omega)}=\doublebar{\nabla^m u}_{L^p(\Omega)}.\end{equation*}
Then $u$ is equal to a polynomial of order $m-1$ (and thus equivalent to zero) if and only if its $\dot W^p_m(\Omega)$-norm is zero.

The use of the dot to denote a homogeneous Sobolev space (in contrast to the inhomogeneous space $W^p_m(\Omega)$ with the norm $\doublebar{u}_{W^p_m(\Omega)}=\sum_{k=0}^m \doublebar{\nabla^k u}_{L^p(\Omega)}$) is by now standard in the theory of Sobolev spaces and related function spaces measuring smoothness. As mentioned above, we also use dots to denote arrays of functions. We apologize for this potentially confusing notation, but the established conventions of these two areas seem to require it.

\subsubsection{Dirichlet boundary data and spaces.}

If $u$ is defined in $\R^\dmn_+$, we let its Dirichelt boundary values be, loosely, the boundary values of the gradient $\nabla^{m-1} u$. More precisely, we let the Dirichlet boundary values be the array of functions $\Tr_{m-1}u=\Tr_{m-1}^+ u$, indexed by multiindices $\gamma$ with $\abs\gamma=m-1$, and given by
\begin{equation}
\label{A:eqn:Dirichlet}
\begin{pmatrix}\Tr_{m-1}^+ u\end{pmatrix}_{\gamma}
=f \quad\text{if}\quad
\lim_{t\to 0^+} \doublebar{\partial^\gamma u(\,\cdot\,,t)-f}_{L^1(K)}=0
\end{equation}
for all compact sets $K\subset\R^n$. If $u$ is defined in $\R^\dmn_-$, we define $\Tr_{m-1}^- u$ similarly.

We will be concerned with boundary values in Lebesgue or Sobolev spaces. However, observe that the different components of $\Tr_{m-1}u$ arise as derivatives of a common function, and thus must satisfy certain compatibility conditions. We will define the Whitney spaces of functions that satisfy these compatibility conditions and have certain smoothness properties as follows.
\begin{defn} \label{A:dfn:Whitney}
Let
\begin{equation*}\mathfrak{D}=\{\Tr_{m-1}\varphi:\varphi\text{ smooth and compactly supported in $\R^\dmn$}\}.\end{equation*}

We let $\dot W\!A^p_{m-1,0}(\R^n)$ be the closure of the set $\mathfrak{D}$ in $L^p(\R^n)$.

We let $\dot W\!A^p_{m-1,1}(\R^n)$ be the closure of $\mathfrak{D}$ in $\dot W^p_1(\R^n)$, so that $\doublebar{\arr f}_{\dot W\!A^p_{m-1,1}(\R^n)}=\doublebar{\nabla_\pureH \arr f}_{L^p(\R^n)}$.

Finally, we let $\dot W\!A^2_{m-1,1/2}(\R^n)$ be the closure of $\mathfrak{D}$ in the homogeneous Besov space $\dot B^{2,2}_{1/2}(\R^n)$, where the Besov norm is given by
\begin{equation*}
\doublebar{f}_{\dot B^{2,2}_{1/2}(\R^n)} = \biggl(\int_{\R^n} \abs{\widehat f(\omega)}^2\,\abs{\omega}\,d\omega\biggr)^{1/2}
\end{equation*}
where $\widehat f$ denotes the Fourier transform of~$f$.
\end{defn}

The double layer potential will initially be defined on the space $\dot W\!A^2_{m-1,1/2}(\R^n)$; the bound~\eqref{A:eqn:D:square:rough:intro} (as well as the known result~\eqref{A:eqn:D:square}) are essentially statements that we may extend $\mathcal{D}^{\mat A}$ by density to $\dot W\!A^2_{m-1,0}(\R^n)$ and $\dot W\!A^2_{m-1,1}(\R^n)$.

From our perspective, the most interesting property of the space $\dot W\!A^2_{m-1,1/2}(\R^n)$ is the following well known trace and extension lemma.

\begin{lem}\label{A:lem:Besov}
If $u\in \dot W^2_m(\R^\dmn_+)$ then $\Tr_{m-1}^+u\in \dot W\!A^2_{m-1,1/2}(\R^n)$, and furthermore
\begin{equation*}\doublebar{\Tr_{m-1}^+u}_{\dot W\!A^2_{m-1,1/2}(\R^n)}\leq C \doublebar{\nabla^m u}_{L^2(\R^\dmn_+)}.\end{equation*}
Conversely, if $\arr f\in \dot W\!A^2_{m-1,1/2}(\R^n)$, then there is some $F\in \dot W^2_m(\R^\dmn_+)$ such that $\Tr_{m-1}^+F=\arr f$ and such that
\begin{equation*}\doublebar{\nabla^m F}_{L^2(\R^\dmn_+)}\leq C \doublebar{\arr f}_{\dot W\!A^2_{m-1,1/2}(\R^n)}.\end{equation*}
\end{lem}

\begin{proof}
If $\dot W^2_m(\R^\dmn_+)$ and $\dot W\!A^2_{m-1,1/2}(\R^n)$ are replaced by their inhomogeneous counterparts, then this lemma is a special case of \cite{Liz60}. For the homogeneous spaces that we consider, the $m=1$ case of this lemma is a special case of \cite[Section~5]{Jaw77}.

If $m\geq 2$ and $u\in \dot W^2_m(\R^\dmn_+)$, then $\partial^\gamma u\in \dot W^2_1(\R^\dmn_+)$ for any $\abs\gamma=m-1$, and so $\Tr_{m-1}^+u=\begin{pmatrix}\Trace\partial^\gamma u\end{pmatrix}_{\abs\gamma=m-1}\in \dot B^{2,2}_{1/2}(\R^n)$. By density of $C^\infty_0(\R^\dmn)$ functions in $\dot W^2_m(\R^\dmn_+)$, $\Tr_{m-1}^+u$ lies in the distinguished subspace $\dot W\!A^2_{m-1,1/2}(\R^n)$ of~$\dot B^{2,2}_{1/2}(\R^n)$.

If $\arr f=\Tr_{m-1}^+\varphi$ for some $\varphi\in C^\infty_0(\R^n)$, then extensions may easily be constructed using the Fourier transform in $\R^n$, for example by letting
\begin{equation*}F(x,t)=F_t(x),\text{ where }\widehat {F_t}(\omega)=
\sum_{j=0}^{m-1}
\frac{t^j}{j!}e^{-(t^2\abs\omega^2)^m}\widehat{\varphi_j}(\omega)
\text{ and }\varphi_j(x)=\partial_\dmn^j\varphi(x,0)
.\end{equation*}
Extensions of arbitrary arrays in $\dot W\!A^2_{m-1,1/2}(\R^n)$ may be constructed by density.
\end{proof}

\subsection{The fundamental solution}\label{A:sec:fundamental}
The double and single layer potentials may be formulated in terms of the fundamental solution for~$L$; we will define the fundamental solution in this section.

For any $\arr H\in L^2(\R^\dmn)$, by the Lax-Milgram lemma there is a unique element $\Pi^L\arr H$ of $\dot W^2_m(\R^\dmn)$ that satisfies
\begin{equation}\label{A:eqn:newton}
\langle \nabla^m\varphi, \mat A\nabla^m \Pi^L\arr H\rangle_{\R^\dmn}=\langle \nabla^m\varphi, \arr H\rangle_{\R^\dmn}\end{equation}
for all $\varphi\in \dot W^2_m(\R^\dmn)$. The (gradient of the) fundamental solution is the kernel of the operator $\Pi^L$. It was constructed and certain properties were established in \cite{Bar16}; we summarize some of the main results here.

\begin{thm}[{\cite[Theorem~62 and Lemma~69]{Bar16}}]\label{A:thm:fundamental}
Suppose that $L$ is an operator of the form~\eqref{A:eqn:L} of order~$2m$, $m\geq 1$, acting on functions defined in $\R^\dmn$, $\dmnMinusOne\geq 1$, associated to coefficients~$\mat A$ that satisfy the bounds \eqref{A:eqn:elliptic} and~\eqref{A:eqn:elliptic:bounded}.
Then there exists a function $E^L(X,Y)$ with the following properties.

Let $q$, $s\in \{0,1\}$.
There is some $\varepsilon>0$ such that if $X_0$, $Y_0\in\R^\dmn$, if $0<4r<R<\abs{X_0-Y_0}/3$, and if $q=0$ or $\dmn\geq 3$, then
\begin{equation}
\label{A:eqn:fundamental:far}
\int_{B(Y_0,r)}\int_{B(X_0,R)} \abs{\nabla^{m-s}_X \nabla^{m-q}_Y E^L(X,Y)}^2\,dX\,dY \leq C r^{2q} R^{2s} \biggl(\frac{r}{R}\biggr)^\varepsilon
.\end{equation}
If $q=1=n$ then we instead have the bound
\begin{equation}
\label{A:eqn:fundamental:far:lowest:2}
\int_{B(Y_0,r)}\int_{B(X_0,R)} \abs{\nabla^{m-s}_X \nabla^{{m-1}}_Y E^L(X,Y)}^2\,dX\,dY \leq C(\delta)\, r^{2} R^{2s} \biggl(\frac{R}{r}\biggr)^\delta
\end{equation}
for all $\delta>0$ and some constant $C(\delta)$ depending on~$\delta$.

We also have the symmetry property
\begin{equation}
\label{A:eqn:fundamental:symmetric}
\partial_X^\zeta\partial_Y^\xi E^L(X,Y) = \overline{\partial_X^\zeta\partial_Y^\xi E^{L^*}(Y,X)}
\end{equation}
as locally $L^2$ functions, for all multiindices $\zeta$, $\xi$ with $\abs{\zeta}$, $\abs{\xi}\in \{m-1,m\}$,
and where $L^*$ is the elliptic operator associated to~$\mat A^*$, the adjoint matrix to~$\mat A$.

If $\abs\zeta=m-1$ and if $\dmn\geq 3$ then
\begin{equation}
\label{A:eqn:fundamental:low}
\partial^\zeta
\Pi^L\arr H(X)
= \ \sum_{\abs{\beta}=m} \int_{\R^\dmn} 	\partial_X^\zeta\partial_Y^\beta E^L(X,Y)\,H_{\beta}(Y)\,dY
\end{equation}
for almost every $X\in\R^\dmn$, and for all $\arr H\in L^2(\R^\dmn)$ that are also locally in $L^{P}(\R^\dmn)$, for some $P>\dmn$.
Furthermore, there is some $\varepsilon>0$ such that if $2-\varepsilon<p<2+\varepsilon$ then $\Pi^L$ extends to a bounded operator $L^p(\R^\dmn)\mapsto \dot W_m^p(\R^\dmn)$. If $\abs\zeta=m-1$, and if $\arr H\in L^p(\R^2)$ is locally in $L^P(\R^2)$ for some $2-\varepsilon<p<2<P$, then formula~\eqref{A:eqn:fundamental:low} is still valid even in dimension $\dmn=2$.

In the case of $\abs{\alpha}=m$, we still have that
if $\abs\alpha=\abs\beta=m$, then $\partial_X^\alpha\partial_Y^\beta E^L(X,Y)$ exists in the weak sense and is locally integrable. Furthermore, if $\abs{\alpha}=m$ and $\dmn\geq 2$ then
\begin{equation}
\label{A:eqn:fundamental:2}
\partial^\alpha\Pi^L\arr H(X)
	= \sum_{\abs{\beta}=m} \int_{\R^\dmn} 	\partial_X^\alpha\partial_Y^\beta E^L(X,Y)\,H_{\beta}(Y)\,dY
\end{equation}
for almost every $X\notin\supp \arr H$, and for all $\arr H\in L^2(\R^\dmn)$ whose support is not all of $\R^\dmn$.

Finally, the fundamental solution is unique in the following sense: if $\widetilde E^L$ is any other function that satisfies the bounds~\eqref{A:eqn:fundamental:far},~\eqref{A:eqn:fundamental:far:lowest:2} and formula~\eqref{A:eqn:fundamental:2}, then \begin{equation*}\nabla_X^{m-q}\nabla_Y^{m-s}E^L(X,Y)=\nabla_X^{m-q}\nabla_Y^{m-s}\widetilde E^L(X,Y)\end{equation*} for all $q$, $s\in \{0,1\}$.
\end{thm}

\begin{rmk}\label{A:rmk:GNS}
We comment on the definition in \cite{Bar16} of the lower order derivatives $\partial^\zeta\Pi^L\arr H$ in the bound~\eqref{A:eqn:fundamental:low}. This definition needs some care, for recall that $\Pi^L\arr H$ is an element of $\dot W^p_m(\R^\dmn)$, that is, is an equivalence class of functions up to adding polynomials of order at most~$m-1$.

However, by the Gagliardo-Nirenberg-Sobolev inequality (see, for example, Section~5.6.1 in \cite{Eva98}), if $w$ is a function with $\nabla^m w\in L^p(\R^\dmn)$, $p<\dmn$, and if $w$ is compactly supported, then $\nabla^{m-1} w\in L^{p_1}(\R^\dmn)$, where $p_1=p\,\pdmn/(\dmn-p)$.
If $p<\infty$ then compactly supported functions are dense in $\dot W^p_m(\R^\dmn)$. By a limiting argument and the definition of weak derivatives, we see that if $\nabla^m w\in L^p(\R^\dmn)$, then there is a function $u$ with $\nabla^m w=\nabla^m u$ and with $\nabla^{m-1} u\in L^{p_1}(\R^\dmn)$.

Recalling the definition of $\dot W^p_m(\R^\dmn)$, we see that if $\Pi^L$ is bounded $L^p(\R^\dmn)\mapsto \dot W^p_m(\R^\dmn)$ and if $\arr H\in L^p(\R^\dmn)$, $p<\dmn$, then there is a representative $u$ of the equivalence class $\Pi^L\arr H$ such that $\nabla^{m-1} u$ in $L^{p_1}(\R^\dmn)$. We write $\partial^\zeta\Pi^L\arr H=\partial^\zeta u$ whenever $\abs\zeta=m-1$.
Because representatives of $\Pi^L\arr H$ differ by polynomials, it is straightforward to establish that any two such representatives $u_1$ and $u_2$ must satisfy $\nabla^{m-1}u_1=\nabla^{m-1}u_2$, and so $\partial^\zeta\Pi^L\arr H$ is well defined.
\end{rmk}

Recall that we are concerned only with operators $L$ associated with coefficients~$\mat A$ that are $t$-independent. This gives us one other property of the fundamental solution. For all multiindices~$\zeta$, $\xi$ as in formula~\eqref{A:eqn:fundamental:symmetric}, by uniqueness we have that
\begin{equation}
\partial_{x,t}^\zeta \partial_{y,s}^\xi E^L(x,t,y,s)
= \partial_{x,t}^\zeta \partial_{y,s}^\xi E^L(x,t-s,y,0)
= \partial_{x,t}^\zeta \partial_{y,s}^\xi E^L(x,0,y,s-t)
\end{equation}
for all $x$, $y\in\R^n$ and all $s$, $t\in\R$. Here $E^L(x,t,y,s)=E^L((x,t),(y,s))$.
Therefore,
\begin{equation}
\label{A:eqn:fundamental:vertical}
\partial_{x,t}^\zeta \partial_{y,s}^\xi \partial_t E^L(x,t,y,s)= -\partial_{x,t}^\zeta \partial_{y,s}^\xi \partial_s E^L(x,t,y,s)
.\end{equation}

\subsection{Layer potentials}\label{A:sec:dfn:potentials}
Layer potentials may also be generalized to the higher order case. In this section we define our formulations of higher-order layer potentials; this is the formulation used in \cite{BarHM17} and is similar to that used in \cite{Agm57,CohG83,Ver05,MitM13B,MitM13A}.

Suppose that $\arr f\in \dot W\!A^2_{m-1,1/2}(\R^n)$.
By Lemma~\ref{A:lem:Besov}, there is some $F\in \dot W^2_m(\R^\dmn_+)$ that satisfies $\arr f=\Tr_{m-1}^+ F$.
We define the double layer potential of $\arr f$ as
\begin{align}
\label{A:dfn:D:newton}
\D^{\mat A}\arr f &= -\1_+ F + \Pi^L(\1_+ \mat A\nabla^m F)
\end{align}
where $\1_+ F$ denotes the extension of $F$ by zero to $\R^\dmn$.
$\D^{\mat A}\arr f$ is well-defined, that is, does not depend on the choice of~$F$; see \cite{BarHM17}.

Similarly, let $\arr g$ be a bounded linear functional on $\dot W\!A^2_{m-1,1/2}(\R^n)$.
There is some $\arr G\in L^2(\R^\dmn_+)$ such that $\langle \arr G, \nabla^m\varphi\rangle_{\R^\dmn_+} = \langle \arr g, \Tr_{m-1}^+\varphi\rangle_{\partial{\R^\dmn_+}}$ for all $\varphi\in \dot W^2_m$; see \cite{BarHM17}. Let $\1_+\arr G$ denote the extension of $\arr G$ by zero to $\R^\dmn$.
We define
\begin{align}
\label{A:dfn:S:newton}
\s^L\arr g&=\Pi^L(\1_+\arr G)
.\end{align}
Again, $\s^L\arr g$ does not depend on the choice of extension~$\arr G$; see \cite{BarHM17}.

\begin{rmk}
In general, the operator $L$ given by formula~\eqref{A:eqn:divergence} or~\eqref{A:eqn:L} may be associated to many different choices of coefficients~$\mat A$. For example, if $L$ is associated to the coefficients~$\mat A$, then $L$ is also associated to the coefficients~$\mat{\widetilde A}$, where $\widetilde A_{\alpha\beta}=A_{\alpha\beta}+M_{\alpha\beta}$ for some constant coefficients $\mat M$ that satisfy $M_{\alpha\beta}=-M_{\beta\alpha}$.

The fundamental solution $E^L$ of Theorem~\ref{A:thm:fundamental} and the single layer potential $\s^L$ given by formula~\eqref{A:dfn:S:newton} are independent of the choice of coefficient matrix~$\mat A$, depending only on the associated operator~$L$. However, the double layer potential $\D^{\mat A}$ given by formula~\eqref{A:dfn:D:newton} does depend on the choice of coefficients~$\mat A$.

These issues are important for some topics in the higher order theory. In particular, the Neumann boundary values $\M_{\mat A}^\Omega$, which were mentioned in the problems \eqref{A:eqn:Neumann:biharmonic}, \eqref{A:eqn:Neumann:regular}, and~\eqref{A:eqn:Neumann:rough} above and studied in \cite{Ver05,BarHM18}, also depend on the choice of $\mat A$. The choice can be important; indeed it was shown in \cite{Ver05} that the Neumann problem for $\Delta^2$ is well posed for some choices of~$\mat A$ and ill posed for others.
These issues are discussed in \cite{BarM16B,BarHM17} and at length in~\cite{Ver10}.
\end{rmk}

Using Theorem~\ref{A:thm:fundamental}, we may rewrite the double and single layer potentials in terms of the fundamental solution. This often allows a more straightforward calculation of various bounds. See \cite[Section~2.4]{BarHM17} for the details. If $\arr f\in \dot W\!A^2_{m-1,1/2}(\R^n)$ and $\arr g\in (\dot W\!A^2_{m-1,1/2}(\R^n))^*$,
then
\begin{align}
\label{A:eqn:D:fundamental}
\partial^\zeta \D^{\mat A} \arr f(x,t)
&=
- \!\!\sum_{\substack{\scriptstyle\abs{\beta}=m\\\abs{\xi}=m}} \int_{\R^\dmn_-} \partial_{x,t}^\zeta \partial_{y,s}^\beta E^L(x,t,y,s)\,A_{\beta\xi}(y,s) \, \partial^\xi F(y,s)\,ds\,dy
,\\
\label{A:eqn:S:fundamental}
\partial^\zeta\s^L\arr g(x,t)
&=
\sum_{\abs{\gamma}=m-1}
\int_{\R^n} \partial_{x,t}^\zeta \partial_{y,s}^\gamma E^L(x,t,y,0)\,g_{\gamma}(y) \,dy
\end{align}
for any $t>0$, any $F\in \dot W^2_m(\R^\dmn_-)$ with $\Tr_{m-1}^-F=\arr f$, and any $\abs\zeta=m$. Here we have adopted the shorthand that
\begin{equation}\label{A:eqn:S:shorthand}\partial_{x,t}^\zeta \partial_{y,s}^\gamma E^L(x,t,y,0)=\partial_{x,t}^\zeta\partial_{y,s}^\gamma E^L(x,t,y,s)\big\vert_{s=0}=\partial_X^\zeta\partial_Y^\gamma E^L(X,Y)\big\vert_{X=(x,t),\>Y=(y,0)}.\end{equation}

\begin{rmk}
\label{A:rmk:S:definition:dense}
The double layer potential is defined on $\dot W\!A^2_{m-1,1/2}(\R^n)$, and in particular on the dense subspace $\mathfrak{D}$ of Definition~\ref{A:dfn:Whitney}. We will establish bounds of the form $\doublebar{\D^{\mat A}\arr f}_{\mathfrak{X}}\leq C\doublebar{\arr f}_{\dot W\!A^p_{m-1,s}(\R^n)}$ for all $\arr f\in\mathfrak{D}$, for various values of $p$ and $s$ and tent spaces~$\mathfrak{X}$; this allows us to extend $\D^{\mat A}$ to an operator on ${\dot W\!A^p_{m-1,s}(\R^n)}$ by density.

It would be convenient to have a similar well-behaved dense subspace of the domain of $\s^L$. Let $\arr g$ be an array of functions that are smooth and compactly supported, and furthermore, let $\int_{\R^n} \arr g=0$. Then $\abs{\widehat{g_\gamma}(\omega)}\leq C(\arr g)\abs{\omega}$, and so by Plancherel's theorem and the definition of $\dot W\!A^2_{m-1,1/2}(\R^n)$, if $\arr f\in \dot W\!A^2_{m-1,1/2}(\R^n)$, then $\abs{\langle \arr g,\arr f\rangle_{\R^n}}\leq C(\arr g)\doublebar{\arr f}_{\dot W\!A^p_{m-1,1/2}(\R^n)}$. Thus, such arrays $\arr g$ are necessarily in the domain of $\s^L$. It is straightforward to establish that such arrays are in fact dense.
\end{rmk}

\begin{rmk}\label{A:rmk:potentials:decay}
In the present paper we will be primarily concerned with the highest order derivatives $\nabla^m \D^{\mat A}\arr f$ and $\nabla^m \s^L\arr g$. In \cite{BarHM18} and in future work, we will also need to consider the lower order derivatives $\nabla^{m-1} \D^{\mat A}\arr f$ and $\nabla^{m-1} \s^L\arr g$.

We will always normalize $\Pi^L$ as in Remark~\ref{A:rmk:GNS}.
If $\dmn\geq 3$, $\arr f\in \dot W\!A^2_{m-1,1/2}(\R^n)$, and $\arr g\in (\dot W\!A^2_{m-1,1/2}(\R^n))^*$, or if $\dmn=2$ and the extensions $F$ and $\arr G$ of formulas \eqref{A:dfn:D:newton} and~\eqref{A:dfn:S:newton} lie in $\dot W^p_m(\R^2_+)$ and $L^p(\R^2_+)$, respectively, for appropriate $p<2$,
then $\nabla^{m-1}\D^{\mat A}\arr f\in \dot W^{p_1}_m(\R^\dmn)$, $\nabla^{m-1}\s^L\arr g\in \dot W^{p_1}_m(\R^\dmn)$. Furthermore, if we impose further conditions on~$\arr f$ and~$\arr g$ (for example, the conditions of Remark~\ref{A:rmk:S:definition:dense}), then we may extend $\arr f$ and $\arr g$ to functions $F\in \dot W_m^p(\R^\dmn_-)\cap \dot W_m^P(\R^\dmn_-)$ and arrays $\arr G\in L^p(\R^\dmn)\cap L^P(\R^\dmn)$ for any $1<p<P<\infty$; thus, choosing $p<2$ and $P$ large enough, we see that formulas \eqref{A:eqn:D:fundamental} and~\eqref{A:eqn:S:fundamental} are valid for any multiindex $\zeta$ with $\abs\zeta\in \{m-1,m\}$.
\end{rmk}

In the second-order case, a variant $\s^L\nabla$ of the single layer potential is often used; see \cite{AlfAAHK11,HofMitMor15,HofMayMou15}. We will define an analogous operator in this case.

Let $\alpha$ be a multiindex with $\abs\alpha=m$. If $\alpha_\dmn>0$, let
\begin{equation}\label{A:eqn:S:S:vertical}\s^L_\nabla (h\arr e_\alpha)(x,t) = -\partial_t\s^L(h\arr e_\gamma)(x,t)\quad\text{where }\alpha=\gamma+\vec e_\dmn.\end{equation}
If $\alpha_\dmn<\abs\alpha=m$, then there is some $j$ with $1\leq j\leq n$ such that $\vec e_j\leq \alpha$. If $h$ is smooth and compactly supported, let
\begin{equation}\label{A:eqn:S:S:horizontal}\s^L_\nabla (h\arr e_\alpha)(x,t) = -\s^L(\partial_{x_j} h\arr e_\gamma)(x,t)\quad\text{where }\alpha=\gamma+\vec e_j.\end{equation}
By applying formula~\eqref{A:eqn:S:fundamental} for the single layer potential, and by either applying formula~\eqref{A:eqn:fundamental:vertical} or integrating by parts, we see that if $\arr h$ is smooth, compactly supported and integrates to zero, then
\begin{equation}
\label{A:eqn:S:variant}
\partial^\zeta \s^L_\nabla\arr h(x,t)
=\sum_{\abs\alpha=m}\int_{\R^n} \partial_{x,t}^\zeta \partial_{y,s}^\alpha E^L(x,t,y,0)\,h_\alpha(y)\,dy
\end{equation}
for $\abs\zeta\in \{m-1,m\}$ and for $\partial_{y,s}^\alpha E^L(x,t,y,0)$ as in formula~\eqref{A:eqn:S:shorthand}.
In particular, if $1\leq \alpha_\dmn\leq m-1$, then the two formulas \eqref{A:eqn:S:S:vertical} and \eqref{A:eqn:S:S:horizontal} coincide, and furthermore, the choice of distinguished direction $x_j$ in formula~\eqref{A:eqn:S:S:horizontal} does not matter.

\begin{rmk}\label{A:rmk:S:S:rough}
Let $\dot W^2_{-1}(\R^n)=(\dot W^2_1(\R^n))^*$ be the space of bounded linear operators on $\dot W^2_1(\R^n)$. An integration by parts argument shows that if $g\in \dot W^2_{-1}(\R^n)$, then formally $g=\Div \vec h$ for some $\vec h\in L^2(\R^n)$.

By formula~\eqref{A:eqn:S:S:horizontal}, if $\arr g\in \dot W^2_{-1}(\R^n)$, then there is some $\arr h\in L^2(\R^n)$ with $\doublebar{\arr h}_{L^2(\R^n)}\approx \doublebar{\arr g}_{\dot W^2_{-1}(\R^n)}$ and such that $\s^L\arr g=\s^L_\nabla\arr h$.

Thus, the bound \eqref{A:eqn:S:square:variant:intro} implies the bound~\eqref{A:eqn:S:square:rough:intro}; by formulas~\eqref{A:eqn:S:S:vertical} and~\eqref{A:eqn:S:S:horizontal}, the bounds \eqref{A:eqn:S:square} and~\eqref{A:eqn:S:square:rough:intro} imply the bound~\eqref{A:eqn:S:square:variant:intro}.
\end{rmk}

\section{Known results and preliminary arguments}

To prove Theorem~\ref{A:thm:square:rough} and our other estimates on layer potentials, we will need to use a number of known results from the theory of higher order differential equations. We gather these results in this section.

\begin{rmk} \label{A:rmk:lower}
Let $\rho(x,t)=(x,-t)$ be the change of variables that interchanges the upper and lower half-spaces. Let $ A_{\alpha\beta}^\rho = (-1)^{\alpha_\dmn+\beta_\dmn} A_{\alpha\beta}$, and let $f_\gamma^\rho=(-1)^{\gamma_\dmn}f_\gamma$. Notice that the map $\arr f\mapsto\arr f^\rho$ preserves the dense subspace $\mathfrak{D}$ of Definition~\ref{A:dfn:Whitney}, and thus the Whitney spaces $\dot W\!A^p_{m-1,s}(\R^n)$.

It is straightforward to establish that $\Pi^L\arr H = (\Pi^{L^\rho} (\arr H^\rho\circ\rho))\circ \rho$, and so
\begin{equation*}E^L(x,t,y,s) = E^{L^\rho}(\rho(x,t),\rho(y,s))\end{equation*}
and so by formula~\eqref{A:eqn:S:fundamental},
\begin{equation*}\s^L\arr g(x,-t) = \s^{L^\rho}\arr g^\rho(x,t).\end{equation*}
Thus, to prove the bound \eqref{A:eqn:S:square:rough:intro} or Theorem~\ref{A:thm:Carleson:intro}, it suffices to work only in the upper half-space, as the results in the lower half-space follow via this change of variables.

By \cite[formula~(2.27)]{BarHM17}, if $\1_-$ denotes extension by zero from the lower half space to $\R^\dmn$, then
\begin{align*}
\D^{\mat A}\arr f &= \1_- F - \Pi^L(\1_- \mat A\nabla^m F) \quad\text{if }\Tr_{m-1}^- F=\arr f
\end{align*}
and so
\begin{equation*}\D^{\mat A}\arr f(x,-t) = -\D^{\mat A^\rho}\arr f^\rho(x,t)\end{equation*}
and so it suffices to prove the bound \eqref{A:eqn:D:square:rough:intro} in the upper half-space as well.
\end{rmk}

\subsection{Regularity of solutions to elliptic equations}
\label{A:sec:Caccioppoli}

We will often use the following higher order analogue to the Caccioppoli inequality. It was proven in full generality in \cite{Bar16} and some important preliminary versions were established in \cite{Cam80,AusQ00}.
\begin{lem}[(The Caccioppoli inequality)]\label{A:lem:Caccioppoli}
Suppose that $L$ is an operator of the form~\eqref{A:eqn:L} of order~$2m$, $m\geq 1$, acting on functions defined in $\R^\dmn$, $\dmnMinusOne\geq 1$, associated to coefficients~$\mat A$ that satisfy the bounds \eqref{A:eqn:elliptic} and~\eqref{A:eqn:elliptic:bounded}.
Let $ u\in \dot W^2_m(B(X,2r))$ with $L u=0$ in $B(X,2r)$.
Then we have the bound
\begin{equation*}
\int_{B(X,r)} \abs{\nabla^j u(x,s)}^2\,dx\,ds
\leq \frac{C}{r^2}\int_{B(X,2r)} \abs{\nabla^{j-1} u(x,s)}^2\,dx\,ds
\end{equation*}
for any $j$ with $1\leq j\leq m$.
\end{lem}

We will also use the following reverse H\"older estimate for gradients and H\"older continuity of solutions to equations of high order. The statement for gradients may be found in \cite{Cam80,AusQ00,Bar16}; the local boundedness result is a straightforward consequence of Morrey's inequality.

\begin{thm}
\label{A:thm:Meyers}
Suppose that $L$ is an operator of the form~\eqref{A:eqn:L} of order~$2m$, $m\geq 1$, acting on functions defined in $\R^\dmn$, $\dmnMinusOne\geq 1$, associated to coefficients~$\mat A$ that satisfy the bounds \eqref{A:eqn:elliptic} and~\eqref{A:eqn:elliptic:bounded}.
Then there is some number $p^+=p_0^+=p^+_L>2$ depending only on $m$, the dimension $\dmn$ and the constants $\lambda$ and~$\Lambda$ in the bounds \eqref{A:eqn:elliptic} and~\eqref{A:eqn:elliptic:bounded}
such that the following statement is true.

Let $X_0\in\R^\dmn$ and let $r>0$. Suppose that $ u\in \dot W^2_m(B(X_0,2r))$ and that
$L u=0$
in $B(X_0,2r)$.
Suppose that $0<p<\infty$ and $0<q<p^+$. Then
\begin{align}
\label{A:eqn:Meyers}
\biggl(\int_{B(X_0,r)
}\abs{\nabla^m u}^q\biggr)^{1/q}
&\leq
	\frac{C(p,q)}{r^{\pdmn/p-\pdmn/q}}
	\biggl(\int_{B(X_0,2r)}\abs{\nabla^m u}^p\biggr)^{1/p}
\end{align}
for some constant $C(p,q)$ depending only on $p$, $q$ and the standard parameters.

We may also bound the lower-order derivatives. Suppose that $m-\pdmn/2 < m-k < m$ and that $0\leq m-k$.
Let $0<p<\infty$ and $0< q<p^+_k$, where
$p^+_k=p^+_L\,\pdmn/(\dmn-k\,p^+_L)$ if $\dmn>k\,p^+_L$ and $p^+_k=\infty$ if $\dmn\leq k\,p^+_L$.
Then
\begin{align}\label{A:eqn:Meyers:lower}
\biggl(\int_{B(X_0,r)
}\abs{\nabla^{m-k} u}^{q}\biggr)^{1/{q}}
&\leq
	\frac{C(p,q)}{r^{\pdmn/p-\pdmn/{q}}}
	\biggl(\int_{B(X_0,2r)}\abs{\nabla^{m-k} u}^{p}\biggr)^{1/p}
.\end{align}

Finally, suppose $m\geq \pdmn/2$. If
$0\leq m-k \leq m-\pdmn/2$, then $\nabla^{m-k} u$ is H\"older continuous and satisfies the bound
\begin{align}\label{A:eqn:Meyers:lowest}
\sup_{B(X_0,r)} \abs{\nabla^{m-k} u}
&\leq
	\frac{C(p)}{r^{\pdmn/p}}
	\biggl(\int_{B(X_0,2r)}\abs{\nabla^{m-k} u}^p\biggr)^{1/p}
\end{align}
provided that $0<p\leq \infty$.
\end{thm}

In particular, if $\dmn=2$, then $\nabla^{m-1}u$ is H\"older continuous. If $\mat A$ is $t$-independent and $\dmn=3$, then $\nabla^{m-1}u$ is still H\"older continuous; we will establish this fact as part of Lemma~\ref{A:lem:atomic} using an argument from \cite{AlfAAHK11}.

We compare this result to the well known De Giorgi-Nash-Moser estimates. If $m=1$ then $u=\nabla^{m-1}u$; De Giorgi, Nash and Moser showed that if $\mat A$ has real-valued coefficients and $m=1$ then the bound~\eqref{A:eqn:Meyers:lowest} (and H\"older continuity of $u=\nabla^{m-1}u$) is valid for $k=1$ irrespective of dimension. In the higher order case, no general comparable bound is known. (Even in the second order case, the De Giorgi-Nash-Moser result is not valid in arbitrary dimensions if $\mat A$ is allowed to be complex; see \cite{Fre08}.) The pointwise local bound \eqref{A:eqn:Meyers:lowest} was established in \cite{Bar16} using Morrey's inequality; this method yields a pointwise bound only for $k\geq \pdmn/2$, and in particular only if $m$ is large enough that $m-k\geq 0$ for some $k\geq \pdmn/2$.

\subsection{Reduction to operators of higher order}

It is often convenient to prove results in the case $2m\geq \dmn$ (in which case, by Theorem~\ref{A:thm:Meyers}, solutions to elliptic equations satisfy pointwise estimates). The following formulas are often useful in passing to the general case; that is, these formulas let us use results valid for the operator $\widetilde L$ of order $2\widetilde m$, $2\widetilde m\geq \dmn$, to establish results valid for the operator $L$ of order~$2m$, $2m\leq\dmnMinusOne$.

Choose some large number $M$. There are constants $\kappa_\zeta$ such that
\begin{equation*}\Delta^M = \sum_{\abs{\zeta}=M} \kappa_\zeta\, \partial^{2\zeta}.\end{equation*}
In fact, $\kappa_\zeta = M!/\zeta!$, where $\zeta!=\zeta_1!\,\zeta_2!\cdots\zeta_\dmn!$, and so we have that $\kappa_\zeta\geq 1$ for all $\abs{\zeta}=M$.

Define the differential operator $\widetilde L=\Delta^M \! L\Delta^M$; that is, $\langle \varphi, \widetilde L \psi\rangle = \langle \Delta^M \varphi, L \Delta^M\psi\rangle$ for all nice test functions $\varphi$ and~$\psi$. We remark that $\widetilde L$ is associated to coefficients~$\widetilde{\mat A}$ that satisfy
\begin{equation}\label{A:eqn:coefficients:high}
\widetilde A_{\delta\varepsilon}(x)
= \sum_{\substack{\alpha+2\zeta=\delta\\\beta+2\xi=\varepsilon}} \kappa_\zeta\,\kappa_\xi\, A_{\alpha\beta}(x)
=
	\sum_{\substack{\abs{\zeta}=M,\>2\zeta\leq \delta \\\abs{\xi}=M,\>2\xi\leq \varepsilon}} \kappa_\zeta\,\kappa_\xi\, A_{(\delta-2\zeta)(\varepsilon-2\xi)}(x)
\end{equation}
for all $\abs\delta=\abs\varepsilon=m+2M$, where $2\zeta\leq \delta $ if $2\zeta_j\leq \delta_j$ for all $1\leq j\leq \dmn$.
Observe that $\widetilde{\mat A}$ is $t$-independent and satisfies the bounds \eqref{A:eqn:elliptic} and~\eqref{A:eqn:elliptic:bounded}. It was shown in the proof of \cite[Theorem~62]{Bar16} that
\begin{equation}
\label{A:eqn:fundamental:high}
E^L(x,t,y,s) = \sum_{\abs{\zeta}=\abs{\xi}=M} \kappa_\zeta \, \kappa_\xi \, \partial^{2\zeta}_{x,t} \partial^{2\xi}_{y,s} E^{\widetilde L}(x,t,y,s).\end{equation}
Let $\arr g$ be an array of functions defined on~$\R^n$ and indexed by multiindices of length~$m$. Let
\begin{equation*}\arr {\widetilde g} = \sum_{\abs\gamma=m-1}\sum_{\abs\xi=M} \kappa_\xi \,g_\gamma(x)\arr e_{\gamma+2\xi}.\end{equation*}
Notice that $\abs{\arr{\widetilde g}(x)}\leq C \abs{\arr g(x)}$. Then, by \cite[formula~(11.2)]{BarHM17}, if $\abs\alpha=m$ then
\begin{align}\label{A:eqn:S:high}
\partial^\alpha \s^L \arr g(x,t)
&=
	\sum_{\abs{\zeta}=M}
	\kappa_\zeta \partial^{\alpha+2\zeta} \s^{\widetilde{L}} \arr{ \widetilde{g}}(x,t)
.\end{align}
Similarly, let
\begin{equation*}\arr {\widetilde h} = \sum_{\abs\alpha=m}\sum_{\abs\xi=M} \kappa_\xi \,h_\alpha(x)\arr e_{\alpha+2\xi}.\end{equation*} If $\abs\xi\in \{m-1,m\}$, then by formulas~\eqref{A:eqn:S:variant} and~\eqref{A:eqn:fundamental:high},
\begin{align}\label{A:eqn:S:variant:high}
\partial^\xi \s^L_\nabla \arr h(x,t)
&=
	\sum_{\abs{\zeta}=M}
	\kappa_\zeta \partial^{\xi+2\zeta} \s^{\widetilde{L}}_\nabla \arr{\widetilde{h}}(x,t)
.\end{align}

\subsection{Square function bounds on operators}

Our ultimate goal is to show that the gradients of the double and modified single layer potentials represent bounded operators, that is, that they satisfy the square-function estimates of Theorem~\ref{A:thm:square:rough}. There exist many known results that imply square-function or Carleson measure estimates on families of operators; in this section, we recall a few such results.

If $Q\subset\R^n$ is a cube and $r>0$, we let $rQ$ denote the concentric cube of volume $r^n\abs{Q}$.
We let
\begin{equation}\label{A:eqn:annuli}
A_0(Q) = 2Q, \qquad
A_j(Q) = 2^{j+1}Q\setminus 2^{j}Q\quad\text{for all $j\geq 1$}
.\end{equation}

\begin{lem}[{\cite[Lemma~3.5]{AlfAAHK11}}]
\label{A:lem:AAAHK}
(i) Suppose that $\{R_t\}_{t\in\R}$ is a family of operators satisfying the decay estimate
\begin{align}\label{A:eqn:AAAHK:decay}
\doublebar{R_t (f\1_{A_k(Q)})}_{L^2(Q)}^2
&\leq
C 2^{-k(n+4)}\doublebar{f}_{L^2({A_k(Q)})}^2
\end{align}
for all $f\in L^2$, all cubes $Q\subset\R^n$, all integers $k\geq 0$ and all $t$ with $\ell(Q)\leq t\leq 2\ell(Q)$.

Suppose also that
$R_t1=0$ for all $t\in\R$. (Our hypotheses allow $R_t1$ to be defined as a locally integrable function.) Then
\begin{equation*}\int_{\R^n} \abs{R_t h(x)}^2\,dx \leq C t^2 \int_{\R^n} \abs{\nabla_\pureH h(x)}^2\,dx\end{equation*}
for all $h\in \dot W^2_1(\R^n)$.

(ii) If in addition $\doublebar{R_t(\nabla_\pureH f)}_{L^2(\R^n)}\leq (C/t)\doublebar{f}_{L^2(\R^n)}$, then
\begin{equation*}\int_0^\infty \doublebar{R_t f}_{L^2(\R^n)}^2\,\frac{dt}{t}\leq C\doublebar{f}_{L^2(\R^n)}^2.\end{equation*}
\end{lem}

Notice that the decay estimate \eqref{A:eqn:AAAHK:decay}, if valid for all cubes~$Q$ and all $k\geq 0$, implies that $R_t$ is bounded on $L^2(\R^n)$, uniformly in~$t$. We observe that the estimate given above is simpler than that originally stated in \cite{AlfAAHK11}.

There is a very long history of results relating square-function estimates to Carleson measure estimates. For our purposes, the following result suffices.

\begin{lem}[{\cite[Lemma~9.1]{BarHM17}}]\label{A:lem:square:carleson}
Let $\{\Theta_t\}_{t>0}$ be a family of linear operators that satisfy
\begin{equation}\label{A:eqn:square:carleson}\int_0^\infty \int_{\R^n}\abs{\Theta_t\arr f(x)}^2\,\frac{1}{t}\,dx\,dt
\leq C_0\doublebar{\arr f}_{L^2(\R^n)}^2\end{equation}
for some constant $C_0$ and for all $\arr f\in L^2(\R^n)$.
Suppose that there exists some $\varepsilon>0$ and $C_1>0$ such that
\begin{equation}
\label{A:eqn:Carleson:decay}
\sup_{\ell(Q)\leq t\leq 2\ell(Q)}
\int_Q\abs{\Theta_t(\1_{A_j(Q)}\arr g)(x)}^2\,dx
\leq C_1 2^{-j(n+\varepsilon)}\int_{A_j(Q)} \abs{\arr g(x)}^2\,dx
\end{equation}
whenever $Q\subset\R^n$ is a cube and $j\geq 1$ is an integer.
Then there is some constant $C$ depending only on $C_0$, $C_1$ and the dimension $\dmn$ such that the Carleson measure estimate
\begin{equation}\label{A:eqn:Carleson:AAAHK11}\sup_{Q\subset\R^n} \frac{1}{\abs{Q}}\int_Q \int_0^{\ell(Q)} \abs{\Theta_t\arr b(x)}^2\,\frac{1}{t}\,dx\,dt
\leq C\doublebar{\arr b}_{L^\infty(\R^n)}^2\end{equation}
is valid.
\end{lem}

\begin{rmk}\label{A:rmk:interpolation:carleson}
Suppose that $\Theta_t$ satisfies the estimates of Lemma~\ref{A:lem:square:carleson}. Notice that we may write the square-function estimate~\eqref{A:eqn:square:carleson} as $\doublebar{\mathcal{A}_2^+(t^{-1}\Theta_t \arr f)}_{L^2(\R^n)}\leq C\doublebar{\arr f}_{L^2(\R^n)}$.
By tent space interpolation \cite[Section~7]{CoiMS85}, $\Theta_t$ also satisfies the estimate
\begin{equation*}\doublebar{\mathcal{A}_2^+(\Theta_t \arr f)}_{L^p(\R^n)}\leq C(p)\doublebar{\arr f}_{L^p(\R^n)}\end{equation*}
for any $2\leq p<\infty$.
\end{rmk}

Here the Lusin area integral $\mathcal{A}_2^+$ is given by
\begin{equation}\label{A:eqn:lusin}
\mathcal{A}_2^\pm H(x)
= \biggl( \int_{\Gamma_\pm(x)} \abs{H(y,t)}^2\frac{dy\,dt}{t^\dmn}\biggr)^{1/2},
\end{equation}
where
\begin{align*}\Gamma_+(x)&=\{(y,t):0<t<\infty, \>\abs{x-y}<t\},
\\
\Gamma_-(x)&=\{(y,t):-\infty<t<0, \>\abs{x-y}<\abs{t}\}.\end{align*}

\subsection{Regularity along horizontal slices}

In Section~\ref{A:sec:Caccioppoli} we reviewed results showing that solutions to elliptic equations are regular in that their gradients are locally in $L^p(\R^\dmn)$ for some $p>2$.

Solutions to elliptic equations with $t$-independent coefficients display further regularity; specifically, their gradients are locally in $L^p(\R^n\times\{t\})$ for any $t\in\R$.

The following lemma was proven in the case $m=1$ in \cite[Proposition 2.1]{AlfAAHK11} and generalized to the case $m\geq 2$, $p=2$ in \cite[Lemma~3.2]{BarHM17}.
\begin{lem}\label{A:lem:slices}
Let $t$ be a constant, and let $Q\subset\R^n$ be a cube.
Suppose that $\partial_s \arr u(x,s)$ satisfies the Caccioppoli-like inequality
\begin{equation*}
\biggl(\int_{B(X,r)} \abs{\partial_s \arr u(x,s)}^p\,dx\,ds\biggr)^{1/p}
\leq \frac{c_0}{r}\biggl(\int_{B(X,2r)} \abs{\arr u(x,s)}^p\,dx\,ds\biggr)^{1/p}
\end{equation*}
whenever $B(X,2r)\subset\{(x,s):x\in 2Q, t-\ell(Q)<s<t+\ell(Q)\}$, for some $1\leq p\leq\infty$.
Then
\begin{equation*}\biggl(\int_Q \abs{\arr u(x,t)}^p\,dx\biggr)^{1/p} \leq C(p,c_0)\,\ell(Q)^{-1/p}
\biggl(
\int_{2Q}\int_{t-\ell(Q)}^{t+\ell(Q)} \abs{\arr u(x,s)}^p\,ds\,dx\biggr)^{1/p}.\end{equation*}

In particular, if $u\in \dot W^2_m(2Q\times(t-\ell(Q),t+\ell(Q)))$ and $Lu=0$ in $2Q\times(t-\ell(Q),t+\ell(Q))$, and $L$ is as in Theorem~\ref{A:thm:Meyers}, then
\begin{equation*}\int_Q \abs{\nabla^{m-j} \partial_t^k u(x,t)}^p\,dx \leq \frac{C(p)}{\ell(Q)}
\int_{2Q}\int_{t-\ell(Q)}^{t+\ell(Q)} \abs{\nabla^{m-j} \partial_s^k u(x,s)}^p\,ds\,dx\end{equation*}
for any $0\leq j\leq m$, any $0< p < p_j^+$, and any integer $k\geq 0$, where $p_j^+$ is as in Theorem~\ref{A:thm:Meyers}.
\end{lem}

\begin{proof}
Begin by observing that
\begin{align*}
\biggl(\int_Q \abs{\arr u(x,t)}^p\,dx \biggr)^{1/p}
&\leq
	\biggl(\int_Q \abs[bigg]{\arr u(x,t)-\fint_{t}^{t+\ell(Q)/2} \arr u(x,s)\,ds}^p\,dx \biggr)^{1/p}
	\\&\qquad+
	\biggl(\int_Q \fint_{t}^{t+\ell(Q)/2} \abs{\arr u(x,s)}^p\,ds\,dx \biggr)^{1/p}
.\end{align*}
But
\begin{align*}
\int_Q \abs[bigg]{\arr u(x,t)-\fint_{t}^{t+\ell(Q)/2} \arr u(x,s)\,ds}^p\,dx
&\leq
	\int_Q \abs[bigg]{\fint_{0}^{\ell(Q)/2}
	\int_0^s \partial_r \arr u(x,t+r)\,dr\,ds}^p\,dx
\\&\leq
	\biggl(\frac{\ell(Q)}{2}\biggr)^{p-1}
	\int_Q \int_{0}^{\ell(Q)/2}
	\abs{\partial_r \arr u(x,t+r)}^p\,dr\,dx
.\end{align*}
Applying the Caccioppoli inequality completes the proof.
\end{proof}

\section{Vertical derivatives of \texorpdfstring{$\s^L$}{the single layer potential}: area integral estimates for \texorpdfstring{$p>2$}{p>2}}

\label{A:sec:S:vertical}

In this section we establish some Carleson measure estimates on certain derivatives of the single layer potential; these are at present the best known estimates on layer potentials with $L^\infty$ inputs. These estimates will be used in Section~\ref{A:sec:S:square:variant} to establish the bound~\eqref{A:eqn:S:square:variant}.

\begin{lem}\label{A:lem:lusin:S}
Suppose that $L$ is an operator of the form~\eqref{A:eqn:L} of order~$2m$, $m\geq 1$, acting on functions defined in $\R^\dmn$, $\dmnMinusOne\geq 1$, associated to coefficients~$\mat A$ that are $t$-independent in the sense of formula~\eqref{A:eqn:t-independent} and satisfy the bounds \eqref{A:eqn:elliptic} and~\eqref{A:eqn:elliptic:bounded}.

Suppose that $k$ is an integer with $k\geq \pdmn/2$.
Then the Carleson measure estimate
\begin{equation}
\label{A:eqn:S:Carleson}\sup_{Q\subset\R^n} \frac{1}{\abs{Q}}\int_Q \int_0^{\ell(Q)} \abs{t^k\partial_t^{m+k}\s^L\arr b(x,t)}^2\,\frac{1}{t}\,dx\,dt
\leq C\doublebar{\arr b}_{L^\infty(\R^n)}^2\end{equation}
is valid, where the supremum is taken over cubes~$Q\subset\R^n$. By the Caccioppoli inequality, the estimate
\begin{equation}\label{A:eqn:S:Carleson:full}\sup_{Q\subset\R^n} \frac{1}{\abs{Q}}\int_Q \int_0^{\ell(Q)} \abs{t^{k+m} \nabla^m\partial_t^{m+k}\s^L\arr b(x,t)}^2\,\frac{1}{t}\,dx\,dt
\leq C\doublebar{\arr b}_{L^\infty(\R^n)}^2\end{equation}
is also valid. Corresponding estimates are valid in the lower half-space.

Furthermore, the bound
\begin{equation}
\label{A:eqn:lusin:p}\doublebar{\mathcal{A}_2^\pm (t^k\,\partial_t^{m+k} \s^L \arr g)}_{L^q(\R^n)}\leq C(q) \doublebar{\arr g}_{L^q(\R^n)}
\end{equation}
is valid for all $2\leq q<\infty$, where $\mathcal{A}_2^\pm$ is as in formula~\eqref{A:eqn:lusin}.
\end{lem}

\begin{proof} By Remark~\ref{A:rmk:lower}, it suffices to work in the upper half-space.
We will use Lemma~\ref{A:lem:square:carleson} and Remark~\ref{A:rmk:interpolation:carleson}.
Let $\Theta_t\arr g(x) = t^{k}\partial_t^{m+k}\s^L \arr g(x,t)$.
We claim that if $k\geq 1$ then $\Theta_t$ satisfies the square function estimate~\eqref{A:eqn:square:carleson}. To see this, observe that
\begin{equation*}\int_{\R^n}\int_0^\infty \abs{\Theta_t\arr g(x)}^2\,\frac{1}{t}\,dx\,dt
=\int_{\R^n}\int_0^\infty t^{2k-1}\abs{\partial_t^{m+k}\s^L \arr g(x,t)}^2\,dx\,dt.\end{equation*}
If $k=1$ then by the bound~\eqref{A:eqn:S:square} we are done. Suppose $k\geq 2$.
Let $\mathcal{G}$ be a grid of Whitney cubes in $\R^\dmn_+$; that is, $\R^\dmn_+=\cup_{Q\in\mathcal{G}} \overline{Q}$, the cubes in $\mathcal{G}$ are pairwise-disjoint, and if $Q\in\mathcal{G}$ then $\ell(Q)=\dist(Q,\partial\R^\dmn_+)$. We then have that
\begin{equation*}\int_0^\infty \int_{\R^n}\abs{\Theta_t\arr g(x)}^2\,\frac{1}{t}\,dx\,dt
\leq
\sum_{Q\in \mathcal{G}}
(2\ell(Q))^{2k-1}
\int_{Q}\abs{\partial_t^{m+k}\s^L \arr g(x,t)}^2\,dx\,dt.\end{equation*}
If $k\geq 2$, then by the Caccioppoli inequality (Lemma~\ref{A:lem:Caccioppoli}) and a covering argument,
\begin{equation*}\int_0^\infty \int_{\R^n}\abs{\Theta_t\arr g(x)}^2\,\frac{1}{t}\,dx\,dt
\leq
\sum_{Q\in \mathcal{G}}
C\ell(Q)
\int_{2Q}\abs{\partial_t^{m+1}\s^L \arr g(x,t)}^2\,dx\,dt.\end{equation*}
But if $x\in\R^\dmn_+$, then $x\in 2Q$ for at most $C$ cubes $Q\in\mathcal{G}$, and so
\begin{equation*}\int_{\R^n}\int_0^\infty \abs{\Theta_t\arr g(x)}^2\,\frac{1}{t}\,dx\,dt
\leq
C\int_0^\infty\int_{\R^n}
\abs{\partial_t^{m+1}\s^L \arr g(x,t)}^2\,t\,dx\,dt.\end{equation*}
By the bound~\eqref{A:eqn:S:square}, $\Theta_t$ satisfies the square-function estimate~\eqref{A:eqn:square:carleson}.

By formula~\eqref{A:eqn:S:fundamental} for the single layer potential, if $Q\subset\R^n$ is a cube and $A_j(Q)$ is as in Lemma~\ref{A:lem:square:carleson}, then
\begin{multline*}\int_Q \abs{\partial_t^{m+k}\s^L(\1_{A_j(Q)}\arr g)(x,t)}^2\,dx
\\\begin{aligned}
&=
	\int_{Q} \abs[bigg]{\sum_{\abs\gamma=m-1}\int_{A_j(Q)} \partial_t^{m+k} \partial_{y,s}^\gamma E^L(x,t,y,0)\,g_\gamma(y)\,dy}^2\,dx
\\&\leq
	\doublebar{\arr g}_{L^2(A_j(Q))}^2
	\int_{Q}
	\int_{A_j(Q)} \abs{\partial_t^{m+k}\nabla_{y,s}^{m-1} E^L(x,t,y,0)}^2\,dy\,dx.
\end{aligned}\end{multline*}
By Lemma~\ref{A:lem:slices}, if $\ell(Q)\leq t\leq 2 \ell(Q)$
then
\begin{multline*}\int_Q \abs{\partial_t^{m+k}\s^L(\1_{A_j(Q)}\arr g)(x)}^2\,dx
\\\leq
	\frac{C}{2^j t^2}\doublebar{\arr g}_{L^2(A_j(Q))}^2
	\int_{\widetilde A_{j,1}(Q)} \int_{2\widetilde Q} \abs{\partial_r^{m+k} \nabla^{m-1}_{y,s} E^L(x,r,y,s)}^2\,dx\,dr\,dy\,ds
\end{multline*}
where
$2\widetilde Q = 2Q\times (t/2,2t)$ and \begin{equation*} \widetilde A_{j,1}(Q) = (A_{j-1}(Q)\cup A_j(Q)\cup A_{j+1}(Q)) \times (-2^{j-4}\ell(Q),2^{j-4}\ell(Q)).\end{equation*}
Notice that $\diam 2\widetilde Q\approx t\approx \ell(Q)$, $\diam \widetilde A_{j,1}(Q)\approx 2^j\ell(Q)$ and $\dist (2\widetilde Q,\widetilde A_{j,1}(Q)) \approx 2^j\ell(Q)$. (Here $A\approx B$ if $A\leq CB$ and $B\leq CA$, for some $C$ depending only on the dimension~$\dmn$.)

Applying the bound \eqref{A:eqn:fundamental:far} and the Caccioppoli inequality, we have that
\begin{equation}\label{A:eqn:S:decay:high}\int_Q \abs{\partial_t^{m+k}\s^L(\1_{A_j(Q)}\arr g)(x)}^2\,dx
\leq
	\frac{C}{ t^{2k}}\doublebar{\arr g}_{L^2(A_j(Q))}^2
	2^{j-2kj-j\varepsilon}
.\end{equation}
Thus,
$\Theta_t=t^k\partial_t^{m+k}\s^L$ satisfies the bound \eqref{A:eqn:Carleson:decay} for any
$k\geq \pdmn/2$.
By Lemma~\ref{A:lem:square:carleson}, $\Theta_t$ satisfies the Carleson measure estimate~\eqref{A:eqn:Carleson:AAAHK11}, and so $\s^L$ satisfies the bound~\eqref{A:eqn:S:Carleson}. We derive the bound~\eqref{A:eqn:S:Carleson:full} using the Caccioppoli inequality in Whitney cubes as before. By Remark~\ref{A:rmk:interpolation:carleson}, the bound \eqref{A:eqn:lusin:p} is valid for $2<q<\infty$.
\end{proof}

\begin{rmk} An important part of the proof of Lemma~\ref{A:lem:lusin:S} was the the decay estimate~\eqref{A:eqn:S:decay:high}.
By the same argument, and under the same assumptions on~$\mat A$, we have the decay estimate on the modified single layer potential
\begin{equation}
\label{A:eqn:S:variant:decay}
\int_Q \abs{t^{k+1} \partial_t^{m+k}\s^L_\nabla(\arr h\1_{A_j(Q)})(x,t)}^2
\,dx
\leq
	C2^{-j-2kj-j\varepsilon}
	\int_{A_j(Q)}\abs{\arr h}^2
\end{equation}
for any $k\geq 0$, any $j\geq 0$ and any $\ell(Q)\leq t\leq 2\ell(Q)$.
\end{rmk}

\section{\texorpdfstring{$\s^L_\nabla$}{The modified single layer potential}: square-function estimates}

\label{A:sec:S:square:variant}

Recall from \cite{BarHM17} (the bounds \eqref{A:eqn:D:square} and \eqref{A:eqn:S:square} above) that the double and single layer potentials satisfy square-function estimates. We would like to prove the following analogous bound for the modified single layer potential.

\begin{thm}\label{A:thm:S:square:variant}
Suppose that $L$ is an operator of the form~\eqref{A:eqn:L} of order~$2m$, $m\geq 1$, acting on functions defined in $\R^\dmn$, $\dmnMinusOne\geq 1$, associated to coefficients~$\mat A$ that are $t$-independent in the sense of formula~\eqref{A:eqn:t-independent} and satisfy the bounds \eqref{A:eqn:elliptic} and~\eqref{A:eqn:elliptic:bounded}.

Then the modified single layer potential $\s^L_\nabla$ satisfies the square-function estimate
\begin{align}\label{A:eqn:S:square:variant}
\int_{\R^n}\int_{-\infty}^\infty \abs{\nabla^m \s^L_\nabla \arr h(x,t)}^2\,\abs{t}\,dt\,dx
	& \leq C \doublebar{\arr h}_{L^2(\R^n)}^2
\end{align}
for all $\arr h\in L^2(\R^n)$.
\end{thm}

The remainder of this section will be devoted to a proof of this theorem. We remark that many of the ideas of this section were inspired by the proof of \cite[Lemma~3.1]{HofMayMou15}; this lemma is the $m=1$ case of Theorem~\ref{A:thm:S:square:variant}.
By Remark~\ref{A:rmk:lower}, it suffices to work in the upper half-space.

We begin by reducing to a special case in three different ways.

First, suppose $2m\leq n$.
Let $\widetilde {\mat A}$ and $\kappa_\xi$ be as in formula~\eqref{A:eqn:coefficients:high}.
By formula~\eqref{A:eqn:S:variant:high}, if the bound \eqref{A:eqn:S:square:variant} is valid for $\s^{\widetilde{L}}_\nabla$, then it is valid for $\s^L_\nabla$ as well.
Thus, it suffices to prove Theorem~\ref{A:thm:S:square:variant} in the case $2m\geq \dmn$.

Our second reduction will require some notation.

Let $\mat A_\pureH$ be the ``purely horizontal'' component of~$\mat A$, so that
\begin{equation*}(\mat A_\pureH \arr h)_\alpha = \sum_{\abs\beta=m,\>\beta_\dmn=0} A_{\alpha\beta} h_\beta\text{ if }\alpha_\dmn=0 \text{ and }(\mat A_\pureH \arr h)_\alpha=0 \text{ otherwise.}\end{equation*}
We may regard $\mat A_\pureH$ as a square matrix or as a rectangular matrix with many entries equal to zero, depending on the context. We may also identify $\mat A_\pureH$ with a matrix with entries indexed by multiindices in $(\N_0)^n$.

\begin{lem}\label{A:lem:Hodge}
Let $\mat A$ and $L$ be as in Theorem~\ref{A:thm:S:square:variant}. If there is a $C_0$ such that the bound
\begin{align*}
\int_{\R^n}\int_{-\infty}^\infty \abs{\nabla^m \s^L_\nabla (\mat A_\pureH \nabla_\pureH^m F)(x,t)}^2\,\abs{t}\,dt\,dx
	& \leq C_0 \doublebar{\nabla_\pureH^m F}_{L^2(\R^n)}^2= C_0 \doublebar{F}_{\dot W_m^2(\R^n)}^2
\end{align*}
is valid for all $F\in\dot W_m^2(\R^n)$, then the bound~\eqref{A:eqn:S:square:variant} is valid
for all $\arr h\in L^2(\R^n)$.
\end{lem}

\begin{proof}
By formula~\eqref{A:eqn:S:S:vertical},
if $\alpha=\delta+\vec e_\dmn$, then
\begin{equation*}\s^L_\nabla (h\arr e_\alpha)(x,t)=-\partial_t \s^L (h\arr e_\delta)(x,t).\end{equation*}
Thus, if $\alpha_\dmn>0$, then the estimate \eqref{A:eqn:S:square:variant} for $\arr h=h\arr e_\alpha$ follows from the square-function bound \eqref{A:eqn:S:square} on~$\s^L$.
Thus, to establish the bound~\eqref{A:eqn:S:square:variant}, we may assume $h_\alpha\neq 0$ only for $\alpha_\dmn=0$.

Choose some such~$\arr h$.
An elementary consequence of the ellipticity condition~\eqref{A:eqn:elliptic} is that
\begin{equation*}\lambda\doublebar{\nabla_\pureH^m\varphi}_{L^2(\R^n)}^2\leq \re\langle \nabla^m_\pureH \varphi,\mat A_\pureH\nabla^m_\pureH\varphi\rangle_{\R^n}
\quad\text{for all $\varphi\in \dot W^2_m(\R^n)$.}\end{equation*}
We use the Hodge decomposition to write
\begin{equation*}\arr h=\mat A_\pureH\nabla_\pureH^m F+\arr \gamma\end{equation*}
for some $F\in \dot W^2_m(\R^n)$ and some $\arr \gamma\in L^2(\R^n)$ with
\begin{equation*}\doublebar{\arr\gamma}_{L^2(\R^n)}+\doublebar{F}_{\dot W^2_m(\R^n)}\leq C\doublebar{\arr h}_{L^2(\R^n)}\end{equation*}
and for which $\Div_m\arr \gamma=0$ in the sense that
\begin{equation*}\langle\nabla^m_\pureH\varphi,\arr \gamma\rangle_{\R^n}=0 \quad\text{for all $\varphi\in \dot W^2_m(\R^n)$.}\end{equation*}
By formula~\eqref{A:eqn:S:variant}, if $\abs\zeta=m$ then
\begin{equation*}\partial^\zeta \s^L_\nabla\arr \gamma(x,t)
=\sum_{\abs\alpha=m,\>\alpha\in (\N_0)^n}\int_{\R^n} \partial_{x,t}^\zeta \partial_{y}^\alpha E^L(x,t,y,0)\,\gamma_\alpha(y)\,dy
.\end{equation*}
By Lemma~\ref{A:lem:slices} and the bound~\eqref{A:eqn:fundamental:far}, for almost every $(x,t)\in \R^\dmn_+$ we have that the function $\varphi$ given by $\varphi(y)=\partial_{x,t}^\zeta E^L(x,t,y,0)$ lies in $\dot W^2_m(\R^n)$, and so $\partial^\zeta \s^L_\nabla\arr \gamma(x,t)=\langle \nabla^m_\pureH\varphi,\arr\gamma\rangle_{\R^n}=0$.

Thus, $\nabla^m\s^L_\nabla \arr h=\nabla^m\s^L_\nabla (\mat A_\pureH\nabla^m_\pureH F)$. This completes the proof.
\end{proof}
Thus, we have reduced matters to establishing the bound
\begin{align}\label{A:eqn:S:square:0}
\int_{\R^n}\int_0^\infty \abs{\nabla^m \s^L_\nabla (\mat A_\pureH\nabla_\pureH^m F)(x,t)}^2\,{t}\,dt\,dx
	& \leq C \doublebar{\nabla_\pureH^m F}_{L^2(\R^n)}^2
\end{align}
for all $F\in \dot W^2_m(\R^n)$, under the additional assumption $2m\geq \dmn$.

Our third reduction to a special case is to show that the bound \eqref{A:eqn:S:square:variant} follows from a bound involving higher-order vertical derivatives. In proving Lemma~\ref{A:lem:square:Theta}, we will not need the assumption $2m\geq \dmn$ or $\arr h=\mat A_\pureH\nabla^m_\pureH F$; our three reductions to a special case are independent.
\begin{lem}\label{A:lem:square:Theta} Let $\mat A$ and $L$ be as in Theorem~\ref{A:thm:S:square:variant} and let $\arr h\in L^2(\R^n)$.
If $k\geq 0$, then we have the bound
\begin{equation*}
\int_{\R^\dmn_+} \abs{\nabla^m \s^L_\nabla \arr h(x,t)}^2\,t\,dx\,dt
\leq
C_k \int_{\R^\dmn_+} \abs{\partial_t^{m+k}\s^L_\nabla \arr h(x,t)}^2\,t^{2k+1}\,dx\,dt
+ C_k\doublebar{\arr h}_{L^2(\R^n)}^2
.\end{equation*}
\end{lem}

\begin{proof}
We follow the similar proof of formula~(5.5) in~\cite{AlfAAHK11}. Let
$Q\subset\R^n$ be a large cube. By Lemma~\ref{A:lem:slices} and the Caccioppoli inequality, if $k\geq 0$ and $t\leq \ell(Q)$ then
\begin{equation*}
\int_Q \abs{\nabla^m \partial_t^{k+m} \s^L_\nabla \arr h(x,t)}^2 \,dx\leq \frac{C}{t^{2m}}\fint_{t/2}^{2t} \int_{2Q} \abs{\partial_t^{k+m} \s^L_\nabla \arr h(x,t)}^2 \,dx\,ds
.\end{equation*}
Thus, we have that
\begin{multline*}\int_Q \int_0^{\ell(Q)} \abs{\nabla^m \partial_t^{k+m} \s^L_\nabla \arr h(x,t)}^2 t^{2k+2m+1}\,dx \,dt
\\\leq
C
\int_{2Q} \int_0^{2\ell(Q)} \abs{\partial_t^{k+m} \s^L_\nabla \arr h(x,t)}^2 t^{2k+1}\,dx \,dt
.\end{multline*}
We need only show that
\begin{multline*}\int_Q \int_0^{\ell(Q)} \abs{\nabla^m \s^L_\nabla \arr h(x,t)}^2 t\,dx \,dt
\\\leq
C
\int_Q \int_0^{\ell(Q)} \abs{\nabla^m \partial_t^{k+m} \s^L_\nabla \arr h(x,t)}^2 t^{2k+2m+1}\,dx \,dt
+ C\doublebar{\arr h}_{L^2(\R^n)}^2
.\end{multline*}
By the monotone convergence theorem, we may take the limit as $Q$ expands to all of $\R^n$ to complete the proof.

Define
\begin{equation*}
U_j(t)=\int_Q \abs{\nabla^m \partial_t^j \s^L_\nabla \arr h(x,t)}^2 \,dx
.\end{equation*}
Arguing as in the proof of formula~\eqref{A:eqn:S:decay:high}, we have that if $t>\ell(Q)/2$ and $j\geq 0$ then
\begin{equation*}U_j(t)\leq \frac{C}{(t+\ell(Q))^{2j+2}} \doublebar{\arr h}_{L^2(\R^n)}^2.\end{equation*}

Now,
\begin{equation*}\int_Q \int_0^{\ell(Q)} \abs{\nabla^m \s^L_\nabla \arr h(x,t)}^2 t\,dx \,dt = \int_0^{\ell(Q)} t\,U_0(t)\,dt.\end{equation*}
Suppose $j\geq 0$. Then
if $0<\varepsilon<\ell(Q)$, we have that
\begin{align*}
\int_{\varepsilon}^{\ell(Q)} t^{2j+1} \,U_j(t)\,dt
&=
\int_{\varepsilon}^{\ell(Q)} t^{2j+1}\, U_j(\ell(Q))\,dt
-\int_{\varepsilon}^{\ell(Q)} t^{2j+1}\int_t^{\ell(Q)} U_j'(s)\,ds\,dt
\\&\leq
C \doublebar{\arr h}_{L^2(\R^n)}^2
+\frac{1}{2j+2}\int_{\varepsilon}^{\ell(Q)} s^{2j+2}\, \abs{U_j'(s)}\,ds.
\end{align*}
By definition of~$U_j$ and by Young's inequality,
$\abs{U_j'(s)}
\leq \frac{1}{s}U_j(s)+s\,U_{j+1}(s)
$.
Thus,
\begin{align*}
\int_{\varepsilon}^{\ell(Q)} t^{2j+1} \,U_j(t)\,dt
&\leq
	C \doublebar{\arr h}_{L^2(\R^n)}^2
	+\frac{1}{2j+2}\int_{\varepsilon}^{\ell(Q)} s^{2j+1}\, U_j(s)\,ds
	\\&\qquad
	+\frac{1}{2j+2}\int_{\varepsilon}^{\ell(Q)} s^{2j+3}\, U_{j+1}(s)\,ds.
\end{align*}
Rearranging terms, we have that if $j\geq 0$ then
\begin{align*}
\int_{\varepsilon}^{\ell(Q)} t^{2j+1} \,U_j(t)\,dt
&\leq
	C \doublebar{\arr h}_{L^2(\R^n)}^2
	+C\int_{\varepsilon}^{\ell(Q)} s^{2j+3}\, U_{j+1}(s)\,ds
.\end{align*}
By the monotone convergence theorem, we may take the limit as $\varepsilon\to 0^+$, and so
\begin{multline*}\int_0^{\ell(Q)} \int_Q \abs{\nabla^m \partial_t^j \s^L_\nabla \arr h(x,t)}^2 \,dx \,t^{2j+1}\,dt
\\\leq
C \doublebar{\arr h}_{L^2(\R^n)}^2
	+C\int_0^{\ell(Q)} t^{2j+3}\, \int_Q \abs{\nabla^m \partial_t^{j+1} \s^L_\nabla \arr h(x,t)}^2 \,dx \,dt.\end{multline*}
By induction, we see that for any $j\geq 1$,
\begin{multline*}\int_0^{\ell(Q)} \int_Q \abs{\nabla^m \s^L_\nabla \arr h(x,t)}^2 \,dx \,t\,dt
\\\leq
C_j \doublebar{\arr h}_{L^2(\R^n)}^2
	+C_j\int_0^{\ell(Q)} t^{2j+1}\, \int_Q \abs{\nabla^m \partial_t^{j} \s^L_\nabla \arr h(x,t)}^2 \,dx \,dt
\end{multline*}
as desired.\end{proof}

Thus, we have reduced matters to establishing the bound
\begin{align}\label{A:eqn:S:square:k}
\int_{\R^n}\int_0^\infty \abs{t^k\partial_t^{m+k} \s^L_\nabla (\mat A_\pureH\nabla_\pureH^m F)(x,t)}^2\,{t}\,dt\,dx
	& \leq C \doublebar{\nabla_\pureH^m F}_{L^2(\R^n)}^2
\end{align}
for some $k\geq 0$, and all $F\in \dot W^2_m(\R^n)$, under the assumption that $2m\geq \dmn$.

We will establish the bound~\eqref{A:eqn:S:square:k} using convolution with a smooth kernel. Let $\eta$ be a Schwartz function defined on~$\R^n$ that integrates to~$1$, let $\eta_t(x)=t^{-n}\eta(x/t)$, and let $\mathcal{Q}_tf(x)=\eta_t*f(x)$. If $\arr b$ is an array of functions, then we establish the notation
\begin{equation*}\partial_\perp^{\ell}\s^L_\nabla (\arr b \mathcal{Q}_t h)(x,t)
=
\partial_t^{\ell}\s^L_\nabla (\arr b \mathcal{Q}_s h)(x,t)\big\vert_{s=t}.\end{equation*}
In other words, $\partial_\perp$ ignores the dependency of $\mathcal{Q}_t$ on~$t$.

We will prove the following lemmas.
\begin{lem}\label{A:lem:R:1}
Let $\mat A$ and $L$ be as in Theorem~\ref{A:thm:S:square:variant}.
Define
\begin{align}
\label{A:eqn:R:1}
R_t^1 \arr f(x) &= t^{k+1} \partial_t^{m+k}\s^L_\nabla(\mat A_\pureH \arr f)(x,t) -t^{k+1} \partial_\perp^{m+k}\s^L_\nabla(\mat A_\pureH \mathcal{Q}_t\arr f)(x,t)
\\\nonumber&= t^{k+1} \partial_\perp^{m+k}\s^L_\nabla(\mat A_\pureH (\arr f - \mathcal{Q}_t\arr f))(x,t)
.\end{align}
Assume that $\int \eta=1$ and that $\int x^\zeta \eta(x)\,dx=0$ whenever $1\leq \abs\zeta\leq m$.
If $F\in \dot W^2_m(\R^n)$ and if $k$ is large enough, then we have the bound
\begin{equation*}\int_{\R^n}\int_0^\infty \abs{R_t^1 \nabla_\pureH^m F(x)}^2\frac{1}{t}\,dt\,dx \leq C \doublebar{\nabla_\pureH^m F}_{L^2(\R^n)}^2\end{equation*}
where the constant $C$ depends only on $m$, $n$, $k$, $\lambda$, $\Lambda$, and the Schwartz constants of~$\eta$.
\end{lem}

\begin{lem}\label{A:lem:R:2}
Let $\mat A$ and $L$ be as in Theorem~\ref{A:thm:S:square:variant}.
Define
\begin{align}
\label{A:eqn:R:2}
R_t^2(\arr b,f)(x)
&= t^{k+1} \partial_\perp^{m+k}\s^L_\nabla(\arr b\, \mathcal{Q}_t f)(x,t)-t^{k+1}\partial_t^{m+k}\s^L_\nabla\arr b(x,t)\, \mathcal{Q}_t f(x)
.\end{align}
If $k$ is large enough, then we have the bound
\begin{equation*}\int_{\R^n}\int_0^\infty \abs{R_t^2 (\arr b,f)(x)}^2\frac{1}{t}\,dt\,dx \leq C \doublebar{\arr b}_{L^\infty}^2 \doublebar{f}_{L^2(\R^n)}^2\end{equation*}
where the constant $C$ depends only on $m$, $n$, $k$, $\lambda$, $\Lambda$ and the Schwartz constants of~$\eta$.
\end{lem}

\begin{lem}\label{A:lem:S:carleson:preliminary}
Let $\mat A$ and $L$ be as in Theorem~\ref{A:thm:S:square:variant}, and additionally assume that $2m\geq \dmn$. Assume that $\eta$ is supported in the ball $\{y\in\R^n:\abs{y}<1/2\}$.
If $\abs\beta=m$, then we have the Carleson measure estimate
\begin{equation*}\sup_Q \frac{1}{\abs{Q}} \int_Q \int_0^{\ell(Q)} \abs{t^{k+1}\partial_t^{m+k}\s^L_\nabla(\mat A_\pureH \arr e_\beta) (x,t)}^2\frac{1}{t}\,dt\,dx \leq C .
\end{equation*}
\end{lem}

Before proving these lemmas, we show how they imply the bound \eqref{A:eqn:S:square:k}. Let $\eta$ satisfy the conditions of Lemmas~\ref{A:lem:R:1} and~\ref{A:lem:S:carleson:preliminary}.
Choose some $F\in \dot W^2_m(\R^n)$. Then
\begin{align*}t^{k+1}\partial_t^{m+k}\s^L_\nabla (\mat A_\pureH \nabla_\pureH^m F)(x,t)
&=
R_t^1(\nabla_\pureH^m F)(x,t)
+\sum_{\abs\beta=m} R_t^2 (\mat A_\pureH \arr e_\beta, \partial^\beta F)(x,t)
\\&\qquad
+\sum_{\abs\beta=m}t^{k+1}\partial_t^{m+k}\s^L_\nabla(\mat A_\pureH \arr e_\beta) (x,t) \,\mathcal{Q}_t (\partial^\beta F)(x)
.\end{align*}
The first two terms satisfy square-function estimates by Lemmas~\ref{A:lem:R:1} and~\ref{A:lem:R:2}, while by Lemma~\ref{A:lem:S:carleson:preliminary}, Carleson's lemma (see, for example, \cite[Section~II.2]{Ste93}) and the fact that
\begin{equation}\label{A:eqn:NQ:carleson}\doublebar{N_+(\abs{\mathcal{Q}_t f}^2)}_{L^1(\R^n)} = \doublebar{N_+(\mathcal{Q}_t f)}_{L^2(\R^n)}^2
\leq C\doublebar{f}_{L^2(\R^n)}^2,\end{equation}
where $N_+$ is the nontangential maximal function in Carleson's lemma,
we have that the third term satisfies a square-function estimate.

\subsection{Proof of Lemma~\ref{A:lem:R:2}}

Fix $\arr b\in L^\infty(\R^n)$ and let $R_t^2 f=R_t^2(\arr b,f)$. We seek to bound $R_t^2$ using Lemma~\ref{A:lem:AAAHK}.

We begin with the bound \eqref{A:eqn:AAAHK:decay}. Let $t>0$ and let $Q\subset\R^n$ be a cube with $\ell(Q)\leq t\leq 2\ell(Q)$.

Recall the decay estimate~\eqref{A:eqn:S:variant:decay}: for such $t$ and~$Q$, and for $A_j(Q)$ as in formula~\eqref{A:eqn:annuli},
\begin{equation*}
\int_Q \abs{t^{k+1} \partial_t^{m+k}\s^L_\nabla(\arr h\1_{A_j(Q)})(x,t)}^2
\,dx
\leq
	C2^{-j[2k+1+\varepsilon]}
	\int_{A_j(Q)}\abs{\arr h}^2
.\end{equation*}
Let $K=2k+1+\varepsilon$.
Notice that this estimate is valid even if $2m\leq n$; the assumption $2m\geq \dmn$ is needed only to prove Lemma~\ref{A:lem:S:carleson:preliminary}, and not to prove the present result.

Suppose that $f_j$ is supported in $A_j(Q)$.
Recall that $\mathcal{Q}_t$ denotes convolution with $\eta_t$, and so for any integers $j\geq 0$ and $\ell\geq 0$,
\begin{align*}\doublebar{\mathcal{Q}_t f_j}_{L^\infty(A_\ell(Q))}
&\leq t^{-n} \sup_{x\in A_j(Q),y\in A_\ell(Q)} \abs[bigg]{\eta\biggl(\frac{x-y}{t}\biggr)}\doublebar{f_j}_{L^1(A_j(Q))}\\&\leq {C} t^{-n/2}2^{jn/2} \doublebar{f_j}_{L^2(A_j(Q))}.
\end{align*}
If $\dist(A_j(Q),A_\ell(Q))>0$, we can improve this estimate.
Because $\eta$ is a Schwartz function, we have that if $\ell\leq j-2$ or $\ell\geq j+2$, then for any integer $N>0$ there is a constant $C_N$ such that
\begin{align*}\doublebar{\mathcal{Q}_t f_j}_{L^\infty(A_\ell(Q))}
\leq {C_N} t^{-n/2}2^{jn/2} 2^{-N\max(j,\ell)}\doublebar{f_j}_{L^2(A_j(Q))}.
\end{align*}

Thus, if $f_j$ is supported in $A_j(Q)$, and if $k$ and $N$ are large enough, then
\begin{align}\label{A:eqn:SQ:decay}
\biggl(\int_Q \abs{t^{k+1} \partial_\perp^{m+k}\s^L_\nabla(\arr b \mathcal{Q}_t f_j)(x,t)}^2\,dx\biggr)^{1/2}
&\leq
\sum_{\ell=0}^\infty C2^{-\ell K/2} \doublebar{\arr b}_{L^\infty} \doublebar{\mathcal{Q}_t f_j}_{L^2(A_\ell(Q))}
\\\nonumber&\leq
C2^{-j (K/2-n)} \doublebar{\arr b}_{L^\infty}\doublebar{f_j}_{L^2(A_j(Q))}
.\end{align}
Furthermore,
\begin{multline*}
\int_Q \abs{t^{k+1}\partial_t^{m+k}\s^L_\nabla\arr b(x,t)\, \mathcal{Q}_t f_j(x)}^2\,dx
\\\begin{aligned}
&\leq
C_N t^{-n} 2^{j(n-2N)}\doublebar{f_j}_{L^2(A_j(Q))}^2
\int_Q \abs{t^{k+1}\partial_t^{m+k}\s^L_\nabla\arr b(x,t)}^2\,dx
\end{aligned}\end{multline*}
and if $k$ is large enough then
\begin{align}
\label{A:eqn:S:infty}
\int_Q \abs{t^{k+1}\partial_t^{m+k}\s^L_\nabla\arr b(x,t)}^2\,dx
&\leq
\int_Q \abs[Big]{\sum_{\ell=0}^\infty t^{k+1}\partial_t^{m+k}\s^L_\nabla(\1_{A_\ell(Q)}\arr b)(x,t)}^2\,dx
\\\nonumber&\leq Ct^n\doublebar{\arr b}_{L^\infty}^2
.\end{align}
Thus, if $N$ and $k$ are large enough then
\begin{equation}
\label{A:eqn:R:2:decay}
\int_Q \abs{R_t^2 f_j(x)}^2\,dx
\leq
C 2^{-j(K-2n)}\doublebar{\arr b}_{L^\infty}^2 \doublebar{f_j}_{L^2(A_j(Q))}^2
.\end{equation}
Thus, if $k$ is large enough, then $R_t^2$ satisfies the decay estimate~\eqref{A:eqn:AAAHK:decay}.

For future reference, we observe that we have the same decay estimate on $R_t^1$. Recall that
\begin{align*}
R_t^1 \arr f(x) &= t^{k+1} \partial_t^{m+k}\s^L_\nabla(\mat A_\pureH \arr f)(x,t) -t^{k+1} \partial_\perp^{m+k}\s^L_\nabla(\mat A_\pureH \mathcal{Q}_t\arr f)(x,t)
.\end{align*}
Bounding the first term by the decay estimate~\eqref{A:eqn:S:variant:decay} and the second term by the bound~\eqref{A:eqn:SQ:decay}, we have that
\begin{equation}
\label{A:eqn:R:1:decay}
\int_Q \abs{R_t^1 (\1_{A_j(Q)}\arr f(x)}^2\,dx
\leq
C 2^{-j(K-2n)}\doublebar{\mat A_\pureH}_{L^\infty}^2 \doublebar{\arr f}_{L^2(A_j(Q))}^2
.\end{equation}

We now return to $R_t^2$. As observed in Lemma~\ref{A:lem:AAAHK}, the estimate~\eqref{A:eqn:R:2:decay} (or~\eqref{A:eqn:S:infty}) means that $R_t^21$ may be defined as a locally integrable function. From the definition of $R_t^2$ we may easily see that $R_t^2 1(x)=0$ for all $x\in\R^n$.

Finally,
$\doublebar{\mathcal{Q}_t(\nabla f)}_{L^2(\R^n)} \leq (C/t) \doublebar{f}_{L^2(\R^n)}$ and so $R_t^2$ satisfies the condition~(ii) of Lemma~\ref{A:lem:AAAHK}. Thus,
\begin{equation*}
\int_0^\infty \doublebar{R_t^2 f}_{L^2(\R^n)}^2\frac{dt}{t}\leq C\doublebar{ f}_{L^2(\R^n)}^2
\end{equation*}
as desired.

\subsection{Proof of Lemma~\ref{A:lem:R:1}}
Recall that
\begin{align*}
R_t^1 \arr f(x)
&=
t^{k+1} \partial_\perp^{m+k}\s^L_\nabla(\mat A_\pureH (\arr f - \mathcal{Q}_t\arr f))(x,t)
.
\end{align*}
Thus,
\begin{multline*}
\frac{1}{t^{k+1}}R_t^1 (\nabla_\pureH^m F)(x)
\\=
\sum_{\substack{\abs\alpha=m\\\alpha_\dmn=0}} \sum_{\substack{\abs\beta=m\\\beta_\dmn=0}}
\int_{\R^n} \partial_t^{m+k} \partial_{y}^\alpha E^L(x,t,y,0)\,A_{\alpha\beta}(y)\,\partial_y^\beta (F(y)-\mathcal{Q}_tF(y))\,dy
.\end{multline*}
Integrating by parts in~$y$, applying the fact (see Theorem~\ref{A:thm:fundamental}) that $v(y,s)=E^{L}(x,t,y,s)=\overline{E^{L^*}(y,s,x,t)}$ satisfies $L^*\bar v=0$, and integrating by parts again (and using formula~\eqref{A:eqn:fundamental:vertical}) we see that
\begin{multline*}
\frac{1}{t^{k+1}}R_t^1 (\nabla_\pureH^m F)(x)
\\\begin{aligned}
&=
	-
	\sum_{\substack{\abs\alpha=m\\\mathclap{\abs\delta\leq m-1,\>\delta_\dmn=0}}}
	\int_{\R^n} \partial_t^{m+k}
	\partial_t^{m-\abs\delta}\partial_{y,s}^\alpha E^L(x,t,y,0)
	\,A_{\alpha\delta}(y)\,\partial_y^\delta(F(y)-\mathcal{Q}_tF(y))\,dy
	\\&\qquad+
	\sum_{\substack{\abs\gamma=m-1\\ \mathclap{\abs\beta=m,\>\beta_\dmn=0}}}
	\int_{\R^n}\partial_t^{m+k+1} \partial_{y,s}^\gamma E^L(x,t,y,0))\,A_{\gamma\beta}(y)\,\partial_y^\beta (F(y)-\mathcal{Q}_tF(y))\,dy
.\end{aligned}\end{multline*}
Here $A_{\zeta\xi}=A_{\tilde\zeta\tilde \xi}$, where $\tilde \zeta=\zeta+(m-\abs\zeta)\vec e_\dmn$ is a multiindex of length~$m$ whose horizontal part coincides with that of~$\zeta$.

Recalling formulas~\eqref{A:eqn:S:fundamental} and~\eqref{A:eqn:S:variant} for $\s^L$ and~$\s^L_\nabla$, we see that
\begin{align}
\label{A:eqn:R:1:S}
R_t^1 (\nabla_\pureH^m F)(x)
&=
	-
	\sum_{j=0}^{m-1}
	t^{k+1}\partial_t^{2m+k-j}
	\s^L_\nabla (\mat A_{mj} \nabla_\pureH^j (F-\mathcal{Q}_t F))(x,t)
	\\\nonumber&\qquad+
	t^{k+1}\partial_\perp^{m+k+1} \s^L(\mat A'\nabla_\pureH^mF)(x,t)
	\\\nonumber&\qquad
	-
	t^{k+1}\partial_\perp^{m+k+1} \s^L(\mat A'\nabla_\pureH^m\mathcal{Q}_tF)(x,t)
\end{align}
where $(\mat A_{mj} \arr f)_\alpha= \sum_{\substack{\abs\delta=j\\\delta_\dmn=0}} A_{\alpha\delta}\,f_\delta$ and $(\mat A' \arr f)_\zeta= \sum_{\substack{\abs\beta=m\\\beta_\dmn=0}} A_{\zeta\beta}\,f_\beta$.

By the bound~\eqref{A:eqn:S:square} and the Caccioppoli inequality,
\begin{equation*}\int_0^\infty \int_{\R^n} \abs{t^{k+1}\partial_t^{m+k+1} \s^L(\mat A'\nabla_\pureH^m F)(x,t)}^2\frac{1}{t}\,dx\,dt\leq
C\doublebar{\nabla_\pureH^m F}_{L^2(\R^n)}^2.
\end{equation*}

If $0\leq j\leq m-1$ and $k$ is large enough, then the bound \eqref{A:eqn:S:variant:decay} implies that $\arr h\mapsto t^{m+k-j+1}\partial_t^{2m+k-j}\s^L_\nabla\arr h(\,\cdot\,,t)$ is bounded on $L^2(\R^n)$, uniformly in~$t$. Thus,
\begin{multline*}\int_0^\infty \int_{\R^n} \abs{t^{k+1} \partial_t^{2m+k-j}
	\s^L_\nabla (\mat A_{mj} \nabla_\pureH^j (F-\mathcal{Q}_t F))(x,t)}^2\,\frac{1}{t}\,dx\,dt
\\\leq
\int_0^\infty \int_{\R^n} \abs{t^{j-m} (\nabla_\pureH^j F(x)-\mathcal{Q}_t \nabla_\pureH^jF(x))}^2\,\frac{1}{t}\,dx\,dt
.\end{multline*}
In the statement of Lemma~\ref{A:lem:R:1}, we assumed that the higher moments of the kernel $\eta$ of $\mathcal{Q}_t$ vanish. This implies that $\partial^\zeta \widehat \eta(0)=0$ if $1\leq\abs\zeta\leq m$, and so $\abs{\widehat \eta(\omega)-1}\leq C_\eta \abs{\omega}^{m+1}$. We also have that $\widehat \eta$ is uniformly bounded, and so
\begin{equation*}\abs{\widehat \eta(\omega)-1}\leq C_\eta \min(1,\abs{\omega}^{m+1}).\end{equation*}
By Plancherel's theorem, \begin{equation*}\int_0^\infty \int_{\R^n} \abs{t^{j-m} (\nabla_\pureH^j F(x)-\mathcal{Q}_t \nabla_\pureH^jF(x))}^2\,\frac{1}{t}\,dx\,dt\leq
C\doublebar{\nabla_\pureH^m F}_{L^2(\R^n)}^2
\end{equation*}
and so the first term on the right-hand side of formula~\eqref{A:eqn:R:1:S} satisfies a square-function estimate.

We are left with the term
$t^{k+1}\partial_\perp^{m+k+1} \s^L(\mat A'\nabla_\pureH^m\mathcal{Q}_tF)(x,t)
$.
Recall the definition \eqref{A:eqn:R:2} of $R_t^2$.
Using formula~\eqref{A:eqn:S:S:vertical}, we compute that
\begin{multline*}
t^{k+1}\partial_t^{m+k+1} \s^L(\mat A'\nabla_\pureH^m \mathcal{Q}_tF)(x,t)
\\\begin{aligned}
&=
\!\!\!\sum_{\substack{\abs\beta=m,\>\beta_\dmn=0}} \!\!\!
t^{k+1}\partial_t^{m+k+1} \s^L(\mat A'\arr e_\beta)(x,t)\, \partial_\pureH^\beta \mathcal{Q}_tF(x)
-R_t^2 (\mat A''\arr e_\beta ,\partial_\pureH^\beta F)(x)
\end{aligned}\end{multline*}
where $(\mat A''\arr e_\beta)_\alpha=A_{\alpha\beta}$ if $\alpha_\dmn\geq 1$ and $(\mat A''\arr e_\beta)_\alpha=0$ if $\alpha_\dmn=0$.
By Lemma~\ref{A:lem:R:2}, $R_t^2 (\mat A''\arr e_\beta ,\partial_\pureH^\beta F)(x)$ satisfies a square-function estimate. By the bound~\eqref{A:eqn:S:Carleson}, \begin{equation*}\abs{t^{k+1}\partial_t^{m+k+1} \s^L(\mat A'\arr e_\beta)(x,t)}^2\frac{1}{t}\,dx\,dt\end{equation*}
is a Carleson measure. By the bound~\eqref{A:eqn:NQ:carleson}, $N_+(\mathcal{Q}_t\partial_\pureH^\beta F)\in L^2(\R^n)$ with $L^2$ norm at most $C\doublebar{\nabla_\pureH^m F}_{L^2(\R^n)}$, and so by Carleson's lemma,
\begin{equation*}
\int_0^\infty
\int_{\R^n}
\abs{t^{k+1}\partial_t^{m+k+1} \s^L(\mat A'\arr e_\beta)(x,t)\, \partial_\pureH^\beta \mathcal{Q}_tF(x)}^2 \frac{1}{t}\,dx\,dt
\leq C\doublebar{\nabla_\pureH^m F}_{L^2(\R^n)}^2
.\end{equation*}
This completes the proof.

\subsection{Proof of Lemma~\ref{A:lem:S:carleson:preliminary}}
Recall that we seek to establish a Carleson measure estimate on $t^{k+1} \partial_t^{m+k} \s^L_\nabla \mat A_\pureH(x)$ given that $2m\geq \dmn$. (We may regard $\s^L_\nabla \mat A_\pureH$ as an array of functions with $(\s^L_\nabla \mat A_\pureH)_\beta=\s^L_\nabla (\mat A_\pureH\arr e_\beta)$.)

We will use some tools from the proof of the Kato conjecture, in particular from the paper \cite{AusHMT01}. The following lemma was established therein.

\begin{lem}
Let the matrix $\mat A_\pureH$ be uniformly bounded and satisfy the ellipticity condition
\begin{equation}
\label{A:eqn:kato:elliptic}
\re\langle \nabla^m_\pureH \varphi,\mat A_\pureH \nabla_\pureH^m\varphi\rangle_{\R^n} \geq \lambda\doublebar{\nabla^m_\pureH \varphi}_{L^2(\R^n)}^2
.\end{equation}
Suppose that $2m\geq n$. There is some $W$ depending only on the standard constants such that, for each cube~$Q\subset\R^n$, there exist $W$ functions $f_{Q,w}$
that satisfy the estimates
\begin{align}
\label{A:eqn:kato:sobolev}
\int_R \abs{\nabla_\pureH^m f_{Q,w}}^2 &\leq C \abs{Q} \qquad \text{for any cube $R$ with $\ell(R)=\ell(Q)$},
\\
\label{A:eqn:kato:L}
\abs{L_\pureH f_{Q,w}(x)} &\leq \frac{C}{\ell(Q)^{m}}
,\end{align}
and such that, for any array~$\arr \gamma_t$,
\begin{multline}
\label{A:eqn:kato:test}
\sup_Q \frac{1}{\abs{Q}} \int_0^{\ell(Q)} \int_Q \abs{\arr \gamma_t(x)}^2 \frac{dx\,dt}{t}
\\\leq
	C \sum_{w=1}^W \sup_{Q} \frac{1}{\abs{Q}}
 	\int_0^{\ell(Q)} \int_Q \abs{\langle\arr \gamma_t(x), A_t^Q \nabla^m_\pureH f_{Q,w}(x)\rangle}^2 \frac{dx\,dt}{t}
\end{multline}
where
$A_t^Q f(x)=\fint_{Q'} f(y)\,dy$, for $Q'\subset Q$ the unique dyadic subcube that satisfies $x\in Q'$ and $t\leq \ell(Q')<2t$.
\end{lem}

Here
\begin{equation*}L_\pureH = (-1)^m\,\,\sum_{\mathclap{\substack{\abs\alpha=\abs\beta=m\\ \alpha_\dmn=\beta_\dmn=0}}} \,\, \partial^\alpha(A_{\alpha\beta}\partial^\beta)\end{equation*}
is the elliptic operator of order $2m$ acting on functions defined on $\R^n$ (rather than on~$\R^\dmn$) associated to the coefficients~$\mat A_\pureH$.
As observed in the proof of Lemma~\ref{A:lem:Hodge}, the bound~\eqref{A:eqn:kato:elliptic} follows from the bound~\eqref{A:eqn:elliptic}.

Specifically, the bound~\eqref{A:eqn:kato:L} is the bound (2.19) in \cite{AusHMT01}. The bound~\eqref{A:eqn:kato:sobolev} follows from the bound (2.18) in~\cite{AusHMT01} (if $R=Q$) and the observation that, by Lemma~3.1 in \cite{AusHMT01} and the definition of $f_{Q,w}$ therein, $\nabla_\pureH^m f_{Q,w}=\nabla_\pureH^m f_{R,w}$ whenever $\ell(Q)=\ell(R)$. Finally, the bound~\eqref{A:eqn:kato:test} is simply Lemma~2.2 of \cite{AusHMT01}. The requirement that $2m\geq n$ is a sufficient condition (see \cite[Propositon~2.5]{AusHMT01} or \cite{Dav95,AusT98}) for $L_\parallel$ to satisfy a pointwise upper bound; this condition is assumed in the proofs of the above results.

Thus, we need only show that, for any cube~$Q\subset\R^n$,
\begin{equation*}\frac{1}{\abs{Q}}
 	\int_0^{\ell(Q)} \int_Q \abs{\langle t^{k+1} \partial_t^{m+k} \s^L_\nabla \mat A_\pureH(x,t), A_t^Q \nabla^m_\pureH f_{Q,w}(x)\rangle}^2 \frac{dx\,dt}{t}
\leq C
.\end{equation*}

Now, by formulas~\eqref{A:eqn:S:variant} and~\eqref{A:eqn:S:fundamental} for $\s^L_\nabla$ and~$\s^L$ and by formula~\eqref{A:eqn:fundamental:vertical},
\begin{equation*}\partial_t^{m+k} \s^L_\nabla (\mat A_\pureH \nabla^m_\pureH f_{Q,w})(x,t)
=-\partial_t^{k+1}\s^L(L_\pureH f_{Q,w}\arr e_{\perp})(x,t)
\end{equation*}
where $\arr e_\perp=\arr e_{(m-1)\vec e_\dmn}$ is the multiindex corresponding to purely vertical derivatives.
Thus, by the bounds~\eqref{A:eqn:S:Carleson} and~\eqref{A:eqn:kato:L}, if $k$ is large enough then
\begin{equation*}\frac{1}{\abs{Q}}
 	\int_0^{\ell(Q)} \int_Q \abs{ t^{k+1} \partial_t^{m+k} \s^L_\nabla (\mat A_\pureH\nabla^m_\pureH f_{Q,w})(x,t)}^2 \frac{dx\,dt}{t}
\leq C
.\end{equation*}

So we need only bound
\begin{equation*}
	t^{k+1}\partial_t^{m+k} \s^L_\nabla (\mat A_\pureH\nabla^m_\pureH f_{Q,w})(x,t)
	-
	t^{k+1}\partial_t^{m+k} \s^L_\nabla \mat A_\pureH(x,t)\cdot A_t^Q \nabla^m_\pureH f_{Q,w}(x)
.\end{equation*}
By formulas~\eqref{A:eqn:R:1} and \eqref{A:eqn:R:2} for $R_t^1$ and~$R_t^2$,
\begin{multline*}
R_t^1 (\nabla_\pureH^m f_{Q,w})(x)
+
\sum_{\abs\beta=m,\>\beta_\dmn=0}
R_t^2 (\mat A_\pureH \arr e_\beta,\partial_\pureH^\beta f_{Q,w})(x)
\\=
t^{k+1} \partial_t^{m+k}\s^L_\nabla(\mat A_\pureH \nabla_\pureH^m f_{Q,w})(x,t)
-t^{k+1}\partial_t^{m+k}\s^L_\nabla\mat A_\pureH(x,t)\cdot \mathcal{Q}_t\nabla_\pureH^m f_{Q,w}(x)
.\end{multline*}
Thus, we must bound $R_t^1 (\nabla_\pureH^m f_{Q,w})$, $R_t^2 (\mat A_\pureH ,\nabla_\pureH^m f_{Q,w})=
\sum_\beta
R_t^2 (\mat A_\pureH \arr e_\beta,\partial_\pureH^\beta f_{Q,w})
$ and
\begin{equation}\label{A:eqn:kato:argument}
	t^{k+1}\partial_t^{m+k}\s^L_\nabla\mat A_\pureH(x,t)\cdot (\mathcal{Q}_t\nabla_\pureH^m f_{Q,w}(x)
	- A_t^Q \nabla^m_\pureH f_{Q,w}(x) )
.\end{equation}

We have established boundedness of $R_t^1(\nabla_\pureH^m F)$ and $R_t^2(\arr b, f)$ for $\nabla_\pureH^m F\in L^2(\R^n)$, $f\in L^2(\R^n)$, rather than for $\nabla_\pureH^m F$ satisfying the bound \eqref{A:eqn:kato:sobolev}. Thus, more work must be done to contend with $R_t^1$ and~$R_t^2$.
Let $\chi$ be a smooth cutoff function that is equal to 1 on $4Q$ and supported on $8Q$. If we normalize $f_{Q,w}$ so that $\int_{8Q} \nabla_\pureH^j f_{Q,w}=0$ whenever $0\leq j\leq m-1$, then by the Poincar\'e inequality and the bound~\eqref{A:eqn:kato:sobolev}, $\chi f_{Q,w}\in \dot W^2_m(\R^n)$ with bounded norm.
By the established square-function estimates on $R_t^1$ and $R_t^2$,
\begin{equation*}\frac{1}{\abs{Q}}
 	\int_0^{\ell(Q)} \int_Q \abs{ R_t^1 (\nabla_\pureH^m (\chi f_{Q,w}))(x) + R_t^2 (\mat A_\pureH,\nabla_\pureH^m (\chi f_{Q,w}))(x) }^2 \frac{dx\,dt}{t}
\leq C
.\end{equation*}
Recall that $R_t^1$ and $R_t^2$ satisfy the decay estimates~\eqref{A:eqn:R:1:decay} and~\eqref{A:eqn:R:2:decay}. Combined with the bound~\eqref{A:eqn:kato:sobolev} on~$\nabla_\pureH^m f_{Q,w}$, we may establish that
\begin{align*}\frac{1}{\abs{Q}}
 	\int_0^{\ell(Q)} \int_Q \abs{ R_t^2 (\mat A_\pureH,\nabla_\pureH^m ((1-\chi)f_{Q,w}))(x) }^2 \frac{dx\,dt}{t}
\leq C
,\\
\frac{1}{\abs{Q}}
 	\int_0^{\ell(Q)} \int_Q \abs{ R_t^1 (\nabla_\pureH^m ((1-\chi)f_{Q,w}))(x) }^2 \frac{dx\,dt}{t}
\leq C
.\end{align*}

We are left with the term \eqref{A:eqn:kato:argument}. By the bound \eqref{A:eqn:S:infty} and the local bound \eqref{A:eqn:Meyers:lowest}, if $k$ is large enough then
\begin{equation*}\abs{t^{k+1} \partial_t^{m+k} \s^L_\nabla \mat A_\pureH(x,t)}
\leq
C\doublebar{\mat A_\pureH}_{L^\infty(\R^n)}^2.\end{equation*}
Thus, it suffices to show that
\begin{equation*}\frac{1}{\abs{Q}}
 	\int_0^{\ell(Q)} \int_Q \abs{ A_t^Q \nabla^m_\pureH f_{Q,w}(x) - Q_t\nabla_\pureH^m f_{Q,w}(x) }^2 \frac{dx\,dt}{t}
\leq C
.\end{equation*}
This follows from a standard orthogonality estimate. See \cite[Section~9.3]{BarHM17} for the details of this argument in the present case, under the assumption that $\eta$ is supported in $B(0,1/2)$. This completes the proof.

\section{\texorpdfstring{$\D^{\mat A}$}{The double layer potential}: square-function estimates}
\label{A:sec:D:square:variant}

In this section we will establish the following square function estimate on the double layer potential; this is the bound~\eqref{A:eqn:D:square:rough:intro} of Theorem~\ref{A:thm:square:rough}. The key argument is formula~\eqref{A:eqn:D:SD}; from this formula the desired bound follows from the known results of \cite{BarHM17} and Section~\ref{A:sec:S:square:variant}.

\begin{thm}\label{A:thm:D:square:variant}
Suppose that $L$ is an operator of the form~\eqref{A:eqn:L} of order~$2m$, $m\geq 1$, acting on functions defined in $\R^\dmn$, $\dmnMinusOne\geq 1$, associated to coefficients~$\mat A$ that are $t$-independent in the sense of formula~\eqref{A:eqn:t-independent} and satisfy the bounds \eqref{A:eqn:elliptic} and~\eqref{A:eqn:elliptic:bounded}.

Then the double layer potential $\D^{\mat A}$ satisfies the square-function estimate
\begin{align}
\label{A:eqn:D:square:variant}
\int_{\R^n}\int_{-\infty}^\infty \abs{\nabla^m \D^{\mat A} \arr f(x,t)}^2\,\abs{t}\,dt\,dx
	& \leq C \doublebar{\arr f}_{L^2(\R^n)}^2
\end{align}
for all $\arr f\in \dot W\!A^2_{m-1,0}(\R^n)$.
\end{thm}

\begin{proof} By Remark~\ref{A:rmk:lower}, it suffices to work in the upper half-space. Furthermore, it suffices to establish this estimate under the assumption that $\arr f=\Tr_{m-1}f$ for some $f\in C^\infty_0(\R^\dmn)$.

Let $\eta$ be a smooth, compactly supported cutoff function defined on $\R$ with $\eta\equiv 1$ near zero. Define
\begin{equation*}\psi(x,t)= \sum_{j=0}^{m-1} \frac{1}{(j+1)!}t^{j+1}\,\eta(t)\,\partial_\perp^j f(x,0)\end{equation*}
and let $\widetilde f(x,t)=\partial_t \psi(x,t)$. Observe that $\arr f=\Tr_{m-1}f=\Tr_{m-1}\widetilde f$. Furthermore, if $\abs\xi=m$, then $\partial^\xi \psi(x,0)=\partial^{\xi-\vec e_\dmn}f(x,0)$ if $\xi_\dmn\geq 1$ and $\partial^\xi \psi(x,0)=0$ if $\xi_\dmn=0$.

By formula~\eqref{A:eqn:D:fundamental}, if $\abs\alpha=m$ then
\begin{align*}
\partial^\alpha \D^{\mat A} \arr f(x,t)
&=
- \!\!\sum_{\substack{\abs{\beta}=\abs{\xi}=m}} \int_{\R^\dmn_-} \partial_{x,t}^\alpha \partial_{y,s}^\beta E^L(x,t,y,s)\,A_{\beta\xi}(y) \, \partial^\xi \partial_s \psi(y,s)\,ds\,dy.
\end{align*}
Integrating by parts in~$s$, we see that
\begin{align*}
\partial^\alpha \D^{\mat A} \arr f(x,t)
&=
- \!\!\sum_{\substack{\abs{\beta}=\abs{\xi}=m}} \int_{\R^n} \partial_{x,t}^\alpha \partial_{y,s}^\beta E^L(x,t,y,0)\,A_{\beta\xi}(y) \, \partial^\xi \psi(y,0)\,dy
\\&\qquad
+ \!\!\sum_{\substack{\abs{\beta}=\abs{\xi}=m}} \int_{\R^\dmn_-} \partial_{x,t}^\alpha \partial_{y,s}^\beta \partial_s E^L(x,t,y,s)\,A_{\beta\xi}(y) \, \partial^\xi \psi(y,s)\,ds\,dy.
\end{align*}
By formula~\eqref{A:eqn:fundamental:vertical} and formulas~\eqref{A:eqn:S:variant} and~\eqref{A:eqn:D:fundamental} for $\s^L_\nabla$ and~$\D^{\mat A}$,
\begin{align}
\label{A:eqn:D:SD}
\partial^\alpha \D^{\mat A} \arr f(x,t)
&=
- \partial^\alpha
\s^L_\nabla(\mat A\nabla^m\psi\big\vert_{\R^n})(x,t)
+\partial^\alpha \partial_t \D^{\mat A}(\Tr_{m-1}\psi)(x,t)
.\end{align}
Thus,
\begin{align*}\int_{\R^n}\int_0^\infty \abs{\nabla^m \D^{\mat A} \arr f(x,t)}^2\,{t}\,dt\,dx
&\leq
C\int_{\R^n}\int_0^\infty \abs{\nabla^m\s^L_\nabla(\mat A\nabla^m\psi\big\vert_{\R^n}) (x,t)}^2\,{t}\,dt\,dx
\\&\qquad+C\int_{\R^n}\int_0^\infty \abs{\nabla^m \partial_t \D^{\mat A}(\Tr_{m-1}\psi)(x,t)}^2\,{t}\,dt\,dx
.\end{align*}
We have that $\doublebar{\nabla_\pureH \Tr_{m-1}\psi}_{L^2(\R^n)}\leq C\doublebar{\arr f}_{L^2(\R^n)}$, and so by the bound \eqref{A:eqn:D:square} the second integral on the right-hand side is at most~$C\doublebar{\arr f}_{L^2(\R^n)}^2$. Also, we have that $\doublebar{\mat A\nabla^m \psi\big\vert_{\R^n}}_{L^2(\R^n)}\leq C\doublebar{\arr f}_{L^2(\R^n)}$ and so by Theorem~\ref{A:thm:S:square:variant} the first integral is at most~$C\doublebar{\arr f}_{L^2(\R^n)}^2$.
\end{proof}

\section{Vertical derivatives of \texorpdfstring{$\s^L_\nabla$}{the modified single layer potential}: estimates for \texorpdfstring{$p>2$}{p>2} and other bounds}
\label{A:sec:S:vertical:variant}

In this section we prove the bounds \eqref{A:eqn:S:variant:Carleson}, \eqref{A:eqn:lusin:p:variant:intro} and~\eqref{A:eqn:S:Schwartz:p} on $\s^L_\nabla$ of
Theorem~\ref{A:thm:Carleson:intro}. These estimates will be used in our paper~\cite{BarHM17pB} to establish certain Fatou-type theorems.

\begin{lem}\label{A:lem:S:carleson:variant}
Suppose that $L$ is an operator of the form~\eqref{A:eqn:L} of order~$2m$, $m\geq 1$, acting on functions defined in $\R^\dmn$, $\dmnMinusOne\geq 1$, associated to coefficients~$\mat A$ that are $t$-independent in the sense of formula~\eqref{A:eqn:t-independent} and satisfy the bounds \eqref{A:eqn:elliptic} and~\eqref{A:eqn:elliptic:bounded}.

Then we have the Carleson measure estimate
\begin{align*}
\sup_{Q\subset\R^n}
\frac{1}{\abs{Q}}\int_Q\int_{0}^{\ell(Q)} \abs{t^k \partial_t^{m+k}\s^L_\nabla \arr b(x,t)}^2\,{t}\,dt\,dx
	& \leq C \doublebar{\arr b}_{L^\infty(\R^n)}^2
\end{align*}
for any $k\geq (n-1)/2$, where the supremum is taken over all cubes $Q$ contained in~$\R^n$. By the Caccioppoli inequality, the estimate
\begin{equation*}\sup_{Q\subset\R^n} \frac{1}{\abs{Q}}\int_Q \int_0^{\ell(Q)} \abs{t^{k+m} \nabla^m \partial_t^{m+k}\s^L_\nabla \arr b(x,t)}^2\,t\,dx\,dt
\leq C\doublebar{\arr b}_{L^\infty(\R^n)}^2\end{equation*}
is also valid. Corresponding estimates are valid in the lower half-space.

Furthermore, the bound
\begin{equation*}
\doublebar{\mathcal{A}_2^\pm (\abs t^{k+1}\,\partial_t^{m+k} \s^L_\nabla \arr h)}_{L^p(\R^n)}\leq C(p) \doublebar{\arr h}_{L^p(\R^n)}
\end{equation*}
is valid for all $2\leq p<\infty$, where $\mathcal{A}_2^\pm$ is as in formula~\eqref{A:eqn:lusin}.
\end{lem}

\begin{proof} By Remark~\ref{A:rmk:lower}, it suffices to work in the upper half-space.
By Theorem~\ref{A:thm:S:square:variant} and the Caccioppoli inequality applied in Whitney cubes, and by the decay estimate~\eqref{A:eqn:S:variant:decay}, if
$2k+1\geq n$, then the operators
\begin{equation*}\Theta_t \arr h=t^{k+1} \partial_t^{m+k}\s^L_\nabla\arr h(\,\cdot\,,t)\end{equation*}
satisfy the conditions of Lemma~\ref{A:lem:square:carleson} and Remark~\ref{A:rmk:interpolation:carleson}, and so the given bounds are valid.
\end{proof}

We conclude this section by establishing the bound~\eqref{A:eqn:S:Schwartz:p}; this estimate will be of use in \cite{BarHM17pB}.
Recall that we have square-function estimates on $R_t^2$, where
\begin{align*}
R_t^2(\arr b,f)(x)
&= t^{k+1} \partial_\perp^{m+k}\s^L_\nabla(\arr b\, \mathcal{Q}_t f)(x,t)-t^{k+1}\partial_t^{m+k}\s^L_\nabla\arr b(x,t)\, \mathcal{Q}_t f(x)
\end{align*}
and where $\partial_\perp^{m+k}\s^L_\nabla(\arr b\mathcal{Q}_t f)(x,t)=\partial_t^{m+k}\s^L_\nabla(\arr b\mathcal{Q}_s f)(x,t)\big\vert_{s=t}$.
We would like to estimate the term $t^{k+1} \partial_\perp^{m+k}\s^L_\nabla(\arr b\, \mathcal{Q}_t f)(x,t)$ alone.

\begin{lem}\label{A:lem:S:Schwartz:p}
Suppose that $L$ is an operator of the form~\eqref{A:eqn:L} of order~$2m$, $m\geq 1$, acting on functions defined in $\R^\dmn$, $\dmnMinusOne\geq 1$, associated to coefficients~$\mat A$ that are $t$-independent in the sense of formula~\eqref{A:eqn:t-independent} and satisfy the bounds \eqref{A:eqn:elliptic} and~\eqref{A:eqn:elliptic:bounded}.

Let $\eta$ be a Schwartz function defined on $\R^n$ with $\int\eta=1$.
Let $\mathcal{Q}_t$ denote convolution with $\eta_t=t^{-n}\eta(\,\cdot\,/t)$.
Let $k>0$ be an integer. Let $\arr b$ be any array of bounded functions. If $k$ is large enough, then for any $p$ with $1<p<\infty$, we have that
\begin{equation*}
\doublebar{\mathcal{A}_2^\pm(\abs t^{k+1}\partial_\perp^{k+m}\s^L_\nabla (\arr b \mathcal{Q}_{\abs t} h))}_{L^p(\R^n)}
\leq C(p)\doublebar{\arr b}_{L^\infty(\R^n)}\doublebar{h}_{L^p(\R^n)}
\end{equation*}
where $\mathcal{A}_2^\pm$ is as in formula~\eqref{A:eqn:lusin}, and
where the constant $C(p)$ depends only on $p$, $k$, the Schwartz constants of~$\eta$, and on the standard parameters $n$, $m$, $\lambda$, and~$\Lambda$.
\end{lem}

\begin{proof} By Remark~\ref{A:rmk:lower}, it suffices to work in the upper half-space.
Let $R_t h(x)=t^{k+1}\partial_\perp^{k+m}\s^L_\nabla (\arr b \mathcal{Q}_t h)(x,t)$.
Observe that
\begin{equation*}R_t h(x)
=
R_t^2(\arr b,h)(x)+
t^{k+1}\partial_t^{m+k}\s^L_\nabla\arr b(x,t)\, \mathcal{Q}_t h(x)
\end{equation*}
where $R_t^2$ is as in Lemma~\ref{A:lem:R:2}. Thus,
\begin{multline*}\int_{\R^n}\int_0^\infty \abs{R_t h(x)}^2\,\frac{1}{t}\,dt\,dx
\\
\leq
2\int_{\R^n}\int_0^\infty \abs{R_t^2(\arr b,h)(x)}^2\,\frac{dt}{t}\,dx
+
2\int_{\R^n}\int_0^\infty \abs{
t^{k}\partial_t^{m+k}\s^L_\nabla\arr b(x,t)\, \mathcal{Q}_t h(x)
}^2\,t\,dt\,dx
.\end{multline*}
If $k$ is large enough, then we may bound the term involving $R_t^2$ using Lemma~\ref{A:lem:R:2}, while by Lemma~\ref{A:lem:S:carleson:variant}, $\abs{t^k\partial_t^{m+k} \s^L_\nabla\arr b(x,t)}^2\,t\,dx\,dt$ is a Carleson measure and so we may bound the second term using the bound~\eqref{A:eqn:NQ:carleson} and Carleson's lemma.

Thus, we have
$L^2$ boundedness of $h\mapsto\mathcal{A}_2^+(R_t h)$. $L^p$ boundedness for $2<p<\infty$ follows from Remark~\ref{A:rmk:interpolation:carleson} and the decay estimate~\eqref{A:eqn:SQ:decay}.

We are left with the case $1<p<2$.
Let $H^1(\R^n)$ be the standard Hardy space. It is well known that interpolation between Hardy and Lebesgue spaces is valid; that is, if we can establish the estimate
\begin{equation}
\label{A:eqn:H1:bound}
\doublebar{\mathcal{A}_2^+ (R_th)}_{L^1(\R^n)}
\leq C \doublebar{h}_{H^1(\R^n)},
\end{equation}
then boundedness $L^p\mapsto L^p$ will be valid for $1<p<2$ by interpolation. See \cite[Section~5]{FefS72}.

One of the most useful properties of the Hardy spaces is the atomic decomposition. That is, to establish the bound~\eqref{A:eqn:H1:bound}, it suffices to show that
\begin{equation*}\doublebar{\mathcal{A}_2^+ (R_ta)}_{L^1(\R^n)}
\leq C \end{equation*}
whenever $a$ is supported in a cube $Q\subset\R^n$, $\doublebar{ a}_{L^2(\R^n)}=\doublebar{ a}_{L^2(Q)}\leq\abs{Q}^{-1/2}$, and $\int_Q a=0$. See~\cite{Coi74,Lat78,MedSV08}.

Choose some such~$a$ and~$Q$ and let $x_Q$ be the midpoint of~$Q$. Then
\begin{align*}
\int_{4Q}\mathcal{A}_2^+(R_t a)(x)\,dx
&\leq 4^{n/2}\abs{Q}^{1/2}\biggl(\int_{cQ}\mathcal{A}_2^+(R_t a)(x)^2\,dx\biggr)^{1/2}
\\&\leq C\abs{Q}^{1/2}\doublebar{a}_{L^2(\R^n)}
\leq C.\end{align*}

Now, let $x\in\R^n\setminus 4Q$, so $\abs{x-x_Q}\geq 2\ell(Q)$. Then
\begin{equation*}\mathcal{A}_2^+(R_t a)(x) = \biggl(\int_0^\infty \int_{\abs{x-z}<t} \abs{R_t a(z)}^2\,\frac{dz\,dt}{t^\dmn}\biggr)^{1/2}.\end{equation*}
Recall that
\begin{align*}
R_ta(z)
&=
	t^{k+1} \partial_\perp^{k+m}\s^L_\nabla (\arr b\mathcal{Q}_t a)(z,t)
=
	\sum_{\abs\alpha=m}
	t^{k+1} \partial_\perp^{k+m} \s^L_\nabla ( b_\alpha\mathcal{Q}_t a\,\arr e_\alpha)(z,t)
.\end{align*}
Let the matrix $\widetilde A$ and the constants $\kappa_\xi$ be as in formula~\eqref{A:eqn:coefficients:high}.
By formula~\eqref{A:eqn:S:variant:high},
\begin{align*}
R_ta(z)
&=
	\sum_{{\abs\alpha=m}}
	 t^{k+1} \partial_\perp^{k+m}
	\sum_{\substack{\abs\zeta=M}} \kappa_\zeta
	\partial^{2\zeta}
	\s^{\widetilde{L}}_\nabla \Bigl(
	\sum_{\abs\xi=M} \kappa_\xi b_\alpha\mathcal{Q}_t a\,\arr e_{\alpha+2\xi}\Bigr)(z,t)
.\end{align*}
Thus,
\begin{align*}\int_{\abs{x-z}<t} \abs{R_ta(z)}^2\,dz
&\leq
C\!
\sum_{{\abs\alpha=m}}
\sum_{{\abs\xi=M}}
\int_{\abs{x-z}<t}\abs[big]{
	 t^{k+1} \nabla^{\widetilde m} \partial_\perp^k
	\s^{\widetilde{L}}_\nabla \bigl(
	 b_\alpha\mathcal{Q}_t a\,\arr e_{\alpha+2\xi}\bigr)(z,t) }^2\,dz
\end{align*}
where $\widetilde m=m+2M$.
Let
\begin{equation*}
u_{t}^{\alpha,\xi}(z,r)
=t^{k+1-\widetilde m} \partial_{{r}}^k
	\s^{\widetilde{L}}_\nabla \bigl(
	 b_\alpha\mathcal{Q}_t a\,\arr e_{\alpha+2\xi}\bigr)(z,r).\end{equation*}
Observe that $\widetilde Lu_{t}^{\alpha,\xi}=0$ in $\R^\dmn_\pm$. Thus, by Lemma~\ref{A:lem:slices} and the Caccioppoli inequality,
\begin{align*}\int_{\abs{x-z}<t} \abs{R_ta(z)}^2\,dz
&\leq
C
\sum_{{\abs\alpha=m}}
\sum_{{\abs\xi=M}}
\int_{\abs{x-z}<t} \abs[big]{
	 t^{\widetilde m} \nabla^{\widetilde m} u_{t}^{\alpha,\xi}(z,t) }^2\,dz
\\&\leq
C
\sum_{{\abs\alpha=m}}
\sum_{{\abs\xi=M}}
\int_{\abs{x-z}<2t} \fint_{t/2}^{2t} \abs[big]{u_{t}^{\alpha,\xi}(z,r) }^2\,dr\,dz
.\end{align*}
Thus,
for any multiindices $\alpha$, $\xi$ and $\varepsilon=\alpha+2\xi$,
we wish to bound
$u_{t}(z,r)=u_{t}^{\alpha,\xi}(z,r) $
for $\abs{z-x}<2t$ and $t/2<r<2t$.

Now, observe that by formula~\eqref{A:eqn:S:variant},
\begin{align*}
u_{t}(z,r)
&=
	t^{k+1-\widetilde m}
	\int_{\R^n}\partial_r^k\partial_{y,{s}}^\varepsilon E^{\widetilde L}(z,r,y,0)\, b_\alpha(y)\,\mathcal{Q}_t a(y)\,dy
\end{align*}
where $\abs\varepsilon=\widetilde m$ and $\widetilde L$ is an elliptic operator of order~$2\widetilde m$.

Recall that $a$ is supported in the cube $Q\subset\R^n$. Let $A_j(Q)$ be as in formula~\eqref{A:eqn:annuli}. Let $J$ be such that $x\in A_J$; then $2^{{J-1}}\ell(Q)\leq \abs{x-x_Q}\leq \sqrt{n}\, 2^J\ell(Q)$. Recall that $x\in \R^n\setminus 4Q$, and so we need only consider $J\geq 2$. Define
\begin{equation*}U_j =U_j(z,r)= t^{k-\widetilde m+1}\int_{A_j(Q)} \abs{\partial_r^k\partial_{y,{s}}^\varepsilon E^{\widetilde L}(z,r,y,0)} \,dy.\end{equation*}
We then have that
\begin{equation*}\abs{u_{t}(z,r)} \leq \doublebar{\arr b}_{L^\infty(\R^n)}\sum_{j=0}^\infty U_j(z,r) \doublebar{\mathcal{Q}_t a}_{L^\infty(A_j(Q))}.\end{equation*}

We now establish bounds on~$U_j$. We choose $M$ large enough that $2\widetilde m>\dmn$, and so vertical derivatives of $E^{\widetilde L}$ are pointwise bounded by Theorem~\ref{A:thm:Meyers}. We assume $k\geq \widetilde m$.

By Lemma~\ref{A:lem:slices}, the Caccioppoli inequality, the bound~\eqref{A:eqn:fundamental:far}, and Theorem~\ref{A:thm:Meyers},
if $r>0$, then
\begin{multline*}
\int_{A_j(Q)} \abs{\partial_r^k\partial_{y,{s}}^\varepsilon E^{\widetilde L}(z,r,y,0)} \,dy
\\\leq
	C (2^j\ell(Q))^{(n-1)/2} (r+\dist(z,A_j(Q)))^{-\pdmn/2}
	\\\times
	\max(2^j\ell(Q),r+\dist(z,A_j(Q)))^{\widetilde m-k}.
\end{multline*}
We will consider the cases $j\geq J+2$, $j\leq J-2$ and $\abs{J-j}\leq 1$ separately.

We begin with $j\geq J+2$. If $\abs{x-z}<2t$ and $t/2<r<2t$, then
\begin{equation*} U_j(z,r) \leq C t^{k+1-\widetilde m} (2^j\ell(Q))^{(n-1)/2} (t+2^j\ell(Q))^{\widetilde m-k-\pdmn/2}
\qquad\text{ for all }j\geq J+2. \end{equation*}

We now consider the cases $j\leq J+1$.
If $t>\abs{x-x_Q}$, let $2^{J'}\ell(Q)\leq 4t\leq 2^{J'+1}\ell(Q)$; because $\abs{x-x_Q}>2^{J-1}\ell(Q)$, we have that $J'\geq J+1$. Then
\begin{equation*}\sum_{j=0}^{J'} U_j(z,r) \leq C .\end{equation*}

If $0 < t\leq \abs{x-x_Q}$, we bound $U_j$ differently. First, observe that
\begin{equation*}\sum_{j=0}^{J-2} U_j(z,r) \leq C t^{k+1-\widetilde m} \abs{x-x_Q}^{\widetilde m-k-1}.\end{equation*}

We have bound $U_j$ for $j\geq J+2$; we are left with $U_{J}$ and $U_{J\pm 1}$.

Let $\Delta_0 = B(x,2t)\subset\R^n$ and $\Delta_\ell=B(x,2^{\ell+1}t)\setminus B(x,2^\ell t)$. We remark that $\Delta_\ell$ differs from $A_j(Q)$ in that $\Delta_\ell$ is centered about $x$ rather than $x_Q$. As before, using Lemma~\ref{A:lem:slices}, the Caccioppoli inequality, the bound~\eqref{A:eqn:fundamental:far}, and Theorem~\ref{A:thm:Meyers}, we may show that
\begin{equation*}U_{J-1}+U_{J}+U_{J+1} \leq \sum_{\ell=0}^\infty t^{k-\widetilde m+1}\int_{\Delta_\ell} \abs{\partial_r^k\partial_{y,{s}}^\varepsilon E^{\widetilde L}(z,r,y,0)} \,dy\leq C
\end{equation*}
provided $k\geq \widetilde m$.

We now establish bounds on ${\mathcal{Q}_t a}$.
Because $\mathcal{Q}_t$ denotes convolution with~$\eta_t$, and $\eta\in L^2(\R^n)$, we have that
\begin{equation*}\doublebar{\mathcal{Q}_t a}_{L^\infty(\R^n)}\leq
Ct^{-n/2}\doublebar{a}_{L^2(\R^n)}\leq \frac{C}{t^{n/2}\sqrt{\abs{Q}}}.\end{equation*}
Because
$\mathcal{Q}_t$ is an approximate identity with a Schwartz kernel~$\eta_t$, and because $\int a=0$, we have that
\begin{equation*}\doublebar{\mathcal{Q}_t a}_{L^\infty(\R^n)}\leq C\diam(\supp a)\doublebar{\nabla \eta_t}_{L^\infty(\R^n)}\doublebar{a}_{L^1(\R^n)}\leq C\frac{\ell(Q)}{t^{n+1}}.\end{equation*}
Furthermore, if $N>0$ is an integer and $j\geq 1$ then
\begin{equation*}
\doublebar{\mathcal{Q}_t a}_{L^\infty(A_j(Q))} \leq C_N\frac{\ell(Q)}{t^{n+1}} \biggl(\frac{t}{2^j\ell(Q)}\biggr)^N
.\end{equation*}
Thus,
\begin{align*}
\abs{u_{t}(z,r)}
&\leq
	C\sum_{j=0}^\infty U_j(z,r) \doublebar{\mathcal{Q}_t a}_{L^\infty(A_j(Q))}
\\&\leq
	C\min\biggl(\frac{1}{ t^{n/2}\sqrt{\abs{Q}}}, \frac{\ell(Q)}{t^{n+1}}\biggr)
	\min\biggl(1, \frac{t}{\abs{x-x_Q}}\biggr)^{k+1-\widetilde m}
	\\&\qquad+
	C_{N_1}\frac{\ell(Q)}{t^{n+1}} \biggl(\frac{t}{\abs{x-x_Q}}\biggr)^{N_1}
	\\&\qquad+
	C_{N_2}\sum_{j=J+2}^\infty
	\frac{\ell(Q)}{t^{n+1}} \biggl(\frac{t}{2^j\ell(Q)}\biggr)^{N_2}
	\frac{t^{k+1-\widetilde m} (2^j\ell(Q))^{(n-1)/2} }{(t+2^j\ell(Q))^{k+1-\widetilde m+(n-1)/2}}
.\end{align*}
If $t<\ell(Q)$, then by letting $N_1=k+2-\widetilde m+n$ and $N_2=n+1$, we see that
\begin{align*}
\abs{u_{t}(z,r)}
&\leq
	\frac{C}{\sqrt{\abs Q}}\frac{t^{k+1-\widetilde m-n/2}}{ \abs{x-x_Q}^{k+1-\widetilde m} }
.\end{align*}
If $\ell(Q)<t<\abs{x-x_Q}$, then by letting $N_1=k+1-\widetilde m$ and $N_2=0$, we see that
\begin{align*}
\abs{u_{t}(z,r)}
&\leq
	\frac{C\ell(Q)t^{k-\widetilde m-n} }{ \abs{x-x_Q}^{k+1-\widetilde m}}
.\end{align*}

If $t>\abs{x-x_Q}$, then we bound $u_t(z,r)$ differently, by writing
\begin{align*}
\abs{u_{t}(z,r)}
&\leq
	\frac{C\ell(Q)}{t^{n+1}}\sum_{j=0}^{J'} U_j
	+
	 C_N\frac{\ell(Q)}{t^{n+1}} \sum_{j=J'}^\infty
	\biggl(\frac{t}{2^j\ell(Q)}\biggr)^{N+k+1-\widetilde m}
\end{align*}
and choosing $N\geq \widetilde m-k$ we see that $\abs{u_{t}(z,r)} \leq {C\ell(Q)}/{t^{n+1}}$.

Thus, we have that if $k$ is large enough, then for all $x\notin cQ$,
\begin{align*}
\mathcal{A}_2^+(R_t a)(x)
&\leq
	C\max_{\alpha,\xi}\biggl(\int_0^\infty \sup_{\abs{z-x}<2t,\>t/2<s<2t} \abs{u_{t}^{\alpha,\xi}(z,r)}^2\frac{dt}{t}\biggr)^{1/2}
\\&\leq
	\biggl(\int_0^{\ell(Q)} \frac{1}{\abs{Q}}\biggl(\frac{Ct^{k+1-\widetilde m-n/2}}{\abs{x-x_Q}^{k+1-\widetilde m}}\biggr)^2
	\frac{dt}{t} \biggr)^{1/2}
	\\&\qquad+
	\biggl(\int_{\ell(Q)}^{\abs{x-x_Q}} \biggl(\frac{C\ell(Q)t^{k-\widetilde m-n}}{\abs{x-x_Q}^{k+1-\widetilde m}}\biggr)^2\frac{dt}{t}\biggr)^{1/2}
	\\&\qquad+
	\biggl(\int_{\abs{x-x_Q}}^\infty \biggl(\frac{C\ell(Q)}{t^{n+1}}\biggr)^2\frac{dt}{t}\biggr)^{1/2}
\\&\leq \frac{C\ell(Q)}{\abs{x-x_Q}^\dmn}.
\end{align*}
This is in $L^1(\R^n\setminus 4Q)$, and so $ h\mapsto \mathcal{A}_2^+(R_t h)$ is bounded $H^1\mapsto L^1$. By interpolation, it is bounded $L^p\mapsto L^p$ for any $1<p<\infty$, as desired.
\end{proof}

\section{Bounds on \texorpdfstring{$\s^L$}{the single layer potential} for \texorpdfstring{$1<p<2$}{1<p<2} in low dimensions}
\label{A:sec:atomic}

In Sections~\ref{A:sec:S:vertical} and~\ref{A:sec:S:vertical:variant}, we established the area integral estimates \eqref{A:eqn:lusin:p:intro} and~\eqref{A:eqn:lusin:p:variant:intro} for $2\leq p<\infty$. In low dimensions we can extend the bound~\eqref{A:eqn:lusin:p:intro} below $p=2$.

\begin{lem}\label{A:lem:atomic} Let $L$, $m$, and $n$ be as in Lemma~\ref{A:lem:lusin:S}.
Suppose that the De Giorgi-Nash-Moser type estimate
\begin{equation}\label{A:eqn:DGNM} \sup_{\substack{X,Y\in B(X_0,r)\\X\neq Y}} \frac{\abs{\nabla^{m-1}u(X)-\nabla^{m-1}u(Y)}}{\abs{X-Y}^\varepsilon} \leq \frac{H}{r^\varepsilon} \biggl(\fint_{B(X_0,2r)} \abs{\nabla^{m-1} u}^2\biggr)^{1/2}\end{equation}
is valid in all balls $B(X_0,r)\subset\R^\dmn$ and for all $u\in \dot W^2_m(B(X_0,2r))$ with $Lu=0$ in $B(X_0,2r)$,
for some positive constants $H$, $\varepsilon>0$ that depend only on $L$ (and not on $u$, $X$, $Y$, $X_0$, or~$r$). Then the area integral estimates \eqref{A:eqn:lusin:atomic} and~\eqref{A:eqn:lusin:atomic:k} are valid for $1<p<2$ and $k\geq 2$.

If the ambient dimension $\dmn$ satisfies either $\dmn=2$ or $\dmn=3$, then the estimate~\eqref{A:eqn:DGNM} is valid, and thus so are
the area integral estimates \eqref{A:eqn:lusin:atomic} and~\eqref{A:eqn:lusin:atomic:k}.
\end{lem}

\begin{proof}
We begin by showing that the bound \eqref{A:eqn:DGNM} is valid whenever $\dmn=2$ or $\dmn=3$.
By Theorem~\ref{A:thm:Meyers}, there is some $p_L^+>2$ such that, if $0<q<p_L^+$ and if $Lu=0$ in $B(X_0,2r)$, then $\nabla^m u\in L^q(B(X_0,r))$. If $\dmn=2$, then by Morrey's inequality, $\nabla^{m-1} u$ is H\"older continuous; furthermore, by these theorems and the Caccioppoli inequality we have that
\begin{equation*}\sup_{\substack{X,Y\in B(X_0,r)\\X\neq Y}} \frac{\abs{\nabla^{m-1} u(X)-\nabla^{m-1} u(Y)}}{\abs{X-Y}^{1-2/q}}\leq \frac{C}{r^{1-2/q}}\biggl(\fint_{B(X_0,2r)}\abs{\nabla^{m-1} u}^2\biggr)^{1/2}. \end{equation*}

The next argument is essentially that of \cite[Appendix~B]{AlfAAHK11}. By Lemma~\ref{A:lem:slices}, if $\mat A$ is $t$-independent then there is some $q>2$ such that, if $Lu=0$ in $2Q\times (t-\ell(Q),t+\ell(Q))$, then $\nabla^m u\in L^q(Q\times\{t\})$. If $\dmn=3$, then $n=2$ and $Q\subset\R^2$, so by Morrey's inequality
\begin{equation*}\sup_{\substack{x,y\in Q\\x\neq y}} \frac{\abs{\nabla^{m-1} u(x,t)-\nabla^{m-1} u(y,t)}}{\abs{x-y}^{1-2/q}}\leq C\ell(Q)^{2/q}\biggl(\fint_{2Q}\fint_{t-\ell(Q)}^{t+\ell(Q)} \abs{\nabla^m u}^2\biggr)^{1/2}. \end{equation*}
We may use the Caccioppoli inequality and Lemma~\ref{A:lem:slices} to bound $\nabla^{m-1}\partial_\dmn u$; thus, we have that if $\dmn=3$ then
\begin{equation*}\sup_{\substack{X,Y\in B(X_0,r)\\X\neq Y}} \frac{\abs{\nabla^{m-1} u(X)-\nabla^{m-1} u(Y)}}{\abs{X-Y}^{1-2/q}}\leq \frac{C}{r^{1-2/q}}\biggl(\fint_{B(X_0,2r)}\abs{\nabla^{m-1} u}^2\biggr)^{1/2}. \end{equation*}

Thus, if $\dmn=2$ or $\dmn=3$, then the bound~\eqref{A:eqn:DGNM} is valid with $\varepsilon=1-2/q$, where $2<q<p^+_L$.

We now establish the estimates \eqref{A:eqn:lusin:atomic} and~\eqref{A:eqn:lusin:atomic:k} for $1<p<2$ and $k\geq 2$.

Let $\dmn\geq 2$ and suppose that the bound~\eqref{A:eqn:DGNM} is valid.
By the bound \eqref{A:eqn:S:square} (and the Caccioppoli inequality), the two bounds are valid for $p=2$. As in the proof of Lemma~\ref{A:lem:S:Schwartz:p}, let $H^1(\R^n)$ be the Hardy space; it suffices to show that
\begin{align*}
\doublebar{\mathcal{A}_2^\pm (\abs{t}\,\nabla^{m-1} \partial_t^{2} \s^L (a\arr e_\gamma))}_{L^1(\R^n)}
&\leq C ,
\\
\doublebar{\mathcal{A}_2^\pm (\abs{t}^k\,\nabla^m \partial_t^{k} \s^L (a\arr e_\gamma))}_{L^1(\R^n)}
&\leq C(k)
\end{align*}
whenever $\abs\gamma=m-1$ and whenever $a$ is an $H^1$-atom.

By Remark~\ref{A:rmk:lower} it suffices to consider $\mathcal{A}_2^+$.
Choose some $H^1$ atom $a$ and some multiindex $\gamma$ with $\abs\gamma=m-1$, and let $\arr a=a\arr e_\gamma$.
By H\"older's inequality,
\begin{align*}\int_{16Q} \mathcal{A}_2^+(t^k\nabla^m\partial_t^k\s^L\arr a)
&\leq 4\abs{Q}^{1/2}\biggl(\int_{\R^n} \mathcal{A}_2^+(t^k\nabla^m\partial_t^k\s^L\arr a)^2\biggr)^{1/2}
\\&=4\abs{Q}^{1/2}\biggl(\int_{\R^n}\int_0^\infty t^{2k-1}\abs{\nabla^m\partial_t^k\s^L\arr a(x,t)}^2\,dt\,dx\biggr)^{1/2}
.\end{align*}
Applying the Caccioppoli inequality in Whitney cubes, we see that if $k\geq 1$ then
\begin{align*}\int_{16Q} \mathcal{A}_2^+(t^k\nabla^m\partial_t^k\s^L\arr a)
&\leq C\abs{Q}^{1/2}\biggl(\int_{\R^n}\int_0^\infty \abs{\nabla^m\partial_t\s^L\arr a(x,t)}^2t\,dt\,dx\biggr)^{1/2}
\end{align*}
and so by the bound~\eqref{A:eqn:S:square},
\begin{align*}\int_{16Q} \mathcal{A}_2^+(t^k\nabla^m\partial_t^k\s^L\arr a)
&\leq C\abs{Q}^{1/2}\doublebar{\arr a}_{L^2(\R^n)}
\leq C.\end{align*}

We want to bound $\mathcal{A}_2^+(t^k\nabla^m\partial_t^{k}\s^L\arr a)(z)$ or $\mathcal{A}_2^+(\nabla^{m-1}\partial_t^2\s^L\arr a)(z)$ for $z\notin 16Q$. By formula~\eqref{A:eqn:S:fundamental},
\begin{equation*}\nabla^{m-1}\partial_t^{k}\s^L\arr a(x,t)
=
\int_Q \nabla_{x,t}^{m-1}\partial_t^{k} \partial_{y,s}^\gamma E^L(x,t,y,0)\,a(y) \,dy
.\end{equation*}
Combining the bound \eqref{A:eqn:DGNM} with the estimate~\eqref{A:eqn:fundamental:far}, we see that if $k\geq 1$ then
\begin{equation*}\abs{\nabla_{x,t}^{m-1}\partial_{t}^{k} \nabla^{m-1}_{y,s} E^L(x,t,y,0)}\leq \frac{C}{(\abs{x-y}+\abs t)^{n+k-1}}
\end{equation*}
and that if $\abs{y_1-y_2}<\abs{x-y_1}/2+\abs{t}/2$, then
\begin{multline*}\abs{\nabla_{x,t}^{m-1}\partial_{t}^{k} \nabla^{m-1}_{y,s} E^L(x,t,y_1,0)-\nabla_{x,t}^{m-1}\partial_{t}^{k} \nabla^{m-1}_{y,s} E^L(x,t,y_2,0)}
\\\leq
\frac{C\abs{y_1-y_2}^\varepsilon}{(\abs{x-y_1}+\abs t)^{n+k-1+\varepsilon}}
.\end{multline*}
Thus, using the cancellation properties of the atom~$a$, we have that
\begin{equation*}
\abs{\nabla^{m-1}\partial_{t}^{k}\s^L\arr a(x,t)}
\leq \frac{C \ell(Q)^\varepsilon}{(\abs t+\dist(x,Q))^{n+k-1+\varepsilon}}
.\end{equation*}
Applying the Caccioppoli inequality in Whitney cubes, we see that if $k\geq 2$ then
\begin{align*}\mathcal{A}_2^+(t^k\nabla^{m}\partial_{t}^{k}\s^L(a\arr e_\gamma))(z)
&\leq \biggl(\int_0^\infty \int_{\abs{x-z}<2t}
t^{2}\abs{\nabla^{m-1}\partial_{t}^2\s^L(a\arr e_\gamma)(x,t)}^2 \,\frac{dx\,dt}{t^\dmn}\biggr)^{1/2}
.\end{align*}
Clearly $\mathcal{A}_2^+(t\nabla^{m-1}\partial_{t}^2\s^L(a\arr e_\gamma))(z)$ is controlled by the right-hand side as well. But
\begin{multline*}
\biggl(\int_0^\infty \int_{\abs{x-z}<2t}
t^{2}\abs{\nabla^{m-1}\partial_{t}^2\s^L(a\arr e_\gamma)(x,t)}^2 \,\frac{dx\,dt}{t^\dmn}\biggr)^{1/2}
\\\leq
\biggl(\int_0^\infty \int_{\abs{x-z}<2t} \frac{C t^2 \ell(Q)^{2\varepsilon}}{( t+\dist(x,Q))^{2n+2\varepsilon}}\,\frac{dx\,dt}{t^\dmn}\biggr)^{1/2}
\leq \frac{C\ell(Q)^\varepsilon}{\dist(z,Q)^{n+\varepsilon}}
\end{multline*}
which is in~$L^1(\R^n\setminus 16Q)$ with norm independent of~$Q$, as desired.
\end{proof}

\section{Bounds on \texorpdfstring{$\s^L$}{the single layer potential} and \texorpdfstring{$\s^L_\nabla$}{the modified single layer potential} for \texorpdfstring{$2-\varepsilon<p<2$}{p<2}}
\label{A:sec:p-}

In this section we conclude the paper by proving the bounds~\eqref{A:eqn:S:p-} and~\eqref{A:eqn:S:p-:variant}. We hope to prove these bounds in the dual range $2<p<2+\varepsilon$ in a future paper.

\begin{thm} \label{A:thm:S:p-} Let $L$, $m$, and $n$ be as in Lemma~\ref{A:lem:lusin:S}. Then there is some $\varepsilon>0$ depending only on the standard constants $m$, $n$, $\lambda$ and~$\Lambda$ such that the bounds
\begin{align*}
\doublebar{\mathcal{A}_2(t\,\nabla^m \partial_t\s^L\arr g)}_{L^p(\R^n)}
\leq C(p)\doublebar{\arr g}_{L^p(\R^n)}
,\\
\doublebar{\mathcal{A}_2(t\,\nabla^m \s^L_\nabla\arr h)}_{L^p(\R^n)}
\leq C(p)\doublebar{\arr h}_{L^p(\R^n)}
\end{align*}
are valid for all $2-\varepsilon< p<2$.
\end{thm}

\begin{proof}
By formula~\eqref{A:eqn:S:S:vertical} and by density it suffices to consider only $\s^L_\nabla \arr h$ with $\arr h\in L^2(\R^n)\cap L^p(\R^n)$. Furthermore, by formula~\eqref{A:eqn:S:variant:high} we may assume $2m>\dmn$.

Let $T_2^p=\{\psi:\mathcal{A}_2\psi\in L^p(\R^n)\}$
with the natural norm. By \cite[Theorem~2, Section~5]{CoiMS85}, if $1<p<\infty$ then under the inner product
\begin{equation*}\langle f,g\rangle = \int_{\R^\dmn_+} f(x,t)\,g(x,t)\,\frac{dx\,dt}{t}\end{equation*}
the dual space to $T_2^p$ is $T_2^{p'}$, where $1/p+1/p'=1$.
Thus,
\begin{equation*}\doublebar{\mathcal{A}_2(t\nabla^m \s^L_\nabla\arr h)}_{L^p(\R^n)}
=
\sup_{0\neq\arr \Psi\in T_2^{p'}} \frac{1}{\doublebar{\mathcal{A}_2(\arr \Psi)}_{L^{p'}(\R^n)}} \abs[bigg]{\int_{\R^\dmn_+}\langle \arr \Psi,\nabla^m \s^L_\nabla\arr h\rangle}
.\end{equation*}
In fact, we may take the supremum over $\arr\Psi$ bounded and with support compactly contained in~$\R^\dmn_+$.

For such~$\arr \Psi$, by formula~\eqref{A:eqn:S:variant}
\begin{align*}\langle \arr \Psi,\nabla^m \s^L_\nabla\arr h\rangle_{\R^\dmn_+}
&=\int_{\R^\dmn_+}\langle \arr \Psi(x,t),\nabla^m \s^L_\nabla\arr h(x,t)\rangle \,dx\,dt
\\&=
\sum_{\abs\alpha=\abs\beta=m}\int_{\R^\dmn_+}
\int_{\R^n} \partial_{x,t}^\beta \partial_{y,s}^\alpha E^L(x,t,y,0)\,h_\alpha(y)\,dy \, \overline{\Psi_\beta(x,t)} \,dx\,dt
\end{align*}
and by formula~\eqref{A:eqn:fundamental:symmetric}
\begin{align*}\langle \arr\Psi, \nabla^m \s^L_\nabla\arr h\rangle_{\R^\dmn_+}
&=
\sum_{\abs\alpha=\abs\beta=m}\int_{\R^n} \int_{\R^\dmn_+}
\overline{\partial_{x,t}^\beta \partial_{y,s}^\alpha E^{L^*}(y,0,x,t) \Psi_\beta(x,t)} \,dx\,dt\,h_\alpha(y)\,dy
\end{align*}
and so by formula~\eqref{A:eqn:fundamental:2}
\begin{align*}\langle \arr\Psi, \nabla^m \s^L_\nabla\arr h\rangle_{\R^\dmn_+}
&=
\sum_{\abs\alpha=m}\int_{\R^n}\! \partial^\alpha
\overline{\Pi^{L^*}(\1_+\arr\Psi)(y,0)}
\,h_\alpha(y)\,dy
=\langle \Tr_m^+\Pi^{L^*}(\1_+\arr\Psi),\arr h\rangle_{\R^n}
\end{align*}
where $\1_+\arr\Psi$ denotes the extension of $\arr\Psi$ by zero from $\R^\dmn_+$ to $\R^\dmn$.
Thus,
\begin{equation*}\doublebar{\mathcal{A}_2(t\nabla^m \s^L_\nabla\arr h)}_{L^p(\R^n)}
=
\sup_{\arr\Psi} \frac{1}{\doublebar{\mathcal{A}_2(\arr \Psi)}_{L^{p'}(\R^n)}} \abs{\langle \Tr_m^+\Pi^{L^*}(\1_+\arr\Psi),\arr h\rangle_{\R^n}}
\end{equation*}
and so we have reduced matters to proving the estimate
\begin{equation}\label{A:eqn:trace:newton}
\doublebar{\Tr_m^+\Pi^{L^*}(\1_+\arr\Psi)}_{L^{p'}(\R^n)}\leq C_p{\doublebar{\mathcal{A}_2\arr \Psi}_{L^{p'}(\R^n)}}\end{equation}
for all $2<p'<2+\varepsilon$.

By Theorem~\ref{A:thm:S:square:variant} and the above duality results, the bound~\eqref{A:eqn:trace:newton} is valid for $p'=2$.
We will apply the following lemma.
\begin{lem}[{\cite[Lemma~3.2]{Iwa98}}] \label{A:lem:iwaniec} Suppose that $g$, $h\in L^q(\R^n)$ are nonnegative real-valued functions, $1<q<\infty$, and that for some $C_0>0$ and for all cubes $Q\subset\R^n$,
\begin{equation*}\biggl(\fint_Q g^q\biggr)^{1/q}\leq C_0\fint_{4Q}g+\biggl(\fint_{4Q} h^q\biggr)^{1/q}.\end{equation*}
Then there exist numbers $s>q$ and $C>0$ depending on $n$, $q$ and $C_0$ such that
\begin{equation*}\int_{\R^n} g^s\leq C\int_{\R^n} h^s.\end{equation*}
\end{lem}

Let $g=\abs{\Tr_m\Pi^{L^*}(\1_+\arr \Psi)}$ and let
\begin{equation*}h(z)=\mathcal{C}\arr\Psi(z)=\sup_{Q\owns z} \biggl( \frac{1}{\abs{Q}} \int_Q \int_0^{\ell(Q)} \abs{\arr \Psi(x,t)}^2\,\frac{dt\,dx}{t}\biggr)^{1/2}\end{equation*}
where the supremum is taken over cubes $Q\subset\R^n$ with $z\in Q$.
By \cite[Theorem 3, Section~6]{CoiMS85}, if $2<p'<\infty$ then $\doublebar{\mathcal{A}_2 H}_{L^{p'}(\R^n)}\approx \doublebar{\mathcal{C} H}_{L^{p'}(\R^n)}$. We claim that $g$ and $h$ satisfy the conditions of the lemma with $q=2$; there is then some $s>2$ such that
\begin{equation*}
\doublebar{\Tr_m^+\Pi^{L^*}\arr\Psi}_{L^s(\R^n)}
\leq
C \doublebar{\mathcal{C}\arr \Psi}_{L^s(\R^n)}
\leq
C{\doublebar{\mathcal{A}_2\arr \Psi}_{L^s(\R^n)}}
\end{equation*}
and so the bound \eqref{A:eqn:trace:newton} is valid for $p'=s$. (By interpolation it is valid for all $2<p'<s$ as well.)

We now prove the claim.

For notational convenience we will let $\arr\Psi=\1_+\arr\Psi$.
Choose some cube $Q\subset\R^n$. Let $R_j=2^{j}Q\times(0,2^{j}\ell(Q))$, and let $\arr\Psi_0=\arr\Psi\1_{R_2}$.

Then
\begin{equation*}\biggl(\fint_Q \abs{\Tr_m \Pi^{L^*} \arr \Psi}^2\biggr)^{1/2}
\leq
\biggl(\fint_Q \abs{\Tr_m \Pi^{L^*} \arr \Psi_0}^2\biggr)^{1/2}
+\biggl(\fint_Q \abs{\Tr_m \Pi^{L^*} (\arr \Psi-\arr\Psi_0)}^2\biggr)^{1/2}
.\end{equation*}
Because formula~\eqref{A:eqn:trace:newton} is valid for $p'=2$, we have that
\begin{align*}
\biggl(\int_{\R^n} \abs{\Tr_m \Pi^{L^*} \arr \Psi_0}^2\biggr)^{1/2}
&\leq
	C\doublebar{\mathcal{A}_2(\arr\Psi_0)}_{L^2(\R^n)}
\\&=
	 C\abs{Q}^{1/2}\biggl( \frac{1}{\abs{4Q}}\int_{4Q}\int_0^{4\ell(Q)} \abs{\arr\Psi(x,t)}^2\,\frac{dt\,dx}{t}\biggr)^{1/2}
.\end{align*}
By definition of~$\mathcal{C}$, the right hand side is at most $C\abs{Q}^{1/2}\mathcal{C}\arr\Psi(z)$ for all $z\in 4Q$,
and so
\begin{equation}
\label{A:eqn:Newton:local}
\biggl(\int_{\R^n} \abs{\Tr_m \Pi^{L^*} \arr \Psi_0}^2\biggr)^{1/2}
\leq
	C\abs{Q}^{1/2}\biggl(\fint_Q(\mathcal{C}\arr\Psi)^2\biggr)^{1/2}
.\end{equation}

We are left with the ${\Tr_m \Pi^{L^*} (\arr \Psi-\arr\Psi_0)}$ term.
Let $u=\Pi^{L^*} (\arr \Psi-\arr\Psi_0)$. Notice that $u$ is a solution to $L^*u=0$ in $4Q\times (-4\ell(Q),4\ell(Q))$ (and also in the lower half-space).
We will apply the following lemma.

\begin{lem}\label{A:lem:iterate} Let $L$ be as in Theorem~\ref{A:thm:S:p-}. Let $Q$ be a cube, let $u\in \dot W^2_m(3Q\times (-2\ell(Q),2\ell(Q)))$, and suppose that $Lu=0$ in $3Q\times (-2\ell(Q),2\ell(Q))$. Let $0\leq j\leq m$ and let $\arr c$ be a constant array. Then
\begin{align*}
\biggl(\fint_{2Q}\fint_{-\ell(Q)}^{\ell(Q)} \abs{\nabla^j u(x,t)}^2\,dt\,dx\biggr)^{1/2}
&\leq
C\fint_{3Q}\fint_{-2\ell(Q)}^{2\ell(Q)}\abs{\partial_t^j u(x,t)}\,dt\,dx
\\&\qquad+\frac{C}{\ell(Q)} \fint_{3Q} \abs{\nabla^{j-1} u(x,0)-\arr c}\,dx
.\end{align*}
In particular, by the Poincar\'e inequality and Lemma~\ref{A:lem:slices},
\begin{align*}
\biggl(\fint_{Q} \abs{\nabla^j u(x,0)}^2\,dx\biggr)^{1/2}
&\leq
C\fint_{3Q}\fint_{-2\ell(Q)}^{2\ell(Q)}\abs{\partial_t^j u(x,t)}\,dt\,dx
+\frac{C}{\ell(Q)} \fint_{3Q} \abs{\nabla^{j} u(x,0)}\,dx
.\end{align*}
\end{lem}

\begin{proof}
Let $0\leq k\leq j-1$. Let $\varepsilon>0$ be a small positive number and let $Q_k=(1+k\varepsilon)Q$. Observe that there is some polynomial $P$ of degree $j-1$ such that $\nabla^{j-1}P=\arr c$, and that $\tilde u=u-P$ is also a solution to $Lu=0$; thus, we need only prove the lemma in the case $\arr c=0$.

By the Caccioppoli inequality,
\begin{multline*}
\fint_{2Q_k} \fint_{-\ell(Q_k)}^{\ell(Q_k)} \abs{\nabla^{j-k}\partial_t^k u(x,t)}^2\,dt\,dx
\\\leq \frac{C(\varepsilon)}{\ell(Q)^2}\fint_{2Q_{k+1/2}} \fint_{-\ell(Q_{k+1/2})}^{\ell(Q_{k+1/2})} \abs{\nabla^{j-k-1}\partial_t^k u(x,t)}^2\,dt\,dx.\end{multline*}
By Theorem~\ref{A:thm:Meyers},
\begin{multline*}
\fint_{2Q_{k+1/2}} \fint_{-\ell(Q_{k+1/2})}^{\ell(Q_{k+1/2})} \abs{\nabla^{j-k-1}\partial_t^k u(x,t)}^2\,dt\,dx
\\\leq
C(\varepsilon)\biggl(\fint_{2Q_{k+1}} \fint_{-\ell(Q_{k+1})}^{\ell(Q_{k+1})} \abs{\nabla^{j-k-1} \partial_t^k u(x,t)}\,dt\,dx\biggr)^2
.\end{multline*}
If $x\in 2 Q_{k+1}$ and $-\ell(Q_{k+1})<t<\ell(Q_{k+1})$, then
\begin{align*}
\abs{\nabla^{j-k-1}\partial_\dmn^k u(x,t)}^2
&\leq
	\abs{
	\nabla^{j-k-1}\partial_\dmn^{k} u(x,t)-\nabla^{j-k-1}\partial_\dmn^{k} u(x,0)
	}
	\\&\qquad
	+ \abs{\nabla^{j-k-1}\partial_\dmn^{k} u(x,0)}
\\&\leq
	C\ell(Q)
	\fint_{-\ell(Q_{k+1})}^{\ell(Q_{k+1})}
	\abs{
	\nabla^{j-k-1}\partial_t^{k+1} u(x,s)
	}\,ds
	\\&\qquad
	+ \abs{\nabla^{j-k-1}\partial_\dmn^{k} u(x,0)}
.\end{align*}
Thus,
\begin{multline*}
\fint_{2Q_k} \fint_{-\ell(Q_k)}^{\ell(Q_k)} \abs{\nabla^{j-k}\partial_t^{k} u(x,t)}^2\,dt\,dx
\\\begin{aligned}
&
\leq
{C(\varepsilon)}
\biggl(\fint_{2Q_{k+1}}
	\fint_{-\ell(Q_{k+1})}^{\ell(Q_{k+1})}
	\abs{
	\nabla^{j-k-1} \partial_t^{k+1} u(x,t)} \,dt\,dx\biggr)^2
\\&\qquad+
\frac{C(\varepsilon)}{\ell(Q)^2}
\biggl(\fint_{2Q_{k+1}} \abs{\nabla^{j-k-1}\partial_\dmn^k u(x,0)}\,dx\biggr)^2
.\end{aligned}\end{multline*}
Iterating, we see that
\begin{align*}
\fint_{2Q_k} \fint_{-\ell(Q)}^{\ell(Q)} \abs{\nabla^{j} u(x,t)}^2\,dt\,dx
&
\leq
{C(\varepsilon)}
\biggl(\fint_{2Q_j}
	\fint_{-\ell(Q_j)}^{\ell(Q_j)}
	\abs{
	\partial_t^j u(x,t)
	}\,dt\,dx\biggr)^2
	\\&\qquad
	+
	\frac{C(\varepsilon)}{\ell(Q)^2}
	\biggl(\fint_{2Q_j} \abs{\nabla^{j-1} u(x,0)}\,dx\biggr)^2
.\end{align*}
Letting $\varepsilon=1/2j$ and so $2Q_j =3Q$ completes the proof.
\end{proof}

Recall that $u=\Pi^{L^*} (\arr \Psi-\arr\Psi_0)$. By Lemma~\ref{A:lem:iterate},
\begin{align*}
\biggl(\fint_Q \abs{\Tr_m \Pi^{L^*} (\arr \Psi-\arr\Psi_0)}^2\biggr)^{1/2}
&\leq
	C\fint_{3Q}\fint_{-2\ell(Q)}^{2\ell(Q)} \abs{\partial_\dmn^m \Pi^{L^*} (\arr \Psi-\arr\Psi_0)(x,t)}\,dt\,dx
	\\&\qquad +
	C\fint_{3Q} \abs{\Tr_m \Pi^{L^*} (\arr \Psi-\arr\Psi_0)}
.\end{align*}
It will be convenient to have an additional vertical derivative; thus,
{\multlinegap=0pt\begin{multline*}
\fint_{3Q}\fint_{-2\ell(Q)}^{2\ell(Q)} \abs{\partial_\dmn^m \Pi^{L^*} (\arr \Psi-\arr\Psi_0)(x,t)}\,dt\,dx
\\\begin{aligned}
&\leq
	\fint_{3Q} \abs{\partial_\dmn^m \Pi^{L^*} (\arr \Psi-\arr\Psi_0)(x,0)}\,dx
	\\&\qquad+
	\fint_{3Q}\fint_{-2\ell(Q)}^{2\ell(Q)} \abs[bigg]{\int_0^t\partial_s^{m+1} \Pi^{L^*} (\arr \Psi-\arr\Psi_0)(x,s)\,ds}\,dt\,dx
\\&\leq
	\fint_{3Q} \abs{\Tr_m \Pi^{L^*} (\arr \Psi-\arr\Psi_0)}
	+
	\fint_{3Q}\int_{-2\ell(Q)}^{2\ell(Q)} \abs{\partial_s^{m+1} \Pi^{L^*} (\arr \Psi-\arr\Psi_0)(x,s)}\,ds\,dx
.\end{aligned}\end{multline*}}
By the bound~\eqref{A:eqn:Newton:local},
\begin{equation*}\fint_{3Q} \abs{\Tr_m \Pi^{L^*} (\arr \Psi-\arr\Psi_0)}
\leq
\fint_{3Q} \abs{\Tr_m \Pi^{L^*} \arr \Psi}
+C\biggl(\fint_Q(\mathcal{C}\arr\Psi)^2\biggr)^{1/2}
.\end{equation*}
Thus,
\begin{align*}
\biggl(\fint_Q \abs{\Tr_m \Pi^{L^*} (\arr \Psi-\arr\Psi_0)}^2\biggr)^{1/2}
&\leq
	C\fint_{3Q}\int_{-2\ell(Q)}^{2\ell(Q)} \abs{\partial_t^{m+1} \Pi^{L^*} (\arr \Psi-\arr\Psi_0)(x,t)}\,dt\,dx
	\\&\qquad +
	C\fint_{3Q} \abs{\Tr_m \Pi^{L^*} \arr \Psi}
	+C\biggl(\fint_Q(\mathcal{C}\arr\Psi)^2\biggr)^{1/2}
.\end{align*}
We now contend with the vertical derivative. Recall that we assumed $2m>\dmn$. If $x\in 3Q$, $\abs{t}<2\ell(Q)$, then by formula~\eqref{A:eqn:fundamental:2},
\begin{align*}
\abs{\partial_t^{m+1} \Pi^{L^*} (\arr \Psi-\arr\Psi_0)(x,t)}
&=
	\abs[bigg]{\sum_{\abs\beta=m}\int_{\R^\dmn_+\setminus R_2}\partial_t^{m+1} \partial_{y,s}^\beta E^{L^*}(x,t,y,s)\,\Psi_\beta(y,s)\,dy\,ds }
.\end{align*}
By H\"older's inequality,
\begin{align*}
\abs{\partial_t^{m+1} \Pi^{L^*} (\arr \Psi-\arr\Psi_0)(x,t)}
&\leq
	\sum_{j=2}^\infty
	\biggl(\int_{R_{j+1}\setminus R_j} \abs{\partial_t^{m+1}\nabla_{y,s}^m E^L(x,t,y,s)}^2\,dy\,ds\biggr)^{1/2}
	\\&\qquad\times \biggl(\int_{R_{j+1}\setminus R_j} \abs{\arr\Psi(y,s)}^2\,dy\,ds \biggr)^{1/2}
\end{align*}
and by the bound \eqref{A:eqn:fundamental:far}, the Caccioppoli inequality and Theorem~\ref{A:thm:Meyers},
\begin{align*}
\abs{\partial_t^{m+1} \Pi^{L^*} (\arr \Psi-\arr\Psi_0)(x,t)}
&\leq
	C\sum_{j=2}^\infty
	\frac{1}{(2^j\ell(Q))^{n/2+3/2}}
	\biggl(\int_{R_{j+1}\setminus R_j} \abs{\arr\Psi(y,s)}^2\,dy\,ds \biggr)^{1/2}
.\end{align*}
By definition of~$\mathcal{C}$,
\begin{align*}
\frac{1}{(2^j\ell(Q))^{\pdmn/2}}
	\biggl(\int_{R_{j+1}\setminus R_j} \abs{\arr\Psi(y,s)}^2\,dy\,ds \biggr)^{1/2}
\leq
	2^{\pdmn/2} \mathcal{C}\arr\Psi(z)
\end{align*}
for all $z\in 2^{j+1}Q$.
Thus if $2m>\dmn$ then
\begin{align*}
\biggl(\fint_Q \abs{\Tr_m \Pi^{L^*} \arr \Psi(x)}^2\,dx\biggr)^{1/2}
&\leq
	C \fint_{3Q} \abs{\Tr_m \Pi^{L^*} \arr \Psi(x)}\,dx
	+
	C\biggl(\fint_{3Q}\mathcal{C}\arr\Psi^2\biggr)^{1/2}
\end{align*}
as desired.
\end{proof}

\subsection*{Acknowledgements}
We would like to thank the American Institute of Mathematics for hosting the SQuaRE workshop on ``Singular integral operators and solvability of boundary problems for elliptic equations with rough coefficients,'' and the Mathematical Sciences Research Institute for hosting a Program on Harmonic Analysis, at which many of the results and techniques of this paper were discussed.

\bibliographystyle{amsplain}
\bibliography{bibli}

\providecommand{\bysame}{\leavevmode\hbox to3em{\hrulefill}\thinspace}
\providecommand{\MR}{\relax\ifhmode\unskip\space\fi MR }
\providecommand{\MRhref}[2]{%
  \href{http://www.ams.org/mathscinet-getitem?mr=#1}{#2}
}
\providecommand{\href}[2]{#2}
\begin{thebibliography}{10}

\bibitem{Agm57}
Shmuel Agmon, \emph{Multiple layer potentials and the {D}irichlet problem for
  higher order elliptic equations in the plane. {I}}, Comm. Pure Appl. Math
  \textbf{10} (1957), 179--239. \MR{0106323 (21 \#5057)}

\bibitem{Agr09}
M.~S. Agranovich, \emph{Potential-type operators and conjugation problems for
  second-order strongly elliptic systems in domains with a {L}ipschitz
  boundary}, Funktsional. Anal. i Prilozhen. \textbf{43} (2009), no.~3, 3--25.
  \MR{2583636 (2011b:35362)}

\bibitem{AlfAAHK11}
M.~Angeles Alfonseca, Pascal Auscher, Andreas Axelsson, Steve Hofmann, and
  Seick Kim, \emph{Analyticity of layer potentials and {$L^2$} solvability of
  boundary value problems for divergence form elliptic equations with complex
  {$L^\infty$} coefficients}, Adv. Math. \textbf{226} (2011), no.~5,
  4533--4606. \MR{2770458}

\bibitem{AusQ00}
P.~Auscher and M.~Qafsaoui, \emph{Equivalence between regularity theorems and
  heat kernel estimates for higher order elliptic operators and systems under
  divergence form}, J. Funct. Anal. \textbf{177} (2000), no.~2, 310--364.
  \MR{1795955 (2001j:35057)}

\bibitem{AusAH08}
Pascal Auscher, Andreas Axelsson, and Steve Hofmann, \emph{Functional calculus
  of {D}irac operators and complex perturbations of {N}eumann and {D}irichlet
  problems}, J. Funct. Anal. \textbf{255} (2008), no.~2, 374--448. \MR{2419965
  (2009h:35079)}

\bibitem{AusAM10A}
Pascal Auscher, Andreas Axelsson, and Alan McIntosh, \emph{Solvability of
  elliptic systems with square integrable boundary data}, Ark. Mat. \textbf{48}
  (2010), no.~2, 253--287. \MR{2672609 (2011h:35070)}

\bibitem{AusHMT01}
Pascal Auscher, Steve Hofmann, Alan McIntosh, and Philippe Tchamitchian,
  \emph{The {K}ato square root problem for higher order elliptic operators and
  systems on {$\mathbb{R}^n$}}, J. Evol. Equ. \textbf{1} (2001), no.~4,
  361--385, Dedicated to the memory of Tosio Kato. \MR{1877264 (2003a:35046)}

\bibitem{AusM14}
Pascal Auscher and Mihalis Mourgoglou, \emph{Boundary layers, {R}ellich
  estimates and extrapolation of solvability for elliptic systems}, Proc. Lond.
  Math. Soc. (3) \textbf{109} (2014), no.~2, 446--482. \MR{3254931}

\bibitem{AusS16}
Pascal Auscher and Sebastian Stahlhut, \emph{Functional calculus for first
  order systems of {D}irac type and boundary value problems}, M\'em. Soc. Math.
  Fr. (N.S.) (2016), no.~144, vii+164. \MR{3495480}

\bibitem{AusT98}
Pascal Auscher and Philippe Tchamitchian, \emph{Square root problem for
  divergence operators and related topics}, Ast\'erisque (1998), no.~249,
  viii+172. \MR{1651262 (2000c:47092)}

\bibitem{Bar13}
Ariel Barton, \emph{Elliptic partial differential equations with almost-real
  coefficients}, Mem. Amer. Math. Soc. \textbf{223} (2013), no.~1051, vi+108.
  \MR{3086390}

\bibitem{Bar16}
\bysame, \emph{Gradient estimates and the fundamental solution for higher-order
  elliptic systems with rough coefficients}, Manuscripta Math. \textbf{151}
  (2016), no.~3-4, 375--418. \MR{3556825}

\bibitem{BarHM17pB}
Ariel Barton, Steve Hofmann, and Svitlana Mayboroda, \emph{{D}irichlet and
  {N}eumann boundary values of solutions to higher order elliptic equations},
  Ann. Inst. Fourier (Grenoble), to appear (a preprint may be found at
  \href{http://arxiv.org/abs/1703.06963}{\texttt{arXiv:1703.06963 [math.AP]}}).

\bibitem{BarHM17}
\bysame, \emph{Square function estimates on layer potentials for higher-order
  elliptic equations}, Math. Nachr. \textbf{290} (2017), no.~16, 2459--2511.
  \MR{3722489}

\bibitem{BarHM18}
\bysame, \emph{The {N}eumann problem for higher order elliptic equations with
  symmetric coefficients}, Math. Ann. \textbf{371} (2018), no.~1-2, 297--336.
  \MR{3788849}

\bibitem{BarM13}
Ariel Barton and Svitlana Mayboroda, \emph{The {D}irichlet problem for higher
  order equations in composition form}, J. Funct. Anal. \textbf{265} (2013),
  no.~1, 49--107. \MR{3049881}

\bibitem{BarM16B}
\bysame, \emph{Higher-order elliptic equations in non-smooth domains: a partial
  survey}, Harmonic analysis, partial differential equations, complex analysis,
  {B}anach spaces, and operator theory. {V}ol. 1, Assoc. Women Math. Ser.,
  vol.~4, Springer, [Cham], 2016, pp.~55--121. \MR{3627715}

\bibitem{BarM16A}
\bysame, \emph{Layer potentials and boundary-value problems for second order
  elliptic operators with data in {B}esov spaces}, Mem. Amer. Math. Soc.
  \textbf{243} (2016), no.~1149, v+110. \MR{3517153}

\bibitem{CafFK81}
Luis~A. Caffarelli, Eugene~B. Fabes, and Carlos~E. Kenig, \emph{Completely
  singular elliptic-harmonic measures}, Indiana Univ. Math. J. \textbf{30}
  (1981), no.~6, 917--924. \MR{632860 (83a:35033)}

\bibitem{Cam80}
S.~Campanato, \emph{Sistemi ellittici in forma divergenza. {R}egolarit\`a
  all'interno}, Qua\-der\-ni. [Publications], Scuola Normale Superiore Pisa,
  Pisa, 1980. \MR{668196 (83i:35067)}

\bibitem{CohG83}
Jonathan Cohen and John Gosselin, \emph{The {D}irichlet problem for the
  biharmonic equation in a {$C^{1}$} domain in the plane}, Indiana Univ. Math.
  J. \textbf{32} (1983), no.~5, 635--685. \MR{711860 (85b:31004)}

\bibitem{CoiMM82}
R.~R. Coifman, A.~McIntosh, and Y.~Meyer, \emph{L'int\'egrale de {C}auchy
  d\'efinit un op\'erateur born\'e sur {$L^{2}$} pour les courbes
  lipschitziennes}, Ann. of Math. (2) \textbf{116} (1982), no.~2, 361--387.
  \MR{672839 (84m:42027)}

\bibitem{CoiMS85}
R.~R. Coifman, Y.~Meyer, and E.~M. Stein, \emph{Some new function spaces and
  their applications to harmonic analysis}, J. Funct. Anal. \textbf{62} (1985),
  no.~2, 304--335. \MR{791851 (86i:46029)}

\bibitem{Coi74}
Ronald~R. Coifman, \emph{A real variable characterization of {$H^{p}$}}, Studia
  Math. \textbf{51} (1974), 269--274. \MR{0358318 (50 \#10784)}

\bibitem{DahKPV97}
B.~E.~J. Dahlberg, C.~E. Kenig, J.~Pipher, and G.~C. Verchota, \emph{Area
  integral estimates for higher order elliptic equations and systems}, Ann.
  Inst. Fourier (Grenoble) \textbf{47} (1997), no.~5, 1425--1461. \MR{1600375
  (98m:35045)}

\bibitem{DahK87}
Bj{\"o}rn E.~J. Dahlberg and Carlos~E. Kenig, \emph{Hardy spaces and the
  {N}eumann problem in {$L^p$} for {L}aplace's equation in {L}ipschitz
  domains}, Ann. of Math. (2) \textbf{125} (1987), no.~3, 437--465. \MR{890159
  (88d:35044)}

\bibitem{Dav95}
E.~B. Davies, \emph{Uniformly elliptic operators with measurable coefficients},
  J. Funct. Anal. \textbf{132} (1995), no.~1, 141--169. \MR{1346221
  (97a:47074)}

\bibitem{Eva98}
Lawrence~C. Evans, \emph{Partial differential equations}, Graduate Studies in
  Mathematics, vol.~19, American Mathematical Society, Providence, RI, 1998.
  \MR{1625845 (99e:35001)}

\bibitem{FabJR78}
E.~B. Fabes, M.~Jodeit, Jr., and N.~M. Rivi{\`e}re, \emph{Potential techniques
  for boundary value problems on {$C^{1}$}-domains}, Acta Math. \textbf{141}
  (1978), no.~3-4, 165--186. \MR{501367 (80b:31006)}

\bibitem{FabMM98}
Eugene Fabes, Osvaldo Mendez, and Marius Mitrea, \emph{Boundary layers on
  {S}obolev-{B}esov spaces and {P}oisson's equation for the {L}aplacian in
  {L}ipschitz domains}, J. Funct. Anal. \textbf{159} (1998), no.~2, 323--368.
  \MR{1658089 (99j:35036)}

\bibitem{FabJK84}
Eugene~B. Fabes, David~S. Jerison, and Carlos~E. Kenig, \emph{Necessary and
  sufficient conditions for absolute continuity of elliptic-harmonic measure},
  Ann. of Math. (2) \textbf{119} (1984), no.~1, 121--141. \MR{736563
  (85h:35069)}

\bibitem{FefS72}
C.~Fefferman and E.~M. Stein, \emph{{$H^{p}$} spaces of several variables},
  Acta Math. \textbf{129} (1972), no.~3-4, 137--193. \MR{0447953 (56 \#6263)}

\bibitem{Fre08}
Jens Frehse, \emph{An irregular complex valued solution to a scalar uniformly
  elliptic equation}, Calc. Var. Partial Differential Equations \textbf{33}
  (2008), no.~3, 263--266. \MR{2429531 (2009h:35084)}

\bibitem{GraH16}
Ana Grau de~la Herr\'an and Steve Hofmann, \emph{A local {$Tb$} theorem with
  vector-valued testing functions}, Some topics in harmonic analysis and
  applications, Adv. Lect. Math. (ALM), vol.~34, Int. Press, Somerville, MA,
  2016, pp.~203--229. \MR{3525561}

\bibitem{HofKMP15B}
Steve Hofmann, Carlos Kenig, Svitlana Mayboroda, and Jill Pipher, \emph{The
  regularity problem for second order elliptic operators with complex-valued
  bounded measurable coefficients}, Math. Ann. \textbf{361} (2015), no.~3-4,
  863--907. \MR{3319551}

\bibitem{HofKMP15A}
\bysame, \emph{Square function\slash non-tangential maximal function estimates
  and the {D}irichlet problem for non-symmetric elliptic operators}, J. Amer.
  Math. Soc. \textbf{28} (2015), no.~2, 483--529. \MR{3300700}

\bibitem{HofMayMou15}
Steve Hofmann, Svitlana Mayboroda, and Mihalis Mourgoglou, \emph{Layer
  potentials and boundary value problems for elliptic equations with complex
  {$L^\infty$} coefficients satisfying the small {C}arleson measure norm
  condition}, Adv. Math. \textbf{270} (2015), 480--564. \MR{3286542}

\bibitem{HofMitMor15}
Steve Hofmann, Marius Mitrea, and Andrew~J. Morris, \emph{The method of layer
  potentials in {$L^p$} and endpoint spaces for elliptic operators with
  {$L^\infty$} coefficients}, Proc. Lond. Math. Soc. (3) \textbf{111} (2015),
  no.~3, 681--716. \MR{3396088}

\bibitem{Iwa98}
Tadeusz Iwaniec, \emph{The {G}ehring lemma}, Quasiconformal mappings and
  analysis ({A}nn {A}rbor, {MI}, 1995), Springer, New York, 1998, pp.~181--204.
  \MR{1488451}

\bibitem{Jaw77}
Bj{\"o}rn Jawerth, \emph{Some observations on {B}esov and {L}izorkin-{T}riebel
  spaces}, Math. Scand. \textbf{40} (1977), no.~1, 94--104. \MR{0454618 (56
  \#12867)}

\bibitem{JerK81A}
David~S. Jerison and Carlos~E. Kenig, \emph{The {D}irichlet problem in
  nonsmooth domains}, Ann. of Math. (2) \textbf{113} (1981), no.~2, 367--382.
  \MR{607897 (84j:35076)}

\bibitem{KenKPT00}
C.~Kenig, H.~Koch, J.~Pipher, and T.~Toro, \emph{A new approach to absolute
  continuity of elliptic measure, with applications to non-symmetric
  equations}, Adv. Math. \textbf{153} (2000), no.~2, 231--298. \MR{1770930
  (2002f:35071)}

\bibitem{KenP93}
Carlos~E. Kenig and Jill Pipher, \emph{The {N}eumann problem for elliptic
  equations with nonsmooth coefficients}, Invent. Math. \textbf{113} (1993),
  no.~3, 447--509. \MR{1231834 (95b:35046)}

\bibitem{KenR09}
Carlos~E. Kenig and David~J. Rule, \emph{The regularity and {N}eumann problem
  for non-symmetric elliptic operators}, Trans. Amer. Math. Soc. \textbf{361}
  (2009), no.~1, 125--160. \MR{2439401 (2009k:35050)}

\bibitem{Lat78}
Robert~H. Latter, \emph{A characterization of {$H^{p}({\bf R}^{n})$} in terms
  of atoms}, Studia Math. \textbf{62} (1978), no.~1, 93--101. \MR{0482111 (58
  \#2198)}

\bibitem{Liz60}
P.~I. Lizorkin, \emph{Boundary properties of functions from ``weight''
  classes}, Soviet Math. Dokl. \textbf{1} (1960), 589--593. \MR{0123103 (23
  \#A434)}

\bibitem{MedSV08}
Stefano Meda, Peter Sj\"ogren, and Maria Vallarino, \emph{On the
  {$H^1$}-{$L^1$} boundedness of operators}, Proc. Amer. Math. Soc.
  \textbf{136} (2008), no.~8, 2921--2931. \MR{2399059}

\bibitem{MitM13B}
Irina Mitrea and Marius Mitrea, \emph{Boundary value problems and integral
  operators for the bi-{L}aplacian in non-smooth domains}, Atti Accad. Naz.
  Lincei Rend. Lincei Mat. Appl. \textbf{24} (2013), no.~3, 329--383.
  \MR{3097019}

\bibitem{MitM13A}
\bysame, \emph{Multi-layer potentials and boundary problems for higher-order
  elliptic systems in {L}ipschitz domains}, Lecture Notes in Mathematics, vol.
  2063, Springer, Heidelberg, 2013. \MR{3013645}

\bibitem{PipV91}
Jill Pipher and Gregory Verchota, \emph{Area integral estimates for the
  biharmonic operator in {L}ipschitz domains}, Trans. Amer. Math. Soc.
  \textbf{327} (1991), no.~2, 903--917. \MR{1024776 (92a:35052)}

\bibitem{PipV92}
\bysame, \emph{The {D}irichlet problem in {$L^p$} for the biharmonic equation
  on {L}ipschitz domains}, Amer. J. Math. \textbf{114} (1992), no.~5, 923--972.
  \MR{1183527 (94g:35069)}

\bibitem{PipV95B}
Jill Pipher and Gregory~C. Verchota, \emph{Dilation invariant estimates and the
  boundary {G}\aa rding inequality for higher order elliptic operators}, Ann.
  of Math. (2) \textbf{142} (1995), no.~1, 1--38. \MR{1338674 (96g:35052)}

\bibitem{Ros13}
Andreas Ros{\'e}n, \emph{Layer potentials beyond singular integral operators},
  Publ. Mat. \textbf{57} (2013), no.~2, 429--454. \MR{3114777}

\bibitem{Rul07}
David~J. Rule, \emph{Non-symmetric elliptic operators on bounded {L}ipschitz
  domains in the plane}, Electron. J. Differential Equations (2007), No. 144,
  1--8. \MR{2366037 (2008m:35070)}

\bibitem{Ste93}
Elias~M. Stein, \emph{Harmonic analysis: real-variable methods, orthogonality,
  and oscillatory integrals}, Princeton Mathematical Series, vol.~43, Princeton
  University Press, Princeton, NJ, 1993, With the assistance of Timothy S.
  Murphy, Monographs in Harmonic Analysis, III. \MR{1232192 (95c:42002)}

\bibitem{Ver84}
Gregory Verchota, \emph{Layer potentials and regularity for the {D}irichlet
  problem for {L}aplace's equation in {L}ipschitz domains}, J. Funct. Anal.
  \textbf{59} (1984), no.~3, 572--611. \MR{769382 (86e:35038)}

\bibitem{Ver05}
Gregory~C. Verchota, \emph{The biharmonic {N}eumann problem in {L}ipschitz
  domains}, Acta Math. \textbf{194} (2005), no.~2, 217--279. \MR{2231342
  (2007d:35058)}

\bibitem{Ver10}
\bysame, \emph{Boundary coerciveness and the {N}eumann problem for 4th order
  linear partial differential operators}, Around the research of {V}ladimir
  {M}az'ya. {II}, Int. Math. Ser. (N. Y.), vol.~12, Springer, New York, 2010,
  pp.~365--378. \MR{2676183 (2011h:35065)}

\bibitem{Zan00}
Daniel~Z. Zanger, \emph{The inhomogeneous {N}eumann problem in {L}ipschitz
  domains}, Comm. Partial Differential Equations \textbf{25} (2000), no.~9-10,
  1771--1808. \MR{1778780 (2001g:35056)}

\end{thebibliography}
\end{document}